\documentclass{amsart}

\usepackage{mathabx}
\usepackage{graphicx}
\usepackage{amsfonts}
\usepackage{amsmath}
\usepackage{amscd}
\usepackage{amssymb}
\usepackage{verbatim}
\usepackage{hyperref}
\usepackage{marginnote}
\usepackage{tikz}

\newenvironment{proof*}{\noindent\emph{Proof}}{$\square$\smallskip}

\newtheorem{theorem}{Theorem}[section]

\newtheorem{Definition}[theorem]{Definition}
\newtheorem{lemma}[theorem]{Lemma}
\newtheorem{Example}[theorem]{Example}

\newtheorem{corollary}[theorem]{Corollary}
\newtheorem{Remark}[theorem]{Remark}
\newtheorem{proposition}[theorem]{Proposition}

\newtheorem{Exercise}[theorem]{Exercise}
\newtheorem{Exercises}[theorem]{Exercises}
\newtheorem{Notation}[theorem]{Notation}
\newtheorem{Convention}[theorem]{Convention}
\newtheorem{standing assumption}[theorem]{Standing Assumption}

\newenvironment{definition}{\begin{Definition}\normalfont}{\end{Definition}}
\newenvironment{example}{\begin{Example}\normalfont}{\end{Example}}
\newenvironment{remark}{\begin{Remark}\normalfont}{\end{Remark}}

\title[Finiteness properties of piecewise projective homeomorphism groups]{Finiteness properties of some groups of piecewise projective homeomorphisms}  
\author[D.~S.~Farley]{Daniel S. Farley}
\address{Department of Mathematics\\ Miami University\\ Oxford, OH 45056 U.S.A.}
\email{farleyds@miamioh.edu}

\date{\today}

\begin{document}

\begin{abstract} 
The Lodha-Moore group $G$ first arose as a finitely presented counterexample to von Neumann's conjecture. The group $G$ acts on the unit interval via piecewise projective homemorphisms. 
 A result of Lodha shows that $G$ in fact has type $F_{\infty}$. 

Here we will describe $G$ as a group that is ``locally determined" by an inverse semigroup $S_{2}$, in the sense of the author's joint work with Hughes. The semigroup $S_{2}$ is generated by three linear fractional transformations $A$, $B$, and $C_{2}$, where $A$ and $B$ are elliptical transformations of the hyperbolic plane and $C_{2}$ is a hyperbolic translation. Following a general procedure delineated by Farley and Hughes, we offer a new proof that $G$ has type $F_{\infty}$. Our proof simultaneously shows that various groups acting on the line, the circle, and the Cantor set have type $F_{\infty}$. We also prove analogous results 
for the groups that are locally determined by an inverse semigroup $S_{3}$, which shares the generators $A$ and $B$ with $S_{2}$, but replaces $C_{2}$ with a different hyperbolic translation $C_{3}$.
\end{abstract}



\maketitle

\setcounter{tocdepth}{2} \tableofcontents

%%%%%%%%%%%%%%%%%%%%5
\section{Introduction} \label{section:introduction}
%%%%%%%%%%%%%%%%%%

Monod \cite{Monod} produced a large family of counterexamples to von Neumann's conjecture; i.e., non-amenable groups with no free subgroups. His examples were groups of piecewise projective homeomorphisms of the line and were not finitely generated. Lodha and Moore \cite{LM} considered a subgroup $G$ of one of Monod's groups. Their group $G$, here called the \emph{Lodha-Moore group}, could be generated by three elements, and was shown in \cite{LM} to be finitely presented; indeed, $G$ admits a presentation with three generators and nine relators. In later work, Lodha \cite{Lodha} showed that $G$ has type $F_{\infty}$. The Lodha-Moore group is also nonamenable, making it an especially economical finitely presented counterexample to von Neumann's conjecture, and the first $F_{\infty}$ counterexample to von Neumann's conjecture. (The paper \cite{OS} provided the first finitely presented non-amenable groups with no free subgroups, but their groups had many more generators and relations. The higher finiteness properties of the groups remain unknown to the best of this author's knowledge.) 

In \cite{FH2}, Hughes and the author described a general approach to establishing finiteness properties for the class of groups that are locally determined by inverse semigroups. In the setting of \cite{FH2}, an inverse semigroup $S$ is alway realized as a collection of partial bijections of some set $X$, where a \emph{partial bijection} of $X$ is a bijection between two subsets of $X$. We define a group $\Gamma_{S}$, the \emph{group locally determined by $S$}, to be the collection of all bijections of $X$ that can be represented as a finite union partial bijections from $S$. (I.e., $\gamma \in \Gamma_{S}$ if, for some $n \in \mathbb{N}$, there are elements $s_{1}$, $s_{2}$, $\ldots$, $s_{n} \in S$ such that the domains of the $s_{i}$, denoted $D_{i}$, form a partition of $X$, the images $s_{i}(D_{i})$ are also a partition of $X$, and $\gamma_{\mid D_{i}} = s_{i}$, for $i=1, \ldots, n$.)

The goal in \cite{FH2} is to build classifying spaces for the groups $\Gamma_{S}$, and then to establish finiteness properties for these groups. The classifying spaces depend upon a sequence of choices. We first assume that the inverse semigroup $S$ and (therefore) the group $\Gamma_{S}$ have been chosen. The choice of $S$ determines a collection of \emph{domains} $\mathcal{D}_{S}$, which are simply the domains of the elements $s \in S$. The set of non-empty domains is denoted $\mathcal{D}^{+}_{S}$. The second choice is that of an \emph{$S$-structure}, which is a function $\mathbb{S}: \mathcal{D}^{+}_{S} \times \mathcal{D}^{+}_{S} \rightarrow \mathcal{P}(S)$ assigning a (possibly empty) collection of transformations from $S$ to each pair $(D_{1},D_{2})$ of non-empty domains. The sets $\mathbb{S}(D_{1},D_{2})$ are required to satisfy various ``groupoid-like'' properties. The specific properties that we need are summarized in Proposition \ref{proposition:closure}.  The $S$-structure $\mathbb{S}$ determines a set $\mathcal{V}_{\mathbb{S}}$ with a partial order $\leq$ called ``expansion" (Definition \ref{definition:expansion}).  The partially ordered set $(\mathcal{V}_{\mathbb{S}}, \leq)$ is a directed $\Gamma_{S}$-set. The construction of $\mathcal{V}_{\mathbb{S}}$ is very much like the one introduced by Brown in \cite{Brown}, where he proved the $F_{\infty}$ property for a wide variety of generalized Thompson groups.   
(We note that, according to \cite{FH2}, an $S$-structure is a pair $(\mathbb{S}, \mathbb{P})$, where $\mathbb{P}$ is a \emph{pattern function}. The function $\mathbb{P}$ is a device that limits the types of subdivisions that are allowable under the expansion partial order. We will make no use of such a function in this paper.)

The simplicial realization $\Delta_{\mathbb{S}}$ of $\mathcal{V}_{\mathbb{S}}$ is a contractible simplicial complex upon which $\Gamma_{S}$ acts simplicially. 
The simplicial complex $\Delta_{\mathbb{S}}$ often has  undesirable properties, however. For instance, it almost always fails to be locally finite. It can be helpful to replace $\Delta_{\mathbb{S}}$ with something smaller. A (third) choice of an \emph{expansion scheme} $\mathcal{E}$ (Definition \ref{definition:expansionscheme}) determines a subcomplex $\Delta^{\mathcal{E}}_{\mathbb{S}} \subset \Delta_{\mathbb{S}}$. The complex $\Delta^{\mathcal{E}}_{\mathbb{S}}$ can be anything from a discrete set of points to the complex $\Delta_{\mathbb{S}}$ itself, depending upon the size of the expansion scheme. Given an appropriate choice of $\mathcal{E}$, the main results of \cite{FH2} show how to deduce the $F_{\infty}$ property for $\Gamma_{S}$, by applying Brown's finiteness criterion to the complex $\Delta^{\mathcal{E}}_{\mathbb{S}}$. 

The above procedure (determined by the choices of semigroup $S$, $S$-structure $\mathbb{S}$, and expansion scheme $\mathcal{E}$) is applied in \cite{FH2} to various generalized Thompson groups, including the Brin-Thompson groups $nV$ \cite{BrinHighD} (and some generalizations), R\"{o}ver's group \cite{Rover}, and the groups that are determined by finite similarity structures \cite{FH1}, recovering proofs of the $F_{\infty}$ property in each case. The original proofs that $nV$ and R\"{o}ver's group have type $F_{\infty}$ appeared in \cite{nV} and \cite{BelkMatucci}, respectively.

The basic theory of \cite{FH2} was produced in the hope that it would unify the existing proofs of finiteness properties for generalized Thompson groups, or at least those groups with ``piecewise" definitions. 
The Lodha-Moore group $G$, as one such group, offers a useful test case. We note from the outset, however, that $G$ is analogous to Thompson's group $F$, in the sense that it is a group of homeomorphisms of the unit interval, while the main argument of \cite{FH2} concerns itself primarily with groups that are analogous to $V$. (For a description of $F$, $T$, and $V$, see the introductory notes \cite{CFP}.) This will affect our arguments in only minor ways.

Here we will give a new proof that $G$ has type $F_{\infty}$, which will largely follow the three-step method indicated above. 

\begin{enumerate}
\item[1)] First, we consider the inverse semigroup $S_{2}$, generated by the following linear fractional transformations:
\[ A(x) = \frac{x}{x+1}; \quad B(x) = \frac{1}{2-x}; \quad C_{2}(x) = \frac{2x}{x+1}. \]
Actually, we restrict the domains of each of these to the interval $I=[0,1)$, which makes $A$ a transformation from $I$ to $[0,1/2)$, $B$ a transformation from $I$ to $[1/2, 1)$, and $C_{2}$ a self-homeomorphism of $I$. Note also that $S_{2}$ is closed under inverses, and thus, denoting these inverses by lower-case letters, $a$, $b$, and $c$ are transformations from $[0,1/2)$, $[1/2,1)$, and $I$ (respectively) to the interval $I$. We define $F(S_{2})$ to be the group of \emph{continuous} bijections of $[0,1/2)$ that are locally determined by $S_{2}$. (The ``$F$" in $F(S_{2})$ is used to indicate the connection of $F(S_{2})$ with Thompson's group $F$ \cite{CFP}.) 

The set of domains produced by $S_{2}$ is not very tractable for our purposes. We will therefore restrict the domains under consideration to what we call a set of ``generating domains" $\mathcal{D}^{+}_{gen}$ (Definition \ref{definition:generatingsetofdomains}), which, for us, are simply the forward iterates of $I$ under the transformations $A$ and $B$. The domains in $\mathcal{D}^{+}_{gen}$ are the intervals between consecutive Farey numbers. (The decision to work with a proper subset of $\mathcal{D}^{+}_{S}$ represents the most important departure from \cite{FH2}.) 

\item[2)] We then define 
$\mathbb{S}$ to be the maximal $S$-structure; i.e., $\mathbb{S}(D_{1},D_{2})$ is the set of all transformations from $S_{2}$ having $D_{1}$ as the domain and $D_{2}$ as the range, where $D_{1}$ and $D_{2}$ are arbitrary members of $\mathcal{D}^{+}_{gen}$. This leads to the usual directed set construction, as described above. The expansions (Definition \ref{definition:expansion}) from the pair $[id_{I}, I]$ (Definition \ref{definition:B}) can be usefully described by numbered binary trees, exactly as was done in \cite{LM}. 

\item[3)] We then define the expansion scheme $\mathcal{E}_{2}$ as in Example \ref{example:THEexpansionscheme}. The expansion scheme $\mathcal{E}_{2}$ essentially assigns a certain simplicial complex $\mathcal{E}_{2}([f,D])$ to each pair $[f,D] \in \mathcal{B}$. The simplicial complex in question is isomorphic to a simplicial cone on a cellulated line -- see Figure \ref{figure:E}. The burden of the rest of the argument is to show that $\mathcal{E}_{2}$ is an ``$n$-connected expansion scheme" (Definition \ref{definition:nconnectedE}). The proof occupies Sections \ref{section:FP} and \ref{section:IVT}, and represents the technical heart of the paper. The material from Sections \ref{section:FP} and \ref{section:IVT} is heavily indebted to the argument from \cite{LM}, but generalizes their work in what we consider to be interesting ways. For instance, our argument also shows that the monoid generated by the linear fractional transformations $\{ A, B, C_{2}, c_{2} \}$ (even without the above restrictions on their domains) admits a finite complete rewrite system.
\end{enumerate}

With 1)-3) complete, the proof of $F_{\infty}$ follows by a standard argument (a variant of the main argument of \cite{FH2}). This standard argument is summarized in Section \ref{section:F-infty}. The basic set-up can be used to prove the $F_{\infty}$ property simultaneously for several groups, which we now describe. 

The author had originally wanted to offer a proof of $F_{\infty}$ for an infinite family of generalizations of the Lodha-Moore group $G$. One approach to producing an infinite family of similar groups is to replace the transformation $C_{2}$ with a transformation $C_{n}:I \rightarrow I$, defined as follows:
\[  C_{n}(x) = \frac{nx}{(n-1)x + 1}. \]
We can then define $S_{n}$ to be the inverse semigroup generated by $\{ A, B, C_{n} \}$ (and their inverses), and let $V(S_{n})$ denote the group that is locally determined by $S_{n}$.  We find that the proof of the $F_{\infty}$ property (sketched above) applies equally to both $V(S_{2})$ and $V(S_{3})$, but falls apart entirely  for $V(S_{n})$, when $n \geq 4$. The proof that $V(S_{3})$ has type $F_{\infty}$ is developed throughout the paper, alongside the proof for $V(S_{2})$, while the apparent absence of a similar proof for $\Gamma_{n}$ ($n \geq 4$) 
is the topic of Section \ref{section:algorithm}. We will also define inverse semigroups $S'_{2}$ and $S'_{3}$, which are partial transformations of the ray $[0,\infty)$. In all, we will establish the $F_{\infty}$ property for ten groups:
$F(S_{i})$, $F(S'_{i})$, $T(S_{i})$, $V(S_{i})$, and $V(S'_{i})$, for $i=2,3$. (Here an ``$F$" indicates a group of homeomorphisms of the interval $[0,1)$ or ray $[0,\infty)$, a ``$T$" indicates a group of homeomorphisms of the circle, and ``$V$" indicates a collection of right-continuous bijections of $[0,1)$ or $[0,\infty)$.)

Let us briefly describe the structure of the paper. In Section \ref{section:family}, we define the inverse semigroups $S_{i}$ and $S'_{i}$ ($i=2,3$) and the groups that are locally defined by these semigroups. In Section \ref{section:generatingdomains}, we define the ``generating domains" that we need. This section also includes a proof of the ``eventual invariance" property, which will be used to construct directed sets later. In Section \ref{section:directedset}, we build directed sets with $\Gamma$ actions (for $\Gamma$ as above) and compute vertex stabilizers. We show, in particular, that the vertex stabilizers in all of our complexes are virtually free abelian of finite rank. In Section \ref{section:algorithm}, we describe an algorithm that analyzes various potential generalizations of the Lodha-Moore group. The conclusion of the section is that such generalizations are surprisingly very thin on the ground, at least among groups of piecewise projective homeomorphisms. In Section \ref{section:expansion}, we review expansion schemes, introduce expansion schemes $\mathcal{E}_{i}$ and $\mathcal{E}'_{i}$ ($i=2,3$), and also introduce subdivision trees, which are used to describe expansions.This section also describes an equivalence relation on subdivision trees. In Section \ref{section:FP}, we compute finite complete semigroup presentations for the inverse semigroups $S_{2}$ and $S_{3}$. These presentations are vital in understanding the equivalence relation on subdivision trees. In Section \ref{section:IVT}, we prove an ``intermediate value theorem", which is what we need in order to show that the expansion schemes $\mathcal{E}_{i}$ and $\mathcal{E}'_{i}$ define contractible complexes. The proof of the latter is in Section \ref{section:F-infty}, which also assembles all of the other ingredients of the proof that the groups have type $F_{\infty}$.

%%%%%%%%%%%%%%%%%%%%
\section{A family of inverse semigroups} \label{section:family}
%%%%%%%%%%%%%%%%%%%%

We consider the usual action of $PSL_{2}(\mathbb{R})$ on the upper half-space model of the hyperbolic plane $\mathbb{H}^{2} \subseteq \mathbb{C}$. A $2 \times 2$ matrix 
\[ M = \left( \begin{smallmatrix} a & b \\ c & d \end{smallmatrix} \right) \]
acts as a linear fractional transformation $f_{M}: \mathbb{H}^{2} \rightarrow \mathbb{H}^{2}$, where
\[ f_{M}(z) = \frac{az + b}{cz + d} \quad (z \in \mathbb{C}; \mathrm{Im} z > 0). \]
It is well-known \cite{Rat} that the assignment $M \mapsto f_{M}$ induces an isomorphism between $PSL_{2}(\mathbb{R})$ and 
the group of all orientation-preserving isometries of $\mathbb{H}^{2}$, denoted $\mathrm{Isom}^{+}(\mathbb{H}^{2})$. In what follows, we will make no distinction between $M$ and $f_{M}$, referring to either one by the matrix $M$. 

In practice, we will be concerned primarily with the action of $PSL_{2}(\mathbb{R})$ on $\partial \mathbb{H}^{2}$, which we identify with $\mathbb{R} \cup \{ \infty \}$. The inverse semigroups alluded to in this section's title act as partial bijections of $\partial \mathbb{H}^{2}$ via (restrictions of) linear fractional transformations.

\begin{definition} \label{definition:inversesemigroup} (partial bijections; inverse semigroups; domains)
Let $X$ be a set.  A \emph{partial bijection} of $X$ is a bijection $h: A_{h} \rightarrow B_{h}$ between subsets $A_{h}$ and $B_{h}$ of $X$. 
The composition of two partial bijections is defined on ``overlaps": if $g: A_{g} \rightarrow B_{g}$ and $h:A_{h} \rightarrow B_{h}$ are partial bijections of $X$, then $g \circ h$ is a bijection from $h^{-1}(A_{g})$ to $g(B_{h} \cap A_{g})$. 

A collection $S$ of partial bijections of $X$ is called an \emph{inverse semigroup} if $S$ is closed under inverses and compositions. We may also refer to such an $S$ as an \emph{inverse semigroup acting on $S$}.

If $S$ is an inverse semigroup and $h: A_{h} \rightarrow B_{h}$, then we refer to $A_{h}$ as a \emph{domain} of $S$. Note that $B_{h}$ is also a domain, since $S$ is closed under inverses. We let $\mathcal{D}_{S}$
denote the set of all domains $A_{h}$, as $h$ ranges over all $h \in S$. We let $\mathcal{D}^{+}_{S} = \mathcal{D}_{S} - \{ \emptyset \}$ (i.e., the set of all non-empty domains). 
\end{definition}

\begin{remark} \label{remark:Cayleylike}
Inverse semigroups can also be defined abstractly (see \cite{SemigroupBook}). The Preston-Wagner Theorem states that any inverse semigroup can be realized as a collection of partial bijections (in the above sense). The proof parallels that of Cayley's Theorem, which states that every group can be realized as a group of permutations.
\end{remark}

\begin{remark} \label{remark:zeroelement} In Definition \ref{definition:inversesemigroup}, the function with empty domain and codomain plays the role of a $0$.
\end{remark}

\begin{definition} \label{definition:S} (The inverse semigroups $S_{n}$ and $S'_{n}$) 
Let $A: [0,1) \rightarrow [0, 1/2)$ be the restriction of the linear fractional transformation 
\[ A = \left( \begin{smallmatrix} 1 & 0 \\ 1 & 1 \end{smallmatrix} \right). \]
Let $B:[0,1) \rightarrow [1/2, 1)$ be the restriction of
\[ B = \left( \begin{smallmatrix} 0 & 1 \\ -1 & 2 \end{smallmatrix} \right). \]
For $n=2,3$, let $C_{n}: [0,1) \rightarrow [0,1)$ be the restriction of
\[ C_{n} = \left(\begin{smallmatrix} n & 0 \\ n-1 & 1 \end{smallmatrix} \right). \]
Let $T: [0, \infty) \rightarrow [1, \infty)$ be the restriction of
\[ T = \left(\begin{smallmatrix} 1 & 1 \\ 0 & 1 \end{smallmatrix} \right). \]
We will use lower-case letters to denote the inverses of the above transformations ($a = A^{-1}$, etc.).

For $n=2,3$, we let
\[ S_{n} = \langle A, B, C_{n} \rangle; \quad S'_{n} = \langle A, B, C_{n}, T \rangle. \]
\end{definition}

\begin{definition} \label{definition:locallydetermined} (locally determined by $S$; the inverse semigroup $\widehat{S}$) 
Let $S$ be an inverse semigroup acting on a set $X$. Let $A$ and $B$ be subsets of $X$. 
A bijective function $f: A \rightarrow B$ is \emph{locally determined by $S$} if there is a finite partition 
$\mathcal{P} = \{ D_{1}, \ldots, D_{m} \}$ of
$A$ into domains (i.e., $\mathcal{P} \subseteq \mathcal{D}^{+}_{S}$) such that $f_{\mid D_{i}} \in S$, for each $i$. 

We let $\widehat{S}$ denote the collection of all functions that are locally determined by $S$. The set $\widehat{S}$ is an inverse semigroup under the operation of composition.
\end{definition}

\begin{definition} \label{definition:Gamma}
Let $n = 2$ or $3$. Let
\begin{itemize}
\item $F(S_{n})$ be the group of homeomorphisms of $[0,1)$ that are locally determined by $S_{n}$;
\item $F(S'_{n})$ be the group of homeomorphisms of $[0,\infty)$ that are locally determined by $S'_{n}$;
\item $T(S_{n})$ be the group of homeomorphisms of the circle $[0,1]/$$\sim$ that are locally determined by $S_{n}$;
\item $V(S_{n})$ be the group of right-continuous bijections of $[0,1)$ that are locally determined by $S_{n}$;
\item $V(S'_{n})$ be the group of right-continuous bijections of $[0,\infty)$ that are locally determined by $S'_{n}$.
\end{itemize}
\end{definition}

\begin{remark}
The group $T(S_{n})$ can also be described as the subgroup of $V(S_{n})$ that preserves a cyclic ordering on $[0,1)$.
\end{remark}

%%%%%%%%%%%%%%%%%%%
\section{A generating set of domains for $S_{i}$ and $S'_{i}$} \label{section:generatingdomains}
%%%%%%%%%%%%%%%%%%%%%%%%

The set $\mathcal{D}^{+}_{S}$ (Definition \ref{definition:inversesemigroup}) will be far too big when $S = S_{i}$ or
$S = S'_{i}$. In this section, we define a subcollection $\mathcal{D}^{+}_{S,gen} \subseteq \mathcal{D}^{+}_{S}$, which will be sufficient for the constructions of later sections. We note that this is in contrast with \cite{FH2}, which always uses the full set $\mathcal{D}^{+}_{S}$.

\begin{definition} \label{definition:generatingsetofdomains}
(Generating sets of domains)
Let $\{ A, B \}^{\ast}$ denote the set of all positive words in the alphabet $\{ A, B \}$, including the empty word.
Let 
\[ \mathcal{D}_{S',gen}^{+} = \{ T^{\alpha}\omega \cdot [0,1) \mid \omega \in \{A, B\}^{\ast}; \alpha \geq 0 \} \cup \{ T^{\alpha} \cdot [0, \infty) \mid \alpha \geq 0 \}\]
and
 \[ \mathcal{D}_{S,gen}^{+} = \{ \omega \cdot [0,1) \mid \omega \in \{A, B\}^{\ast} \}\]
 We will often refer to $\mathcal{D}_{S,gen}^{+}$ or $\mathcal{D}_{S',gen}^{+}$ by the notation $\mathcal{D}_{gen}^{+}$ if doing so should cause no ambiguity.
\end{definition}

\begin{remark} \label{remark:discussionofintervals}
It will be convenient to write $I$ in place of $[0,1)$, and to write $\omega I$ in place of $\omega \cdot I$. 

The half-open intervals $\omega I$ of $\mathcal{D}_{S,gen}^{+}$ are in one-to-one correspondence with the vertices of an infinite binary tree. The intervals $\omega AI$ and 
$\omega BI$ correspond to the left and right children (respectively) of $\omega I$.
In particular, $\omega' I$ contains $\omega I$ if and only if $\omega'$ is a prefix of 
$\omega$, and the intervals are disjoint if neither $\omega'$ nor $\omega$ is a prefix of the other. 

Note, however, that the intervals $\omega I$ are very far from being the standard dyadic intervals when the length of $\omega$ is two or more. For instance, $ABAI = [1/3, 2/5)$ and $BAI = [1/2, 2/3)$. It appears that $\omega I$, for $\omega \in \{ A, B \}^{\ast}$, is always an interval between consecutive Farey fractions (as noted in \cite{LM}), although we will not need to use this fact. 
The intervals $T^{\alpha}\omega I$ are simply the translates of the intervals $\omega I$
by non-negative integers.

It will be useful to keep in mind that the products $aB$ and $bA$ are $0$ in what follows. 
\end{remark}
  
\begin{lemma} \label{lemma:eventual} (An eventual invariance property) Let $s \in S_{i}$, where $i=2$ or $3$. Let $D \in \mathcal{D}^{+}_{S,gen}$ be such that $D$ is contained in the domain of $s$. There is a finite partition $\mathcal{P} \subseteq \mathcal{D}^{+}_{S,gen}$ of $D$ such that $sP \in \mathcal{D}^{+}_{S,gen}$, for each $P \in \mathcal{P}$.

The analogous statement also holds true for $S'_{n}$, $i=2,3$.
\end{lemma}

\begin{proof}
Let $D = \omega I$, where $\omega \in \{ A, B \}^{\ast}$.
By induction on the length of $s$, it suffices to prove the lemma in the case 
$s \in \{ A, B, C_{2}, C_{3}, a, b, c_{2}, c_{3} \}$.

Suppose first that $s = A$. 
It follows directly that $sD = A\omega I$, so $sD \in \mathcal{D}^{+}_{S,gen}$ and we 
may set $\mathcal{P} = \{ D \}$. If $s = a$, it must be that $\omega =  A \omega'$, for some $\omega' \in \{ A, B \}^{\ast}$, and therefore $sD = \omega' I$. We can therefore again let $\mathcal{P} = \{ D \}$. 

If $s = B$ or $s= b$, the proof is very similar. 

Let $s = C_{2}$. A straightforward check shows that 
\begin{align*}
C_{2}AA &= AC_{2}; \\
C_{2}AB &= BAc_{2}; \\
C_{2}B &= BBC_{2}
\end{align*}
We use these identities to ``push'' $C_{2}$ as close to the end of the word $\omega$ as possible. (Note that the inverse $c_{2}$ can appear during this process. This poses no problems, since we can use the same identities to push $c_{2}$ forward, as well.)
In doing so, we can arrange that 
\[ C_{2}\omega = \omega'C_{2}^{\epsilon} \omega'', \]
 where $\omega', \omega'' \in 
\{ A, B \}^{\ast}$, $\epsilon = \pm 1$, and 
\begin{enumerate}
\item $\omega''$ is an empty word, or 
\item $\omega'' = A$ if $\epsilon = 1$, or 
\item $\omega'' = B$ if $\epsilon = -1$. 
\end{enumerate}
If $\omega''$ is not the empty word, we then 
let $\mathcal{P} = \{ \omega AI, \omega BI \}$; if $\omega''$ is empty, we set 
$\mathcal{P} = \{ D \}$; these are the required partitions. (For instance, if $\epsilon = 1$ and
$\omega'' = A$, we have 
\[ C_{2}\omega AI = \omega' C_{2}AAI = \omega' AC_{2}I = \omega' AI \in
\mathcal{D}^{+}_{gen} \]
and
\[ C_{2}\omega BI = \omega' C_{2}ABI = \omega' BAc_{2}I = \omega' BAI \in
\mathcal{D}^{+}_{gen}.\]
Similar checking handles the remaining case.) 

The case in which $s = c_{2}$ is very similar, and features the same identities, suitably rewritten so that $c_{2}A$, $c_{2}BA$, and $c_{2}BB$ appear on the left-hand sides of the equations.

Now suppose that $s = C_{3}$. We have the following matrix identities:
\begin{align*}
C_{3}AAA &= AC_{3}; \\
C_{3}AAB &= BAAc_{3}; \\
C_{3}ABA &= BABC_{3}; \\
C_{3}ABB &= BBAc_{3}; \\
C_{3}B &= BBBC_{3}; 
\end{align*}
We can then follow the same strategy as we did in the case $s = C_{2}$. ``Push" $C_{3}$ as close to the end of $\omega$ as possible. The result is $\omega' C_{3}^{\epsilon} \omega''$, where $\omega' \in \{ A, B \}^{\ast}$ and 
\begin{enumerate}
\item $\omega''$ is empty, or
\item $\epsilon = 1$ and 
$\omega'' \in \{ A, AA, AB \}$, or 
\item $\epsilon = -1$ and $\omega'' \in \{ B, BA, BB \}$. 
\end{enumerate}
If $\omega''$ is empty, then $\mathcal{P} = \{ D \}$ is the required partition of $D$. If $\epsilon =1$ and $\omega'' = A$, then we set $\mathcal{P} = \{ \omega AAI, \omega ABI, \omega BAI, \omega BBI \}$. This is the required partition; indeed,
\[ C_{3}\omega AAI = \omega' C_{3} AAAI = \omega' AC_{3}I = \omega'AI \in \mathcal{D}^{+}_{gen}, \]
and similar calculations show that $C_{3}(\mathcal{P}) \subseteq \mathcal{D}^{+}_{gen}$. If $\omega'' \in \{ AA, AB \}$, then the required partition is $\mathcal{P} = \{ \omega AI, \omega BI \}$. If $\epsilon = -1$ and $\omega'' \in \{ B, BA, BB \}$, then one proceeds similarly. The required partitions are $\{ \omega AAI,
\omega ABI, \omega BAI, \omega BBI \}$ (in the first case, when $\omega'' = B$) and 
$\{ \omega AI, \omega BI \}$ (when $\omega'' = BA$ or $BB$).

The extension of these arguments to $S'_{i}$ is entirely straightforward.
\end{proof}

 %%%%%%%%%%%%%%%%%%%%%%
 \section{A directed set construction} \label{section:directedset}
 %%%%%%%%%%%%%%%%%%%%%%

 In this section, we will specify an $S$-structure $\mathbb{S}$ for $S \in \{ S_{2}, S_{3}, S'_{2}, S'_{3} \}$, in essentially the sense of \cite{FH2}. In fact, the only difference is that we will define our $S$-structure using the domains $\mathcal{D}^{+}_{S,gen}$
 and $\mathcal{D}^{+}_{S',gen}$, rather than the entire collection $\mathcal{D}^{+}$, as required in \cite{FH2}. The $S$-structure leads to a directed set construction of a contractible simplicial complex, exactly as in \cite{FH2}. We will first consider these directed set constructions for
 $V(S_{i})$ and $V(S'_{i})$ ($i=2,3$) in Subsection \ref{subsection:V}. The simplicial complexes for the related groups can then be obtained as subcomplexes; this is spelled out in Subsection \ref{subsection:FT} 
 
 It will be useful to let $\mathcal{D}^{+}_{gen}$ denote either $\mathcal{D}^{+}_{S,gen}$ or $\mathcal{D}^{+}_{S',gen}$, depending on the context.
 
%%%%%%%%%%%%%%%
\subsection{The directed set constructions for $V(S_{i})$ and $V(S'_{i})$}  \label{subsection:V}
%%%%%%%%%%%%%%%%

We will first show how to make $V(S_{i})$ and $V(S'_{i})$ act on directed sets, and (therefore) on contractible simplicial complexes. 
The basic approach follows \cite{FH2}, but we are able to use a simplified version of the basic theory, with suitable modifications.
All of the results in this subsection work in the same way for all $\Gamma \in \{ V(S) \mid S \in \{ S_{2}, S_{3}, S'_{2}, S'_{3} \} \}$, so we will use the generic notation $\Gamma$ to refer to any group from the latter collection.

\begin{definition} \label{definition:structuresets} (Structure sets; domain types)
Let $S \in \{ S_{2}, S_{3}, S'_{2}, S'_{3} \}$.
Let $D_{1}, D_{2} \in \mathcal{D}^{+}_{gen}$. We set 
\[ \mathbb{S}(D_{1},D_{2}) = \{ s \in S \mid \text{The domain of }s \text{ is }D_{1} \text{ and the range is }D_{2} \}. \]

Two domains $D_{1}$ and $D_{2}$ have the \emph{same type} if $\mathbb{S}(D_{1},D_{2}) \neq \emptyset$; i.e., if there is some $s \in S$ such that $D_{1}$ is the domain of $s$ and $D_{2}$ is the image of $s$.
\end{definition}

\begin{remark} \label{remark:domaintype} (description of domain types)
There are one or two domain types, depending on whether $S \in \{ S_{2}, S_{3} \}$, on the one hand, or $S \in \{ S'_{2}, S'_{3} \}$, on the other. The first of the domain types consists of those of the form $\omega I$, where $\omega \in \{ A, B \}^{\ast}$. These domain types occur in both $S_{i}$ and $S'_{i}$, and they are the only types if $S \in \{ S_{2}, S_{3} \}$. The second of the domain types (present only when $S \in \{ S'_{2}, S'_{3} \}$) consists of domains of the form $[n, \infty)$, where $n$ is a non-negative integer. 
\end{remark}

\begin{theorem} \label{theorem:structureofstructuresets} (Explicit description of structure sets)
Let $S = S_{2}$ or $S_{3}$.
Given $\omega I$ and $\omega' I \in \mathcal{D}^{+}_{gen}$ ($\omega, \omega' \in \{ A, B \}^{\ast}$), the associated structure set takes the following form:
\[ \mathbb{S}(\omega I, \omega' I) = \{ \omega' C^{k} \omega^{-1} \mid k \in \mathbb{Z} \}. \] 

Let $S = S'_{2}$ or $S'_{3}$.
Given $\omega I$ and $\omega' I \in \mathcal{D}^{+}_{gen}$ ($\omega, \omega' \in \{ A, B, T \}^{\ast}$), the associated structure set takes the following form:
\[ \mathbb{S}(\omega I, \omega' I) = \{ \omega' C^{k} \omega^{-1} \mid k \in \mathbb{Z} \}. \]
The set $\mathbb{S}([m,\infty), [n,\infty))$ is $\{ T^{n-m} \}$, when $m$,$n$ are non-negative integers.
\end{theorem}

\begin{proof}
We first consider the case of $\mathbb{S}(I,I)$ when $S = S_{2}$.
 
Let $\omega \in \mathbb{S}(I,I)$. Thus, $\omega I = I$. Let $G$ be the group generated by the linear fractional transformations $A$, $B$, and $C$, each viewed as a transformation of $\mathbb{H}^{2}$ or the projective line $\mathbb{R} \cup \{ \infty \}$. We note that the inverses of $A$, $B$, and $C$ may be represented by the matrices
\[ a = \left( \begin{smallmatrix} 1 & 0 \\ -1 & 1 \end{smallmatrix} \right); \quad
b = \left( \begin{smallmatrix} 2 & -1 \\ 1 & 0 \end{smallmatrix} \right); \quad
c = \left( \begin{smallmatrix} 1 & 0 \\ -1 & 2 \end{smallmatrix} \right). \]
It follows that, if $\omega$ is expressed as a product of matrices, then
$\mathrm{det}(\omega) = 2^{n}$, for some nonnegative integer $n$. We note also that $\omega$ fixes the points $0$ and $1$ on the projective line, by our assumptions.

We let 
\[ \omega = \left( \begin{smallmatrix} \alpha & \beta \\ \gamma & \delta \end{smallmatrix} \right), \]
where $\alpha$, $\beta$, $\gamma$, and $\delta$ are integers.
The equality $\omega(0) = 0$ directly implies that $\beta = 0$. The equality 
$\omega(1) = 1$ then implies that $\alpha = \gamma + \delta$. Computing determinants, we find that
\[ (\gamma + \delta) \delta = 2^{n}. \]
It follows that $\gamma + \delta = 2^{k}$ and $\delta = 2^{\ell}$, where $k+\ell =n$ and $k$ and $\ell$ are nonnegative integers. Either $k \leq \ell$ or $\ell \leq k$; in the first case,
\[ \left(\begin{smallmatrix} 2^{k} & 0 \\ 2^{k} - 2^{\ell} & 2^{\ell} \end{smallmatrix} \right) \sim
\left(\begin{smallmatrix} 1 & 0 \\ 1 - 2^{\ell - k} & 2^{\ell -k} \end{smallmatrix} \right) = c^{\ell - k}. \]
Similarly, if $\ell \leq k$, then $\omega = C^{k - \ell}$. In either case,
$\omega = C^{\alpha}$, for appropriate $\alpha$.

If $S = S_{3}$, then the set $\mathbb{S}(I,I)$ takes the same form. Here
\[ C = \left(\begin{smallmatrix} 3 & 0 \\ 2 & 1 \end{smallmatrix} \right); \quad c = \left( \begin{smallmatrix} 1 & 0 \\ -2 & 3 \end{smallmatrix} \right), \]
and therefore $\mathrm{det}(\omega) = 3^{n}$, for some $n \in \mathbb{Z}$. The remainder of the argument differs from the case of $S_{2}$ primarily in the fact that it involves powers of $3$, rather than powers of $2$.

Now we consider a general structure set $\mathbb{S}(\omega I, \omega' I)$, where $\omega, \omega' \in \{ A, B \}^{\ast}$ and $S$ may be any of the semigroups $S_{2}$, $S_{3}$, $S'_{2}$, or $S'_{3}$. Let $\sigma \in \mathbb{S}(\omega I, \omega' I)$. It follows that
$(\omega')^{-1} \sigma \omega \in \mathbb{S}(I,I)$, so 
$(\omega')^{-1} \sigma \omega = C^{k}$, for some $k \in \mathbb{Z}$. Thus,
$\sigma = \omega' C^{k} \omega^{-1}$, as claimed. Conversely, it is clear that any transformation of the form $\omega' C^{k} \omega^{-1}$ is in $\mathbb{S}(\omega I, \omega' I)$. 

The final statement, about $\mathbb{S}([m,\infty), [n,\infty))$, is straightforward.
\end{proof}

\begin{proposition} \label{proposition:closure}
(Closure properties of $\mathbb{S}$)
Let $D_{1}, D_{2} \in \mathcal{D}^{+}_{gen}$. 
\begin{enumerate}
\item (compositions) If $h \in \mathbb{S}(D_{1}, D_{2})$ and $g \in \mathbb{S}(D_{2},D_{3})$,
then $gh \in \mathbb{S}(D_{1},D_{3})$;
\item (inverses) If $h \in \mathbb{S}(D_{1},D_{2})$, then $h^{-1} \in \mathbb{S}(D_{2},D_{1})$;
\item (identities) $id_{D_{1}} \in \mathbb{S}(D_{1}, D_{1})$.
\end{enumerate}
\end{proposition}

\begin{proof}
All of these properties follow directly from the definition of the set $\mathbb{S}(D_{1},D_{2})$.
\end{proof}

\begin{definition} \label{definition:B}
(The set $\mathcal{B}$)
Let 
\[ \mathcal{A} = \{ (f,D) \mid f \in \widehat{S}; D \in \mathcal{D}^{+}_{gen}; D
\text{ is contained in the domain of } f \}. \]
[Recall that $\widehat{S}$ is the inverse semigroup of functions that are locally determined by $S$ (Definition \ref{definition:locallydetermined}).]
We write $(f_{1}, D_{1}) \sim (f_{2},D_{2})$ if there is some
$h \in \mathbb{S}(D_{1}, D_{2})$ such that $f_{1} = f_{2}h$. 
It is easily checked that $\sim$ is an equivalence relation on $\mathcal{A}$, using Proposition \ref{proposition:closure}. We let $\mathcal{B}$ denote the set of all equivalence classes. The equivalence class of $(f,D)$ will be denoted $[f,D]$.
\end{definition}

\begin{definition} \label{definition:vertices}
(vertices; the type of a vertex)
A finite subset 
\[ \{ [f_{1}, D_{1}], \ldots, [f_{m}, D_{m}] \} \subseteq \mathcal{B} \]
is a \emph{vertex} if
\[ \bigcup_{i=1}^{m} f_{i}(D_{i}) = \mathbb{R}^{+} \text{ or } [0,1), \]
depending upon whether the underlying semigroup is $S'_{n}$ or $S_{n}$, respectively.
(Here $\mathbb{R}^{+}$ is the set of non-negative real numbers and $m$ may be any natural number.)

We let $\mathcal{V}_{S}$ denote the set of all vertices, where $S \in \{ S_{2}, S_{3}, S'_{2}, S'_{3} \}$. We may sometimes write 
$\mathcal{V}$ in place of $\mathcal{V}_{S}$ if this will result in no ambiguity.

Two vertices $\{ [f_{1}, D_{1}], \ldots, [f_{m}, D_{m}] \}$ and $\{ [g_{1}, E_{1}], \ldots, [g_{n},E_{n}] \}$ have the \emph{same type}
if the multisets $\{ [D_{1}], \ldots, [D_{m}] \}$ and $\{ [E_{1}], \ldots, [E_{n}] \}$ are identical; i.e., $m=n$ and $[D_{j}] = [E_{j}]$, for 
$j = 1, \ldots, m$.
\end{definition}

\begin{definition} \label{definition:expansion} (expansion; contraction)
Let $v = \{ [f_{1}, D_{1}], \ldots, [f_{n},D_{n}] \}$ be a vertex. We say that a vertex $v'$ is obtained from $v$ by \emph{expansion at $[f_{i},D_{i}]$} if there is some $h \in \mathbb{S}(D_{i},D_{i})$ and a finite partition
$\mathcal{P} \subseteq \mathcal{D}^{+}_{gen}$ of $D_{i}$ into domains such that 
\[ v' = (v - \{ [f_{i}, D_{i}] \}) \cup \{ [f_{i}h, P] \mid P \in \mathcal{P} \}. \]
We write $v \nearrow v'$. We also say that $v$ is the result of \emph{contraction} from $v'$.

We let $\leq$ be the reflexive, transitive closure of $\nearrow$. 
\end{definition}

\begin{remark} \label{remark:explicitexpansion} (an explicit description of expansion)
Consider $[f,\omega I]$, where $\omega \in \{ A, B \}^{\ast}$. We note that 
$[f, \omega I] = [f \omega, I]$ by the definition of $\mathcal{B}$ (Definition \ref{definition:B}) and because $\omega \in \mathbb{S}(I, \omega I)$. An arbitrary partition of
$I$ takes the form 
\[ \{ \tau I \mid \tau \in \mathcal{C} \}, \]
where $\mathcal{C} \subseteq \{ A, B \}^{\ast}$ is a cut set (in the sense of Section \ref{section:algorithm}). It follows directly that an expansion at $[f,\omega I]$ (equivalently, $[f \omega, I]$) 
involves replacing $[f,\omega I]$ by the members of
\[ \{ [f\omega C^{k} \tau, I ] \mid \tau \in \mathcal{C} \}, \]
for some $k \in \mathbb{Z}$ and some cut set $\mathcal{C}$. (This is by Definition \ref{definition:expansion} and Theorem \ref{theorem:structureofstructuresets}.)

The above description is particularly simple when $\mathcal{C}$ is the cut set $\{ A, B \}$. It then follows that the expansion replaces $[f,\omega I]$ with the pairs
\[ [f \omega C^{k}A, I] \quad \text{and} \quad [f \omega C^{k}B, I], \]
 for some $k \in \mathbb{Z}$. Moreover, this is essentially the general case, since any expansion can be realized as a sequence of such expansions.
 
 An expansion at a pair $[f,D]$, where $D = [m,\infty)$, is much more straightforward to describe: such an expansion simply replaces $[f,D]$ with the members of
 \[ \{ [f,P] \mid P \in \mathcal{P} \}, \]
 where $\mathcal{P}$ is a finite partition of $D$ into domains from $\mathcal{D}^{+}_{gen}$. This is because $\mathbb{S}(D,D)$ is trivial.
\end{remark}

\begin{proposition} \label{proposition:welldefinedexpansion}
Expansion is well-defined and $\Gamma$-invariant. I.e.,
\begin{enumerate}
\item if $v = \{ [f_{1}, D_{1}], \ldots, [f_{m},D_{m}] \}$, $\hat{v} = \{ [g_{1},E_{1}], \ldots, [g_{m},E_{m}] \}$, and $v'$ is the result of expansion from $v$ at $[f_{i},D_{i}]$, then $v'$ is also the result of expansion from $\hat{v}$ at some $[g_{j},E_{j}]$.
\item if $v \nearrow v'$ (where $v$ and $v'$ are as above) and $\gamma \in \Gamma$, then
$\gamma \cdot v \nearrow \gamma \cdot v'$.
\end{enumerate}
\end{proposition}

\begin{proof}
We prove (1). Assume, without loss of generality, that $[f_{k}, D_{k}] = [g_{k}, E_{k}]$ for $k = 1, \ldots, m$. We suppose that
$v'$ is the result of expansion from $v$ at $[f_{i},D_{i}]$; thus, there is some 
$h \in \mathbb{S}(D_{i},D_{i})$ and a finite partition $\mathcal{P} \subseteq \mathcal{D}^{+}_{gen}$
of $D_{i}$ such that 
\[ v' = (v - \{ [f_{i},D_{i}] \}) \cup \{ [f_{i}h,P] \mid P \in \mathcal{P} \}. \]
Let $j \in \mathbb{S}(D_{i}, E_{i})$ be such that $j(\mathcal{P})$ is a finite partition
of $E_{i}$ by members of $\mathcal{D}^{+}_{gen}$. Since $[f_{i}, D_{i}] = [g_{i}, E_{i}]$, there is also some $j_{1} \in \mathbb{S}(D_{i}, E_{i})$ such that $g_{i} j_{1} = f_{i}$, by Definition \ref{definition:B}.   

We claim that
\[ \{ [f_{i}h, P] \mid P \in \mathcal{P} \} = \{ [g_{i}j_{1}hj^{-1}, j(P)] \mid P \in \mathcal{P} \}. \]
Indeed, for each $i$, $(g_{i}j_{1}h^{-1}j^{-1}) \circ j = f_{i}h$, so
$[f_{i}h,P] = [g_{i}j_{1}hj^{-1},j(P)]$ by Definition \ref{definition:B}. This proves (1).

The proof of (2) is straightforward. Indeed, 
\[ \gamma \cdot v = \{ [\gamma f_{1}, D_{1}], \ldots, [\gamma f_{m}, D_{m}] \} \]
and
\[ \gamma \cdot v' = ( \gamma \cdot v - \{ [\gamma f_{i}, D_{i}] \} ) \cup
\{ [ \gamma f_{i}h, P] \mid P \in \mathcal{P} \}, \]
from which it directly follows that $\gamma \cdot v'$ is obtained from $\gamma \cdot v$ via 
expansion at $[ \gamma f_{i}, D_{i}]$ (with respect to the same choices of $h$ and $\mathcal{P}$).
\end{proof}  

\begin{corollary} \label{corollary:partialorder} (the partial order on vertices)
The relation $\leq$ is a partial order on $\mathcal{V}$. The group $\Gamma$ acts on $(\mathcal{V}, \leq)$ in an order-preserving fashion.
\end{corollary}

\begin{proof}
This follows directly from Proposition \ref{proposition:welldefinedexpansion}.
\end{proof}

\begin{definition} (simplicial realization; the complexes $\Delta(S_{n})$ and $\Delta(S'_{n})$)
Let $\mathcal{P}$ be a partially ordered set. The \emph{simplicial realization} of
$P$ is the simplicial complex whose vertex set is $P$ and whose simplices are finite ascending chains in $P$.

We let $\Delta(S_{n})$ and $\Delta(S'_{n})$ denote the simplicial realizations of $V(S_{n})$ and $V'(S'_{n})$, respectively.
\end{definition}

\begin{theorem} \label{theorem:directedset} (The directed $\Gamma$-set of vertices)
The relation $\leq$ is a partial order on $\mathcal{V}$, and $\mathcal{V}$ is a directed set with respect to $\leq$. The group $\Gamma$ acts on $(\mathcal{V}, \leq)$ in an order-preserving fashion.

In particular, the simplicial realizations $\Delta(S_{i})$ and $\Delta(S'_{i})$ are contractible $\Gamma$-complexes.
\end{theorem}

\begin{proof}
It is already clear from Proposition \ref{proposition:welldefinedexpansion} and Definition \ref{definition:expansion} that $(\mathcal{V}, \leq)$
is a partially ordered set on which $\Gamma$ acts in an order-preserving fashion.

Let $S = S_{i}$ or $S'_{i}$, for $i=2$ or $3$.
We must show that $(\mathcal{V}, \leq)$ is a directed set. The main step is to show that any 
vertex
\[ v = \{ [f_{1},D_{1}], \ldots, [f_{m},D_{m}] \} \]
can be expanded into a vertex of the form 
\[ \hat{v} = \{ [id_{E_{1}}, E_{1}], \ldots, [id_{E_{n}}, E_{n}] \}. \]

Note that each $D_{i}$ can be partitioned into finitely many elements of $\mathcal{D}^{+}_{gen}$ in such a way that the restriction of $f_{i}$ to each piece acts as a member of $S$ (see Definition \ref{definition:locallydetermined}).
Thus, we may assume, possibly after expansion, that $v$ already has this property. Consider the pair $[f_{1},D_{1}]$. By Lemma \ref{lemma:eventual}, there is a finite partition $\mathcal{P} \subseteq \mathcal{D}^{+}_{gen}$ of $D_{1}$ such that $f_{1}(P) \in \mathcal{D}^{+}_{gen}$, for each $P \in \mathcal{P}$. It follows that
$f_{1} \in \mathbb{S}(P, f_{1}(P))$, for each $P \in \mathcal{P}$, so
\[ \{ [f_{1},P] \mid P \in \mathcal{P} \} = \{ [id, f_{1}(P)] \mid P \in \mathcal{P} \}, \]
by the definition of the equivalence relation on pairs (see Definition \ref{definition:B}). Note that 
the act of replacing $[f_{1},D_{1}]$ by the collection $\{ [f_{i},P] \mid P \in \mathcal{P} \}$ is an expansion at
$[f_{1},D_{1}]$. By performing similar expansions at the remaining $[f_{i}, D_{i}]$ ($i=2, \ldots, m$), we arrive at the required $\hat{v}$.

Now suppose that $v_{1}$ and $v_{2}$ are any two vertices. By the argument of the previous paragraph, we can find $\hat{v}_{1}$ and $\hat{v}_{2}$ such that $v_{1} \leq \hat{v}_{1}$ and $v_{2} \leq \hat{v}_{2}$, and both $\hat{v}_{1}$ and $\hat{v}_{2}$ have the general form of the vertex $\hat{v}$; i.e., each pair in  
$\hat{v}_{i}$ ($i=1,2$) has the form $[id,E]$, where $id$ denotes the identity function on $E$ and $E \in \mathcal{D}^{+}_{gen}$. Thus, we can identify $\hat{v}_{i}$ ($i=1,2$) with a partition of the non-negative real numbers. (Under this identification, $\hat{v}$ would correspond to the partition $\{ E_{1}, \ldots, E_{n} \}$.)

Finally, we observe that the partitions determined by the $\hat{v}_{i}$
 have a common finite refinement $\mathcal{P}'$ that is also a subset of $\mathcal{D}^{+}_{gen}$. Letting
 $\tilde{v}$ denote the vertex corresponding to $\mathcal{P}'$, we find that $\hat{v}_{i} \leq \tilde{v}$, for $i=1,2$.
 Thus, $v_{1}, v_{2} \leq \tilde{v}$, from which it follows that $(\mathcal{V}, \leq)$ is a directed set.
 
 The final statement is standard.
\end{proof}

%%%%%%%%%%%%%%%
\subsection{Vertex stabilizers} \label{subsection:stab}
%%%%%%%%%%%%%

In this subsection, we consider the stabilizer $\Gamma_{v}$, where $v$ is a vertex and
$\Gamma$ is one of the groups $V(S_{i})$ or $V(S'_{i})$ ($i=2$ or $3$). We will largely follow the proof of Proposition 5.3 in \cite{FH2}. We include the proof for the reader's convenience.

\begin{proposition} \label{proposition:freeabelian} (virtually free abelian vertex stabilizers)
Let 
\[ v = \{ [f_{1}, D_{1}], \ldots, [f_{m}, D_{m}] \}, \] where $v \in \Delta(S_{n})$ or $v \in \Delta(S'_{n})$ ($n=2$ or $3$). Let $\Gamma = V(S_{n})$ or $V(S'_{n})$ (respectively). 

The stabilizer group $\Gamma_{v}$ is virtually free abelian of rank at most $m$.
\end{proposition}

\begin{proof}
The elements of the group $\Gamma_{v}$ permute the elements of $v$. That is, for each $\gamma \in \Gamma_{v}$,
there is a permutation $\sigma_{\gamma} \in S_{m}$ such that
\[ \gamma \cdot [f_{j}, D_{j}] = [\gamma \circ f_{j}, D_{j}] = [f_{\sigma_{\gamma}(j)}, D_{\sigma_{\gamma}(j)}]. \]
The assignment $\gamma \mapsto \sigma_{\gamma}$ is a homomorphism from $\Gamma_{v}$ to $S_{m}$; the kernel $K$ of the latter homomorphism has finite index in $\Gamma_{v}$. Each $\gamma \in K$ fixes the members of $v$ pointwise; i.e., $\gamma \cdot [f_{j},D_{j}] = [\gamma \circ f_{j}, D_{j}] = [f_{j},D_{j}]$, for $j=1,\ldots, m$. It follows, from the definition of the equivalence relation, that there are $h_{j} \in \mathbb{S}(D_{j}, D_{j})$ such that 
$\gamma_{\mid f_{j}(D_{j})} = f_{j} h_{j} f_{j}^{-1}$, for $j=1, \ldots, m$. The latter equalities determine an injective homomorphism
\[ \Phi: K \rightarrow \prod_{j=1}^{m} \mathbb{S}(D_{j},D_{j}) \]
defined by the rule $\gamma \mapsto (h_{1}, \ldots, h_{m})$. Since each of the groups $\mathbb{S}(D_{j},D_{j})$
is either cyclic or trivial by Theorem \ref{theorem:structureofstructuresets}, the proposition follows.
\end{proof}

%%%%%%%%%%%%%%%%%%%
\subsection{The directed set constructions for ``$F$" and ``$T$" groups} \label{subsection:FT}
%%%%%%%%%%%%%%%%%%

The ``$F$" and ``$T$" groups act on a subcomplex of the complexes for $\Delta(S_{i})$ and $\Delta(S'_{i})$.

\begin{definition}
Let $\Gamma \in \{ F(S_{i}), F(S'_{i}), T(S_{i}) \}$.
We consider the smallest subcomplex of $\Delta(S_{i})$ (or $\Delta(S'_{i})$ if
$\Gamma = F(S'_{i})$) that contains the vertices $[\gamma, X]$ ($X = [0,1)$ or $[0,\infty)$, respectively), for all $\gamma \in \Gamma$, and is closed under expansion.

We denote this complex by $\Delta_{F}(S_{i})$, $\Delta_{F}(S'_{i})$, or $\Delta_{T}(S_{i})$, respectively.
\end{definition}

\begin{proposition} \label{proposition:contractibilityforFTV}
The vertices of $\Delta_{F}(S_{i})$, $\Delta_{F}(S'_{i})$, and $\Delta_{T}(S_{i})$ form directed sets under expansion.

In particular, the complexes $\Delta_{F}(S_{i})$, $\Delta_{F}(S'_{i})$, and $\Delta_{T}(S_{i})$ are contractible $\Gamma$-simplicial complexes, where $\Gamma = F(S_{i})$, $F(S'_{i})$, or $T(S_{i})$, respectively.
\end{proposition}

\begin{proof}
The $\Gamma$-equivariance of the complexes in question follows from the fact that the expansion relation is $\Gamma$-equivariant. The contractibility of these complexes follows from the fact that the vertex sets are still directed, since the vertex sets in question are closed under expansion.
\end{proof}  

%%%%%%%%%%%%%%%%%%%
 \section{An algorithm} \label{section:algorithm}
 %%%%%%%%%%%%%%%%%%%%
 
 In this section, we will describe a simple algorithm. We first need to set some conventions. 
 Let $A: [0,1) \rightarrow [0,1/2)$ and $B: [0,1) \rightarrow [1/2,1)$ be defined as in Definition \ref{definition:S}.
 The vertices of a rooted infinite binary tree can be labelled by words in the monoid $\{ A, B \}^{\ast}$, as follows: The root is labelled by the empty word. If a given vertex $v$ is labelled by $\omega \in \{ A, B \}^{\ast}$, then the left and right children of $v$ are labelled by $\omega A$ and $\omega B$, respectively. Let us denote the label of $v$ by $L(v)$. We can then assign a half-open interval $I(v)$ to each vertex by the rule
 \[ I(v) = L(v) \cdot I, \]
 where $I$ denotes the interval $[0,1)$. Note that $I(v_{1}) \subseteq I(v_{2})$ if and only if 
 $L(v_{1})$ is a prefix of $L(v_{2})$. 
 
 By a \emph{cut set} of a rooted infinite binary tree, we mean a set $C$ of vertices such that every embedded ray issuing from the root passes through exactly one member of $C$. We may also refer to a set of words in 
$\{ A, B \}^{\ast}$ as a cut set if the corresponding set of vertices is a cut set in the above sense. 
 
 The input of the algorithm is a matrix $C_{n}$, defined as in the introduction:
\[ C_{n}(x) = \frac{nx}{(n-1)x + 1}, \]
where $C_{n}$ is defined only on the interval $[0,1)$. 
The (hoped-for) output is a collection of matrix identities, of the general form
\begin{align*}
C_{n} \omega_{1} &= \omega'_{1}C_{n}^{\pm}; \\
C_{n} \omega_{2} &= \omega'_{2}C_{n}^{\pm}; \\
&\vdots \\
C_{n} \omega_{k} &= \omega'_{k}C_{n}^{\pm}, \\
\end{align*}
where $\omega_{i}, \omega'_{i} \in \{ A, B \}^{\ast}$ for $i=1, \ldots, k$, and the sets
$\{ \omega_{1}, \ldots, \omega_{k} \}$ and $\{ \omega'_{1}, \ldots, \omega'_{k} \}$ are cut sets. (Collections of such equations figured prominently in the proof of Lemma \ref{lemma:eventual}.) Given the above identities and the corresponding cut sets,
we can then define directed sets just as we did in Section \ref{section:directedset}. The groups that are locally determined by $\{ A, B, C_{n}, a, b, c_{n} \}^{\ast}$ would then act on these directed sets exactly as before.

The algorithm works in the following way. Each vertex of the tree is assigned a type. Initially, this type is ``$u$" for all vertices, indicating a vertex of unknown type. (Actually, the program creates new vertex objects during its run time, although we can ignore this detail for the sake of the current discussion.) Each vertex is also assigned a toggle that is initially set to ``$0$". When the program encounters a vertex $v$, it performs an action depending on the type of the vertex, which is one of $n$, $l$, or $u$, and the value of its toggle, which is either $0$ or $1$. 

If the vertex is of unknown type (``$u$"), the program runs the following test:
\begin{enumerate}
\item It first appends $C_{n}$ to the beginning of the string $L(v)$. This initializes the \emph{matrix string product} of $v$, which we will here denote $M(v)$. It is a string over the alphabet $\{ A, B, C_{n}, a, b, c_{n} \}$;
\item The program interprets $M(v)$ as a product of matrices and computes the interval $M(v) \cdot I$: 
\begin{enumerate}
\item if $M(v) \cdot I \subseteq [0,1/2)$, then the program appends $a$ to the front of $M(v)$; the result is defined to be the new $M(v)$. The program then returns to step 2);
\item if $M(v) \cdot I \subseteq [1/2,1)$, then the program appends $b$ to the front of $M(v)$; the result is defined to be the new $M(v)$. The program then returns to 
step 2);
\item if $M(v) \cdot I = I$, then $M(v)$ is necessarily equivalent (as a linear fractional transformation) to a power of $C_{n}$; $k$, let's say [this needs to be previously justified]. In this case, the program appends $c_{n}^{k}$ to the front of $M(v)$ (creating a new $M(v)$). The program now classifies the current vertex as a leaf (changing the unknown ``$u$'' designation to ``$l$"). The toggle of the current vertex is also set to ``$1$" (changed from ``$0$");
\item if $M(v) \cdot I$ satisfies none of the above (i.e., $1/2 \in M(v) \cdot I$, but $M(v) \cdot I \neq I$), then $v$ is reclassified as an (internal) node ``$n$". The toggle stays at $0$.
\end{enumerate}
\end{enumerate}
At the end of the above process, the vertex $v$ has been reclassified as an internal node (``$n$") or a leaf (``$l$"). In the latter case, the toggle has been set to $1$ and a certain matrix string product has been produced. By construction, the (final) matrix string product of a leaf necessarily evaluates to the identity matrix when interpreted as a product of matrices. 

If the current vertex $v$ is an internal node (i.e., designated by ``$n$") and its toggle value is $0$, then the program determines the toggle value of the left child of $v$. If this value is $0$, it moves to the left child. If the toggle value of the left child is $1$, but the toggle value of the right child is $0$, then the program moves to the right child. If both children have toggle value $1$, then the program flips the toggle of $v$ to $1$. 

If the toggle value of $v$ is $1$, then the program moves to the parent of $v$. If there is no such parent (i.e., $v$ is the root), then the program terminates, and records the matrix string products for each leaf. The latter matrix string products are readily interpretable as a collection of identities having the desired form, indicated above. The leaves determine a (finite) cut set. This completes the description of the algorithm.

We omit the proof of the validity of the algorithm -- i.e., the proof that the program finds appropriate cut sets and associated matrix identities, if such things exist.

The author's experience of running the program has led to unexpected results. If $n=2$, then the program finds a cut set with three elements, and returns the three matrix equations ($C_{2}AA = AC_{2}$; etc.) displayed in the proof of Lemma \ref{lemma:eventual}. If $n=3$, then the program finds a cut set with five elements, and the five matrix equations associated to $C_{3}$, as described in the proof of Lemma \ref{lemma:eventual}. If $n=4$, the program fails to terminate, although it finds many leaves during its run time. The same is true for all values of $n \geq 4$ that the author has tried. [It may be worth noting here that the program computes using only integer values, not floating-point numbers, so round-off errors are apparently not a source for the problems that are encountered here.] 
It follows from this that an analysis of the groups $V(S_{n})$ and $V(S'_{n})$ for $n \geq 4$ (and, indeed, the corresponding ``$F$" and ``$T$'' versions of these groups)
lies beyond the techniques described in this paper.

It is also possible to run similar tests for different transformations. One might change not only $C_{n}$, but also the transformations $A$ and $B$. The author has run such tests in a few cases, but with no success.

%%%%%%%%%%%%%%%%%%%%%%%%%%%%%%
\section{The expansion schemes $\mathcal{E}_{i}$ and $\mathcal{E}'_{i}$} \label{section:expansion}
%%%%%%%%%%%%%%%%%%%%%%%%%%%%%%

In this section, we will introduce subdivision trees as a device for diagramming expansions, and describe how subdivision trees represent partitions of $[0,1)$ into subintervals. Similar trees were considered in \cite{LM}.

We will then describe expansion schemes $\mathcal{E}_{i}$ and $\mathcal{E}'_{i}$, which will eventually be used to simplify the directed set constructions from Section \ref{section:directedset}. In order to establish the required properties of $\mathcal{E}_{i}$ and $\mathcal{E}'_{i}$, we will need to understand when two subdivision trees define the same partition. The latter will be the project of Sections \ref{section:FP} and \ref{section:IVT}.

%%%%%%%%%%%%%%%%%%%
\subsection{Subdivision trees, equivalence, and elementary equivalence}
%%%%%%%%%%%%%%%%%%%

\begin{definition} \label{definition:subdivtree} (Subdivision trees)
Let $T$ be a finite rooted binary tree. The vertices of degree one are \emph{leaves}; all other vertices are \emph{nodes}. The topmost node is the \emph{root}. We say that $T$ is a \emph{subdivision tree} if each node is labelled by an integer. 

We let $T_{\ell}$ and $T_{r}$ denote the left and right branches of the subdivision tree $T$.
\end{definition}

\begin{figure}[!h]
\begin{center}
\begin{tikzpicture}
%right figure
\draw[gray](6.5,3) -- (7,4); 
\node at (7,3.6) {2};
\draw[gray](7,4) -- (7.5,3);
\node at (7.5,2.6) {-1};
\draw[gray](6.1,2) -- (6.5,3);
\node at (6.5,2.6) {1};
\draw[gray](6.5,3) -- (6.9,2);
\node at (6.9,1.6){3};
\draw[gray](7.1,2) -- (7.5,3);
\draw[gray](7.5,3) -- (7.9,2);
\draw[gray](6.5,1) -- (6.9,2);
\draw[gray](6.9,2) -- (7.3,1);

%%left figure
\draw[gray] (1.5,1) -- (3,4);
\node at (3,3.6){0};
\draw[gray] (2,2) -- (2.5,1);
\node at (2,1.6){-1};
\draw[gray] (2.5,3) -- (3,2);
\node at (2.5,2.6){1};
\draw[gray] (3,4) -- (3.5,3);
\end{tikzpicture}
\end{center}
\caption{Subdivision trees}
\label{figure:subdivtree}
\end{figure}
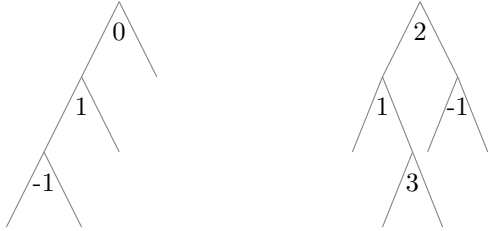

\begin{remark} \label{remark:subdiv} (the subdivision represented by a subdivision tree)
Each leaf in a subdivision tree is labelled by a word in the alphabet $\{ A, B, C, c \}$. The labelling is obtained as follows. Trace the (unique) path $p$ from the root to a given leaf $\ell$. Suppose that 
\[ v_{1}, e_{1}, v_{2}, e_{2}, \ldots, e_{k}, v_{k+1} \]
is a complete list of the vertices and edges encountered along the path $p$, written in the order that they are encountered. Thus, in particular, $v_{1}$ is the root of the tree and $v_{k+1} = \ell$. Let $n_{1}, n_{2}, \ldots, n_{k}$ be the integers labelling the nodes $v_{1}, \ldots, v_{k}$; for $i=1, \ldots, k$, let $X_{i}$ be $A$ if $e_{i}$ points downward and to the left, and let $X_{i}$ be $B$ if $e_{i}$ points downward and to the right. The labelling of the leaf $\ell$ is then
\[ C^{n_{1}}X_{1}C^{n_{2}}X_{2}\ldots C^{n_{k}}X_{k}. \]

For instance, the leaves of the left tree in Figure \ref{figure:subdivtree} are labelled by  the words $ACAcA$, $ACAcB$, $ACB$, and $B$. The leaves of the right tree  
are labelled by $C^{2}ACA$, $C^{2}ACBC^{3}A$, $C^{2}ACBC^{3}B$, $C^{2}BcA$, and
$C^{2}BcB$. This labelling scheme works the same way, no matter whether we are letting $C$ represent $C_{2}$ or $C_{3}$. 

We obtain a partition of $[0,1)$ by applying these words to the interval $[0,1)$. Thus, $ACAcA$ determines the interval $ACAcA \cdot [0,1)$, and so forth.

The partition of $[0,1)$ determined by the trivial subdivision tree is
$\{ [0,1) \}$.
\end{remark}

\begin{remark} \label{remark:subdivinterpret} (subdivision trees over $S_{2}$ and $S_{3}$)
A subdivision tree $T$ represents one of two subdivisions of the interval $[0,1)$, depending upon whether the ``$C$" is interpreted as $C_{2}$ or $C_{3}$. In most contexts, it should be clear which is intended, but, in cases of possible ambiguity, we may refer to $T$ as \emph{a subdivision tree over $S_{2}$} or \emph{over $S_{3}$}, as the case may be.
\end{remark}

\begin{definition} \label{definition:equivalentsubdivtrees} (Equivalent subdivision trees; the functions $n$ and $N$)
Two subdivision trees $T_{1}$ and $T_{2}$ (both over either $S_{2}$ or $S_{3}$) are \emph{equivalent} if they 
represent the same collection of intervals. We write $T_{1} \approx T_{2}$.

If $T$ is a nontrivial subdivision tree, then $n(T)$ denotes the label of the root. Let
\[ N(T) = \{ n(T') \mid T' \approx T \}. \]
The set $N(T)$ is empty if $T$ is the trivial subdivision tree. 
\end{definition}

\begin{lemma} \label{lemma:finite}(Finiteness of $N(T)$)
If $T$ is a subdivision tree, then $N(T)$ is a finite set.
\end{lemma}

\begin{proof}
We prove the lemma in the case of $S_{2}$, the argument for the case of $S_{3}$ being similar. 

Note that, if $T' \approx T$ and $n(T') = k$, then the collection of intervals $\mathcal{C}$ determined by $T$ refines $\{ [0, 2^{k}/(2^{k}+1)), [2^{k}/(2^{k}+1), 1) \}$. Thus, if $N(T)$ were infinite, $\mathcal{C}$ would refine an infinite partition
of $[0,1)$, which would force $\mathcal{C}$ to be infinite. This is impossible, since $T$ has only finitely many leaves.
\end{proof}

\begin{lemma} \label{lemma:l-rlemma}
Let $T$ and $T'$ be nontrivial subdivision trees (both over $S_{2}$ or $S_{3}$); assume that $n(T) = n(T')$.
Then $T \approx T'$ if and only if $T_{\ell} \approx T'_{\ell}$ and
$T_{r} \approx T'_{r}$.
\end{lemma}

\begin{proof}
We prove the lemma in the case that $T$ and $T'$ are subdivision trees over $S_{2}$; the case of $S_{3}$ differs in only minor ways.  
Assume that $T$ and $T'$ are subdivision trees, and that $n(T) = n(T')= t$.

Assume that $T \approx T'$. Let $\ell_{1}, \ldots, \ell_{k}$ label the leaves of $T_{\ell}$ and let $\ell_{k+1}, \ldots, \ell_{m}$ label the leaves of $T_{r}$. (Here, and throughout the proof, the labels of the leaves are read from left to right, so $\ell_{1}$ is the label of the leftmost leaf of $T_{\ell}$, etc.) Let $\ell'_{1}, \ldots, \ell'_{\hat{k}}$ label the leaves of $T'_{\ell}$ and let $\ell'_{\hat{k}+1}, \ldots, \ell'_{m}$ label the leaves of $T'_{r}$. It follows that  
\[ C^{t}A\ell_{1}, \ldots, C^{t}A\ell_{k}, C^{t}B\ell_{k+1}, \ldots, C^{t}B\ell_{m} \]
label the leaves of $T$ and
\[ C^{t}A\ell'_{1}, \ldots, C^{t}A\ell'_{\hat{k}}, C^{t}B\ell_{\hat{k}+1}, \ldots,
C^{t}B\ell_{m} \]
label the leaves of $T'$.  It follows that the above labels pairwise determine equal intervals, in the given order: $C^{t}A\ell_{1} \cdot [0,1) = C^{t}A\ell'_{1} \cdot [0,1)$, etc. Since $\ell_{k}$ and $\ell'_{\hat{k}}$ label rightmost leaves (of the trees $T_{\ell}$ and $T'_{\ell}$, respectively), $C^{t}A\ell_{k} \cdot [0,1)$ and $C^{t}A\ell'_{\hat{k}} \cdot [0,1)$ have the same supremum, namely $2^{t}/(2^{t}+1)$ (since $\ell_{k} \cdot [0,1)$ and $\ell'_{\hat{k}} \cdot [0,1)$ have the supremum $1$). It follows directly that $C^{t}A\ell_{k}$ and $C^{t}A\ell'_{\hat{k}}$ determine the same interval; thus, $k = \hat{k}$.

It follows easily that $\ell_{j}$ and $\ell'_{j}$ 
determine the same interval, for $j=1, \ldots, k$ (simply cancel $C^{t}A$ in the relevant products). Thus, $T_{\ell} \approx T'_{\ell}$. By similar reasoning, $T_{r} \approx T'_{r}$.

Conversely, assuming that $T_{\ell} \approx T'_{\ell}$ and $T_{r} \approx T'_{r}$, we easily conclude that $T \approx T'$.
\end{proof}

\begin{proposition} \label{proposition:equalityofleaves}(equality of leaves)
Let $\omega, \omega' \in \{ A, B, C, c \}^{\ast}$, where $C = C_{2}$ or $C_{3}$. The intervals $\omega \cdot [0,1)$ and $\omega' \cdot [0,1)$ are equal if and only if $\omega = \omega' C^{k}$, for some $k \in \mathbb{Z}$. 
\end{proposition}

\begin{proof}
We first consider the case in which $\omega, \omega' \in S_{2}$.
 
If $\omega = \omega' C^{k}$, then 
\[ \omega \cdot [0,1) = \omega'C^{k} \cdot [0,1) = \omega' \cdot [0,1), \]
where the final equality follows from the fact that $C \cdot [0,1) = [0,1)$.

Conversely, suppose that $\omega I = \omega' I$. It follows that $(\omega')^{-1} \omega I = I$, so $(\omega')^{-1} \omega \in \mathbb{S}(I,I)$, so $(\omega')^{-1}\omega = C^{k}$, for some $k \in \mathbb{Z}$, by Theorem \ref{theorem:structureofstructuresets}. It follows that $\omega = \omega' C^{k}$. 
\end{proof}

\begin{definition} \label{definition:ee} (elementary equivalence)
The two transformations in Figure \ref{figure:m2} define \emph{elementary equivalence} between subdivision trees over $S_{2}$.
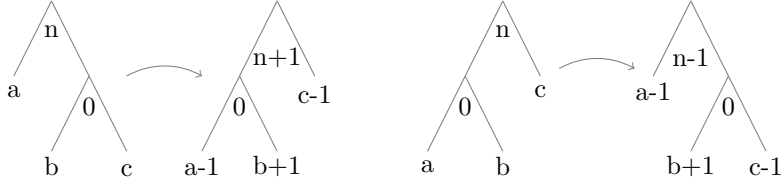
\begin{figure}[!h]
\begin{center}
\begin{tikzpicture}
%left tree of first equivalence
\draw[gray] (1,3) -- (2,1);
\draw[gray] (1,3) -- (.5,2);
\draw[gray] (1.5,2) -- (1,1);
\node at (1,2.6){n};
\node at (1.5, 1.6){0};
\node at (.5, 1.8){a};
\node at (1,.8){b};
\node at (2,.75){c};

%connecting arrow in first equivalence
\draw[gray, ->] (2,2) arc (120:60:1);

%right tree of first equivalence
\draw[gray] (4,3) -- (3,1);
\draw[gray] (3.5,2) -- (4,1);
\draw[gray] (4,3) -- (4.5,2);
\node at (4,2.25){n+1};
\node at (3.5, 1.6){0};
\node at (3,.8){a-1};
\node at (4,.8){b+1};
\node at (4.5,1.75){c-1};

%left tree of second equivalence
\draw[gray] (6,1) -- (7,3);
\draw[gray] (6.5,2) -- (7,1);
\draw[gray] (7,3) -- (7.5,2);
\node at (7,2.6){n};
\node at (6.5,1.6){0};
\node at (6,.8){a};
\node at (7,.8){b};
\node at (7.5, 1.8){c};

%connecting arrow in second equivalence
\draw[gray, ->] (7.75,2.1) arc (120:60:1);

%right tree of second equivalence
\draw[gray] (9,2) -- (9.5,3);
\draw[gray] (9.5,3) -- (10.5,1);
\draw[gray] (9.5,1) -- (10,2);
\node at (9.5, 2.25){n-1};
\node at (10,1.6){0};
\node at (9,1.8){a-1};
\node at (9.5,.8){b+1};
\node at (10.5,.8){c-1};
\end{tikzpicture}
\end{center}
\caption{The relations that define elementary equivalence between subdivision trees over $S_{2}$.}
\label{figure:m2}
\end{figure}

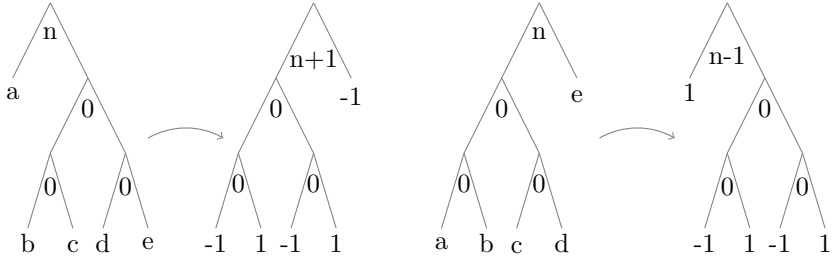
\begin{figure}[!h]
\begin{center}
\begin{tikzpicture}
%left tree of first equivalence M_3
\draw[gray] (1,3) -- (2,1);
\draw[gray] (1,3) -- (.5,2);
\draw[gray] (1.5,2) -- (1,1);
\draw[gray] (.7,0) -- (1,1);
\draw[gray] (1.3,0) -- (1,1);
\draw[gray] (2,1) -- (2.3,0);
\draw[gray] (1.7,0) -- (2,1);
\node at (1,2.6){n};
\node at (1.5, 1.6){0};
\node at (.5, 1.8){a};
\node at (.7, -.2){b};
\node at (1.3, -.24){c};
\node at (1.7, -.2){d};
\node at (2.3, -.2){e};
\node at (1,.55){0};
\node at (2,.55){0};

%connecting arrow in first equivalence M_3
\draw[gray, ->] (2.3,1.2) arc (120:60:1);

%right tree of first equivalence M_3
\draw[gray] (4.5,3) -- (3.5,1);
\draw[gray] (4,2) -- (4.5,1);
\draw[gray] (4.5,3) -- (5,2);
\draw[gray] (3.2,0) -- (3.5,1);
\draw[gray] (3.5,1) -- (3.8,0);
\draw[gray] (4.2,0) -- (4.5,1);
\draw[gray] (4.8,0) -- (4.5,1);
\node at (4.5,2.25){n+1};
\node at (4, 1.6){0};
\node at (3.5,.6){0};
\node at (4.5,.6){0};
\node at (5,1.75){-1};
\node at (3.2,-.2){-1};
\node at (3.8,-.2){1};
\node at (4.2,-.2){-1};
\node at (4.8,-.2){1};

%left tree of second equivalence M_3
\draw[gray] (7.5,3) -- (6.5,1);
\draw[gray] (7,2) -- (7.5,1);
\draw[gray] (7.5,3) -- (8,2);
\draw[gray] (6.2,0) -- (6.5,1);
\draw[gray] (6.5,1) -- (6.8,0);
\draw[gray] (7.2,0) -- (7.5,1);
\draw[gray] (7.8,0) -- (7.5,1);
\node at (7.5,2.6){n};
\node at (7, 1.6){0};
\node at (6.5,.6){0};
\node at (7.5,.6){0};
\node at (8,1.75){e};
\node at (6.2,-.2){a};
\node at (6.8,-.2){b};
\node at (7.2,-.25){c};
\node at (7.8,-.2){d};

%connecting arrow in second equivalence M_3
\draw[gray, ->] (8.3,1.2) arc (120:60:1);

%right tree of second equivalence M_3
\draw[gray] (10,3) -- (11,1);
\draw[gray] (10,3) -- (9.5,2);
\draw[gray] (10.5,2) -- (10,1);
\draw[gray] (9.7,0) -- (10,1);
\draw[gray] (10.3,0) -- (10,1);
\draw[gray] (11,1) -- (11.3,0);
\draw[gray] (10.7,0) -- (11,1);
\node at (10,2.3){n-1};
\node at (10.5, 1.6){0};
\node at (9.5, 1.8){1};
\node at (9.7, -.2){-1};
\node at (10.3, -.2){1};
\node at (10.7, -.2){-1};
\node at (11.3, -.2){1};
\node at (10,.55){0};
\node at (11,.55){0};

\end{tikzpicture}
\end{center}
\caption{The relations that define elementary equivalence between subdivision trees over $S_{3}$.}
\label{figure:m3}
\end{figure}

To apply one of these transformations to a subdivision tree $T$ is to replace a subtree of the form on the left with a subtree of the form on the right. Here the labels $a$, $b$, $c$ represent the integer labels of the nodes of $T$ that are attached at the leaves labelled by $a$, $b$, $c$ (respectively). An application of the given transformation changes the integer labels of the nodes, as indicated on the right-hand tree. If one of the integers $a$, $b$, $c$ labels a leaf in $T$, then that integer is ignored. We also say that two subdivision trees $T_{1}$ and $T_{2}$ are elementary equivalent over $S_{2}$ if one can be transformed into the other by a sequence of such transformations.

\emph{Elementary equivalence over $S_{3}$} is defined by the tree pairs in Figure \ref{figure:m3}.

In Figure \ref{figure:m3}, the labels on the leaves of the right-hand trees have been abbreviated to avoid creating an over-crowded figure. The leftmost ``$-1$" on the second tree from the left represents ``$a-1$", and so on.
\end{definition}

\begin{remark} \label{remark:inverses} We note that the transformations in Figure \ref{figure:m2} are inverses of each other; similarly for Figure \ref{figure:m3}.
\end{remark}

\begin{example} For instance, Figure \ref{figure:elementaryequivalence} depicts an elementary equivalence between two subdivision trees over $S_{2}$. The right-hand tree is the result of applying the second relation to the left-hand tree at the node labelled by ``$4$". 

\begin{figure} [!h]
\begin{center}
\begin{tikzpicture}
%left tree of example
\draw[gray] (7.5,3) -- (6.5,1);
\draw[gray] (7,2) -- (7.5,1);
\draw[gray] (7.5,3) -- (8,2);
\draw[gray] (6.2,0) -- (6.5,1);
\draw[gray] (6.5,1) -- (6.8,0);
\draw[gray] (6.4,-1) -- (6.2,0);
\draw[gray] (5.8,-1) -- (6.2,0);
\draw[gray] (6.6,-1) -- (6.8,0);
\draw[gray] (7.2,-1) -- (6.8,0);
\node at (7.5,2.6){2};
\node at (7, 1.6){4};
\node at (6.5,.5){0};
\node at (6.15,-.65){-1};
\node at (6.85,-.65){3};

%connecting arrow 
\draw[gray, ->] (8.25,1) arc (120:60:1);

%right tree of example
\draw[gray] (11,3) -- (10,1);
\draw[gray] (10.5,2) -- (11,1);
\draw[gray] (11,3) -- (11.5,2);
\draw[gray] (9.7,0) -- (10,1);
\draw[gray] (10,1) -- (10.3,0);
\draw[gray](11,1) -- (10.7,0);
\draw[gray](11,1) -- (11.3,0);
\draw[gray](10.7,0) -- (10.4,-1);
\draw[gray](10.7,0) -- (11,-1);
%\draw[gray] (9.4,-1) -- (9.2,0);
%\draw[gray] (8.8,-1) -- (9.2,0);
%\draw[gray] (9.6,-1) -- (9.8,0);
%\draw[gray] (10.2,-1) -- (9.8,0);
\node at (11,2.6){2};
\node at (10.5, 1.6){3};
\node at (10,.35){-2};
\node at (11,.35){0};
\node at (10.7,-.65){4};
%\node at (9.15,-.65){-1};
%\node at (9.85,-.65){3};

\end{tikzpicture}
\end{center}
\caption{Elementary equivalence between two subdivision trees over $S_{2}$.}
\label{figure:elementaryequivalence}
\end{figure}
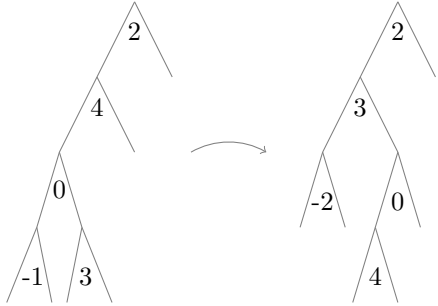
One easily checks that the two trees are indeed equivalent.
\end{example}

\begin{lemma} \label{lemma:elementaryequivalence}
If two subdivision trees $T_{1}, T_{2}$ (over $S_{2}$ or $S_{3}$) are elementary equivalent, then they are equivalent.
\end{lemma}

\begin{proof}
The proof that the left-hand transformation in Figure \ref{figure:m2} preserves equivalence relies on the system of equalities
\begin{align*}
C^{n}AC^{a} &= C^{n+1}AAC^{a-1}; \\
C^{n}BAC^{b} &= C^{n+1}BAC^{b+1}; \\
C^{n}BBC^{c} &= C^{n+1}BC^{c-1},
\end{align*}
all of which are easily verified, and from the interpretation of subdivision trees (Remark \ref{remark:subdiv}. The other three verifications follow similarly.
\end{proof}

%%%%%%%%%%%%%%%%%
\subsection{The correspondence between subdivision trees and expansions} \label{subsection:subdivexp}
%%%%%%%%%%%%%%%%%%

\begin{theorem} \label{theorem:thelink}
(subdivision trees and expansions) We let $C$ denote either $C_{2}$ or $C_{3}$.
Let $v \in \mathcal{V}$ be the result of a sequence of expansions from $\{ [id_{I},I] \}$; i.e.,  $\{ [id,I] \} \leq v$. It follows that there is some subdivision tree $T$ such that the set $\mathcal{L}$ of labels on the leaves satisfies
\[ v = \{ [\omega, I] \mid \omega \in \mathcal{L} \}. \]
Conversely, any subdivision tree $T$ determines a vertex $v$ by the above equality, and 
$\{ [id_{I}, I] \} \leq v$ for this $v$.

If $\{ [id_{I},I] \} \leq v, v'$ and $T$ and $T'$ are the subdivision trees corresponding to $v$ and $v'$, then $v=v'$ if and only if $T$ and $T'$ are equivalent.
\end{theorem}

\begin{proof} 
The correspondence between subdivision trees $T$ and vertices $v$ satisfying $\{ [id_{I},I] \} \leq v$ is straightforward, in view of the discussion in Remark \ref{remark:explicitexpansion}. 

We will now show that $v=v'$ if and only if $T$ is equivalent to $T'$. Assume first that $T$ and $T'$ are equivalent. Thus,   
\[ v = \{ [\omega_{1},I], \ldots, [\omega_{n},I] \} \quad \text{and} \quad v' = \{ [\omega'_{1},I], \ldots, [\omega'_{n},I] \}, \] 
where the $\omega_{i}$ are the labels of the leaves of $T$ (listed from left to right) 
and, similarly, $\omega'_{i}$ are the labels of the leaves of $T'$ (also listed from left to right). Since $T$ and $T'$ are equivalent, we have 
\[ \omega_{i} = \omega'_{i} C^{k_{i}}, \]
for $i=1, \ldots, n$ and for some $k_{i} \in \mathbb{Z}$, by Definition \ref{definition:equivalentsubdivtrees} and Proposition \ref{proposition:equalityofleaves}. It follows directly that, for all $i$, $[\omega_{i},I] = [\omega'_{i},I]$, by Definition \ref{definition:B}, letting
$h = C^{k_{i}}$ (since $C^{k_{i}} \in \mathbb{S}(I,I)$). Thus, $v=v'$.

If we carry over the notation from above, the converse essentially follows from the fact that the equality $[\omega_{i},I] = [\omega'_{i},I]$ implies the equality $\omega_{i}=\omega'_{i} C^{k_{i}}$ (for appropriate $k_{i}$); this is a direct consequence of Definition \ref{definition:B} and the description of $\mathbb{S}(I,I)$ from Theorem \ref{theorem:structureofstructuresets}.
\end{proof}

%%%%%%%%%%%%%%%
\subsection{A discussion of expansion schemes; the expansion schemes $\mathcal{E}_{i}$ and $\mathcal{E}'_{i}$} \label{subsection:expansionscheme}
%%%%%%%%%%%%%%%%%

The directed set construction from Section \ref{section:directedset}, and the generalizations considered in \cite{FH2}, lead to complexes that are often too difficult to analyze when (for instance) attempting to establish finiteness properties for the acting group. One device for simplifying the complexes appeared in \cite{FH2} under the name of ``expansion schemes". An expansion scheme $\mathcal{E}$ assigns to each pair $[f,D] \in \mathcal{B}$ a collection of expansions. This assignment determines a simplicial complex $\Delta^{\mathcal{E}}$ in which the simplices are chains
\[ v_{1} < v_{2} < \ldots < v_{n} \]
such that the vertices $v_{i}$ ($i=2, \ldots, n$) are all the result of expansions from $v_{1}$ that are allowed by $\mathcal{E}$. Thus, for instance, the trivial expansion scheme, which allows no expansions, results in a discrete set of vertices. At the opposite extreme, an expansion scheme may impose no restraint at all, resulting in the original directed set construction. Since the topology of $\Delta^{\mathcal{E}}$ depends significantly on the choice of $\mathcal{E}$, it would be useful to have a criterion that recognizes when the complex $\Delta^{\mathcal{E}}$ is $n$-connected. The idea of an ``$n$-connected expansion scheme" offers such a criterion. The necessary definitions follow.

We will begin with a general discussion of expansion schemes; the definitions of $\mathcal{E}_{i}$ and $\mathcal{E}'_{i}$ are in Example \ref{example:THEexpansionscheme}. The definition of ``pseudovertex" is from Section 4 of \cite{FH2}, while the other definitions and theorems in this subsection are from Section 6 of \cite{FH2}.

\begin{definition} \label{definition:pseudovertex}
(pseudovertices) 
Let $v = \{ [f_{1}, D_{1}], \ldots, [f_{m},D_{m}] \} \subseteq \mathcal{B}$. We say that $v$
is a \emph{pseudovertex} if the sets $f_{i}(D_{i})$ $(i=1\ldots, m)$ are pairwise disjoint. 
\end{definition}

\begin{remark} \label{remark:pseudovertex} (the partial order on pseudovertices; the action of $\widehat{S}$ on pseudovertices)
The pseudovertices are partially ordered by expansion, which can be defined exactly as it was for 
vertices (Definition \ref{definition:expansion}). The pseudovertices do not form a directed set, since the \emph{support} of a given pseudovertex
\[ \left( f_{1}(D_{1}) \cup \ldots \cup f_{m}(D_{m}) \right) \]
is invariant under expansion. The proof of Theorem \ref{theorem:directedset} still shows that any two pseudovertices with the same support have an upper bound. Thus, the simplicial realization of the set of all pseudovertices is a disjoint union of contractible sets.

There is a (partial) action of $\widehat{S}$ (Definition \ref{definition:locallydetermined}) on $\mathcal{B}$, defined by 
\[ \hat{s} \cdot [f,D] = [\hat{s}f,D]. \]
This action is defined for suitable $[f,D]$ and $\hat{s}$; i.e., for all pairs $[f,D]$ and $\hat{s} \in \widehat{S}$
such that $f(D)$ is a subset of the domain of $\hat{s}$.
\end{remark}

\begin{definition} \label{definition:expansionscheme} (expansion schemes) 
($\mathcal{E}$-expansion; expansion scheme) Let $\mathcal{PV}$ denote the collection of all pseudovertices. 
Assume that  $\mathcal{E}: \mathcal{B} \rightarrow 2^{\mathcal{PV}}$ satisfies (1)-(3), for each $[f,D] \in \mathcal{B}$. We let $b$, rather than $[f,D]$, denote a typical member of $\mathcal{B}$  in order to simplify notation.
\begin{enumerate}
\item Each $w \in \mathcal{E}(b)$ is the result of a sequence of expansions from $\{ b \}$; i.e., for each $w \in \mathcal{E}([f,D])$, $\{ [f,D] \} \leq w$;
\item $\{ b \} \in \mathcal{E}(b)$;
\item ($\widehat{S}$-invariance) for each $\hat{s} \in \widehat{S}$, and each $b \in \mathcal{B}$ for which $\hat{s} \cdot b$ is defined, $\hat{s} \cdot \mathcal{E}(b) = \mathcal{E}(\hat{s} \cdot b)$. 
\end{enumerate}
Let $v \in \mathcal{PV}$; we write $v = \{ b_{1}, \ldots, b_{m} \}$, where $b_{1}$, $\ldots$, $b_{m} \in \mathcal{B}$. We say that $v'$ is a result of \emph{$\mathcal{E}$-expansion from $v$} if there are $v'_{i} \in \mathcal{E}(b_{i})$, for $i = 1, \ldots, m$, such that 
\[ v' = \bigcup_{i=1}^{m} v'_{i}. \]
We say that $\mathcal{E}$ is an \emph{expansion scheme} if
\begin{enumerate}
\item[(4)] for every $[f,D] \in \mathcal{B}$ and every $w_{1}, w_{2} \in \mathcal{E}([f,D])$ such that $w_{1} \leq w_{2}$, $w_{2}$ is the result of $\mathcal{E}$-expansion from $w_{1}$.
\end{enumerate}
\end{definition}

\begin{definition} (the complex $\Delta^{\mathcal{E}}$) \label{definition:DE}
Let $\mathcal{E}$ be an expansion scheme. We let $\Delta^{\mathcal{E}}$ be the subcomplex of the directed set construction made up of \emph{$\mathcal{E}$-simplices}; i.e., simplices
\[ v_{1} < v_{2} < \ldots < v_{m} \]
such that the vertices $v_{j}$ ($j \in \{ 2, \ldots, m \}$) are obtained from $v_{1}$ by 
$\mathcal{E}$-expansion.
\end{definition}

\begin{definition} \label{definition:intervalsubcomplexes} (interval subcomplexes; relative ascending links) 
Let $v'$ and $v''$ be pseudovertices such that $v' \leq v''$ (in the sense of the expansion partial order; see Remark \ref{remark:pseudovertex} and Definition \ref{definition:expansion}). We let $\Delta^{\mathcal{E}}_{[v',v'']}$ denote the set of all $\mathcal{E}$-simplices 
\[ v_{1} < \ldots < v_{m} \]
such that $v' \leq v_{1} < v_{m} \leq v''$. This is the \emph{interval subcomplex} determined by $v'$, $v''$, and $\mathcal{E}$. 

The \emph{ascending link of $v'$ relative to $v''$} is the link of $v'$ in the complex $\Delta^{\mathcal{E}}_{[v', v'']}$.
\end{definition}

\begin{definition} \label{definition:nconnectedE} ($n$-connected expansion schemes)
Let $\mathcal{E}$ be an expansion scheme. We say that $\mathcal{E}$ is \emph{$n$-connected}
if, for each $b \in \mathcal{B}$ and each pseudovertex $v$ such that $\{ b \} \leq v$, the ascending link of $\{ b \}$ relative to $v$ is $(n-1)$-connected.
\end{definition}

\begin{theorem} \label{theorem:nconnected} ($n$-connectedness of $\Delta^{\mathcal{E}}$)
If $\mathcal{E}$ is an $n$-connected expansion scheme, then the complex $\Delta^{\mathcal{E}}$ is $n$-connected.
\end{theorem}

\begin{remark} ($n$-connectedness of complexes determined by pseudovertices) The conclusion of Theorem \ref{theorem:nconnected} carries over to complexes determined by pseudovertices in a component-by-component fashion; i.e., each connected component is $n$-connected.
\end{remark}

\begin{example} \label{example:thompsonsV} (the case of Thompson's group $V$)
We consider a basic example of an expansion scheme. Let 
\[ \mathcal{C} = \prod_{n=1}^{\infty} \{ 0, 1\} \]
denote the usual binary Cantor set. The elements of $\mathcal{C}$ are infinite binary strings. We let $\mathcal{C}^{fin}$ denote the set of all finite binary strings. For each $\omega \in \mathcal{C}^{fin}$, we let 
$D_{\omega}$ denote the set of all infinite binary strings that
begin with the prefix $\omega$. For $\omega_{1}, \omega_{2} \in \mathcal{C}^{fin}$, the transformation
$\sigma_{\omega_{1},\omega_{2}} : D_{\omega_{1}}  \rightarrow D_{\omega_{2}}$ removes the prefix $\omega_{1}$ from the input and adds the prefix 
$\omega_{2}$ in its place. We let
\[ S_{V} = \{ \sigma_{\omega_{1},\omega_{2}} \mid \omega_{1}, \omega_{2} \in \mathcal{C}^{fin} \} \cup \{ 0 \}, \]
where $0$ represents the empty function. The set $S_{V}$ is an inverse monoid under composition. The associated set of domains $\mathcal{D}^{+}_{S_{V}}$ consists of all of the 
sets $D_{\omega}$, where $\omega \in \mathcal{C}^{fin}$.

The set of all bijections $\gamma: \mathcal{C} \rightarrow \mathcal{C}$ that are locally determined by $S_{V}$ make up a group, which we denote by $V$. This is Thompson's well-known group $V$, as described in \cite{CFP}. We define an $S_{V}$ structure as follows. For each 
pair $(D_{\omega_{1}}, D_{\omega_{2}})$, we define
\[ \mathbb{S}(D_{\omega_{1}}, D_{\omega_{2}}) = \{ \sigma_{\omega_{1}, \omega_{2}} \}. \]
The verification that this assignment does, indeed, define an $S_{V}$-structure is routine. 

We now define an expansion scheme $\mathcal{E}$. For each $[f,D_{\omega}]$, let
\[ \mathcal{E}([f,D_{\omega}]) = \{ \{ [f,D_{\omega}] \}, \{ [f,D_{\omega 0}], [f,D_{\omega 1}] \} \}. \]
Thus, the set $\mathcal{E}([f,D_{\omega}])$ consists of two pseudovertices: the base pseudovertex $\{ [f,D_{\omega}] \}$, and the pseudovertex obtained by performing the simplest possible expansion at $[f,D_{\omega}]$, namely the expansion that subdivides $D_{\omega}$ into left and right halves ($D_{\omega0}$ and $D_{\omega1}$, respectively). It is straightforward to check that the assignment $\mathcal{E}$ satisfies the conditions of Definition \ref{definition:expansionscheme}. 

A simplex in $\Delta^{\mathcal{E}}$ is a chain 
\[ v_{1} < v_{2} < \ldots < v_{m}, \]
where $v_{1} = \{ [f_{1},D_{\omega_{1}}], \ldots, [f_{n},D_{\omega_{n}}] \}$, 
and each vertex $v_{j}$ $(2 \leq j \leq m)$ can be obtained from $v_{1}$ by, for a given $i \in \{ 1, \ldots, n \}$, either replacing $[f_{i}, D_{\omega_{i}}]$ with its left and right halves (in the sense described above), or leaving $[f_{i}, D_{\omega_{i}}]$ unchanged.
 
It is also straightforward to check that the expansion scheme $\mathcal{E}$ is $n$-connected, for all $n$. Indeed, let $b \in \mathcal{B}$ and let $\{ b \} < v$.   There is a unique $\mathcal{E}$-expansion from $b$, and thus the $1$-simplex connecting $\{ b \}$ to $\{ b_{\ell}, b_{r} \}$ is the star of $\{ b \}$
in $\Delta^{\mathcal{E}}_{[\{ b \}, v]}$. The ascending link of $\{ b \}$ relative to $v$ is therefore always a point.
 It follows from Theorem \ref{theorem:nconnected} that $\Delta^{\mathcal{E}}$ is contractible. 
\end{example}

\begin{example} \label{example:THEexpansionscheme}
We now return to the main examples of this paper. Let $S = S_{2}$, $S_{3}$, $S'_{2}$, or $S'_{3}$, and let the $S$-structure $\mathbb{S}$ be defined as in Definition \ref{definition:structuresets}. 
We will define expansion schemes $\mathcal{E}_{i}$ and $\mathcal{E}'_{i}$, for $i=2,3$. In order to do so, we must first introduce some useful notation.

For each $k \in \mathbb{Z}$, we let $u_{k}$ denote the vertex that corresponds to the subdivision tree consisting of a single caret in which the root is numbered $k$. (The correspondence in question is that of Theorem \ref{theorem:thelink}.) Thus, 
\[ u_{k} = \{ [C^{k}A,I], [C^{k}B,I] \}, \]
where $C = C_{2}$ or $C_{3}$, depending on the semigroup $S$ in question. If $S = S_{2}$ or $S'_{2}$, we let $u_{k-\frac{1}{2}}$ be the vertex corresponding to the subdivision tree consisting of two carets: a top caret (with root labelled $k$), and a second caret, attached to the left child of the root, labelled $0$. Thus,
\[ u_{k-\frac{1}{2}} = \{ [C^{k}AA,I], [C^{k}AB,I], [C^{k}B,I] \}. \]
If $S = S_{3}$ or $S'_{3}$, then $u_{k-\frac{1}{2}}$ is the vertex represented by the subdivision tree consisting of a top caret (with root labelled $k$), and a complete depth-two binary tree, attached at the left child of the root, in which each node is labelled $0$. Thus,
\[ u_{k-\frac{1}{2}} = \{ [C^{k}AAA,I], [C^{k}AAB,I], [C^{k}ABA,I], [C^{k}ABB,I], [C^{k}B,I] \}. \]

With the above conventions, we can set
\[ \mathcal{E}_{i}([id_{I},I]) = \mathcal{E}'_{i}([id_{I},I]) = \{ \{[id_{I},I] \} \} \cup \{ u_{k/2} \mid k \in \mathbb{Z} \}, \]
for $i=2$ or $3$. By extending $\widehat{S}$-equivariantly, we arrive at a definition of
$\mathcal{E}_{i}(b) = \mathcal{E}'_{i}(b)$, for any $b = [f,I]$ and for $i=2$ or $3$, where $I$ is contained in the domain of $f \in \widehat{S}$:
\[ \mathcal{E}_{i}([f,I]) = \mathcal{E}'_{i}([f,I]) = \{ \{[f,I] \} \} \cup \{ f \cdot u_{k/2} \mid k \in \mathbb{Z} \}. \]
The well-definedness of this assignment is easy to check.

It is straightforward to check that $u_{k} \leq u_{k-\frac{1}{2}}$ and $u_{k-1} \leq u_{k-\frac{1}{2}}$, for each integer $k$. (The first inequality is clear; the second inequality follows directly after applying an elementary equivalence.) Moreover, no two of the vertices $u_{k_{1}}$ and $u_{k_{2}}$ are comparable and no two of the vertices $u_{k_{1} - \frac{1}{2}}$ and $u_{k_{2}-\frac{1}{2}}$ are comparable (if $k_{1} \neq k_{2}$). It follows that the simplicial realizations of $\mathcal{E}_{i}(b)$ and $\mathcal{E}'_{i}(b)$ take the form indicated in Figure \ref{figure:E}. 

\begin{figure} [!h]
\begin{center}
\begin{tikzpicture}

%%%%%%%%%%%2-simplices of figure
\filldraw[lightgray] (4,1.5) -- (0,3) -- (1,3.25) -- cycle;
\filldraw[lightgray] (4,1.5) -- (2,3) -- (3,3.25) -- cycle;
\filldraw[lightgray] (4,1.5) -- (4,3) -- (5,3.25) -- cycle;
\filldraw[lightgray] (4,1.5) -- (6,3) -- (7,3.25) -- cycle;
\filldraw[lightgray] (4,1.5) -- (1,3.25) -- (2,3) -- cycle;
\filldraw[lightgray] (4,1.5) -- (3,3.25) -- (4,3) -- cycle;
\filldraw[lightgray] (4,1.5) -- (5,3.25) -- (6,3) -- cycle;
\filldraw[lightgray] (4,1.5) -- (7,3.25) -- (8,3) -- cycle;

%%%%%%%%%%%%grey lines: (4,1.5) is the bottommost point
\draw[gray, thick] (4,1.5) -- (0,3);
\draw[gray, thick] (0,3) -- (1,3.25);
\draw[gray, thick] (4,1.5) -- (1,3.25);
\draw[gray, thick] (1,3.25) -- (2,3);
\draw[gray, thick] (4,1.5) -- (2,3);
\draw[gray, thick] (2,3) -- (3,3.25);
\draw[gray, thick] (4,1.5) -- (3,3.25);
\draw[gray, thick] (3,3.25) -- (4,3);
\draw[gray, thick] (4,1.5) -- (4,3);
\draw[gray, thick] (4,3) -- (5,3.25);
\draw[gray, thick] (4,1.5) -- (5,3.25);
\draw[gray, thick] (5,3.25) -- (6,3);
\draw[gray, thick] (4,1.5) -- (6,3);
\draw[gray, thick] (6,3) -- (7,3.25);
\draw[gray, thick] (4,1.5) -- (7,3.25);
\draw[gray, thick] (7,3.25) -- (8,3);
\draw[gray, thick] (4,1.5) -- (8,3);

%%%%%%%%%%labels of vertices
\node at (0,3.25){-2};
\node at (2,3.25){-1};
\node at (4,3.25){0};
\node at (6,3.25){1};
\node at (8,3.25){2};
\node at (4,1.25){b};

%%%%%%%%%%%%%%%%%%%left ellipsis
\filldraw [gray] (0,2.25) circle (1.5pt);
\filldraw [gray] (-.5,2.25) circle (1.5pt);
\filldraw [gray] (-1,2.25) circle (1.5pt);
%%%%%%%%%%%%%%

%%%%%%%%%right ellipsis
\filldraw [gray] (8,2.25) circle (1.5pt);
\filldraw [gray] (8.5,2.25) circle (1.5pt);
\filldraw [gray] (9,2.25) circle (1.5pt);
%%%%%%%%%%%%%%
\end{tikzpicture}
\end{center}
\caption{Above we have depicted the simplicial complex $\mathcal{E}(b)$ associated to $b=[id_{I},I]$ by the expansion schemes $\mathcal{E}_{i}$ and $\mathcal{E}'_{i}$. An integer $k$ refers to the vertex $u_{k}$.}
\label{figure:E}
\end{figure}
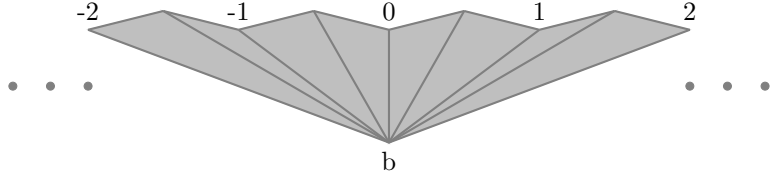

We recall that $[f,\omega I] = [f \omega, I]$ when $\omega \in \{ A, B \}^{\ast}$ (see the beginning of Remark \ref{remark:explicitexpansion}). It follows that the description of $\mathcal{E}_{i}$ is complete for $i=2$ and $3$. 

We next define
\[ \mathcal{E}'_{i}([f,[m,\infty)) = \{ \{[f,[m,\infty)] \}, \{ [f,[m,m+1)], [f,[m+1,\infty)] \} \}. \]
This completes the definition of $\mathcal{E}'_{i}$, for $i=2$ and $3$. 

Observe that, if $\{ [f,[m,\infty)] \} \leq v$, then
\[ v = \{ [f,D] \mid D \in \mathcal{P} \}, \]
where $\mathcal{P} \subseteq \mathcal{D}^{+}_{gen}$ is a finite partition of $[m,\infty)$ into generating domains (see Remark \ref{remark:explicitexpansion}). The partition $\mathcal{P}$ is necessarily a refinement of $\{ [m,m+1), [m+1,\infty) \}$. It follows that the ascending link of $\{ [f,[m,\infty)] \}$  relative to $v$ is always a point, and thus contractible. Thus, the expansion scheme
$\mathcal{E}'_{i}$ is $n$-connected for a given $n$ if and only if
$\mathcal{E}_{i}$ is. We may therefore concentrate on $\mathcal{E}_{i}$ in what follows.
\end{example}

%%%%%%%%%%%%%%%%%%%%%%%%%%%%%
\section{Finite complete presentations of semigroups} \label{section:FP}
%%%%%%%%%%%%%%%%%%%%%%%%%%%%

In order to understand equivalence between subdivision trees, we will need a full analysis of the monoids $M_{2}$ and $M_{3}$, which, by definition, are generated by the linear fractional transformations that we have denoted $A$, $B$, $C_{i}$, and $c_{i}$ (for $i=2,3$). 

In this section, in contrast to our usual practice, the letters $A$, $B$, and $C$ will be used as formal symbols. We will define abstract monoid presentations 
$\mathcal{P}_{i}$ and $\widehat{\mathcal{P}}_{i}$ ($i=2,3$), with the ultimate goal of proving that the monoid $M(\mathcal{P}_{i})$ defined by $\mathcal{P}_{i}$ is isomorphic to $M_{i}$. (The monoids $M(\widehat{\mathcal{P}}_{i})$ represent a necessary intermediate device.)   

The arguments in this section parallel those from Section 5 of \cite{LM}. 

%%%%%%%%%%%%%%%%%%%%%%%
\subsection{Monoid presentations and string-rewriting systems}
%%%%%%%%%%%%%%%%%%%%%%%

\begin{definition} \label{definition:monoid}
(Monoid presentations)
Let $\Sigma$ be a set. The \emph{free monoid on $\Sigma$}, denoted $\Sigma^{\ast}$ is the set of all positive (possibly empty) words in $\Sigma$, with the operation of concatenation. The empty word is  denoted by $1$. We write $\omega_{1} \equiv \omega_{2}$ if $\omega_{1}, \omega_{2} \in \Sigma^{\ast}$ are identical as words.

Let $\mathcal{R}$ be a set of ordered pairs $(r_{1},r_{2}) \in \Sigma^{\ast} \times \Sigma^{\ast}$. We view such a pair as an equality between words in $\Sigma^{\ast}$, writing $r_{1} = r_{2}$ if either 
$(r_{1},r_{2})$ or $(r_{2}, r_{1})$ is in $\mathcal{R}$. The pair $\mathcal{P} = \langle \Sigma \mid \mathcal{R} \rangle$ is called a \emph{monoid presentation}; the set $\mathcal{R}$ is the set of \emph{relations}. These relations determine an equivalence relation on $\Sigma^{\ast}$ in the following way. If $\omega_{1}, \omega_{2} \in \Sigma^{\ast}$, then we write $\omega_{1} \approx \omega_{2}$ if $\omega_{1} \equiv
\alpha r_{1} \beta$ and $\omega_{2} \equiv \alpha r_{2} \beta$ for some words $\alpha, \beta \in \Sigma^{\ast}$, and $(r_{1},r_{2}) \in \mathcal{R}$. The symmetric, transitive closure of $\approx$, denoted $\sim$, is an equivalence relation on $\Sigma^{\ast}$. We sometimes denote the equivalence class of a word $\omega$ by $[\omega]$.

The concatenation operation on $\Sigma^{\ast}$ determines a well-defined associative operation on the set of equivalence classes $\Sigma^{\ast} / \sim$. We let $M(\mathcal{P})$ denote the set of these equivalence classes, with the operation induced by concatenation. The set $M(\mathcal{P})$ is a monoid with respect to this operation, called the \emph{monoid determined by $\mathcal{P}$}.
\end{definition}

\begin{definition} (Rewrite Systems; String-Rewriting Systems)
A \emph{rewrite system} is a directed graph $\Gamma$. We allow loops and multiple edges. If $v_{1}$ and $v_{2}$ are vertices of $\Gamma$, we write $v_{1} \rightarrow v_{2}$ if there is a directed edge issuing from $v_{1}$ and terminating at $v_{2}$. We write $v_{1} \dot{\rightarrow} v_{2}$ if there is a directed edge path from $v_{1}$ to $v_{2}$. Equivalently, $\dot{\rightarrow}$ is the transitive closure of $\rightarrow$. 

Let $\mathcal{P} = \langle \Sigma \mid \mathcal{R} \rangle$ be a monoid presentation. We define a rewrite system $\Gamma(\mathcal{P})$  as follows. The  vertex set of $\Gamma(\mathcal{P})$ is $\Sigma^{\ast}$. There is a directed edge from $\omega_{1}$ to $\omega_{2}$ if 
$\omega_{1} \equiv \alpha r_{1} \beta$ and $\omega_{2} \equiv \alpha r_{2} \beta$, where $\alpha, \beta \in \Sigma^{\ast}$ and $(r_{1}, r_{2}) \in \mathcal{R}$. The directed graph $\Gamma(\mathcal{P})$ is called the \emph{string-rewriting system} associated to the monoid presentation $\mathcal{P}$.  
\end{definition}

\begin{remark}
Let $\dot{\leftrightarrow}$ denote the symmetric, transitive closure of $\rightarrow$. Thus, $\dot{\leftrightarrow}$ is an equivalence relation on the vertices of $\Gamma(\mathcal{P})$.
The above definitions easily show that the relation $\dot{\leftrightarrow}$ coincides with $\sim$ on $\Sigma^{\ast}$. In other words, equivalence classes of words in $\Sigma^{\ast}$ modulo $\approx$ are in one-to-one correspondence with (undirected) path components of $\Gamma(\mathcal{P})$. 

In view of this close identification between the monoid $M(\mathcal{P})$ and the string-rewriting system
$\Gamma(\mathcal{P})$, it causes no harm to write $r_{1} \rightarrow r_{2}$ for a relation $(r_{1}, r_{2}) \in \mathcal{R}$.
\end{remark}

\begin{definition}
(terminating; confluent; locally confluent; reduced)
A rewrite system $\Gamma$  is \emph{terminating} if every sequence of vertices $v_{1} \rightarrow v_{2} \rightarrow v_{3} \rightarrow \ldots$ is finite. We say $\Gamma$ is \emph{confluent} if whenever 
$v_{1} \dot{\rightarrow} w_{1}$ and $v_{1} \dot{\rightarrow} w_{2}$, there is some $v_{2}$ such that $w_{1} \dot{\rightarrow} v_{2}$ and $w_{2} \dot{\rightarrow} v_{2}$. 
   We say $\Gamma$ is \emph{locally confluent} if whenever 
$v_{1} \rightarrow w_{1}$ and $v_{1} \rightarrow w_{2}$, there is some $v_{2}$ such that $w_{1} \dot{\rightarrow} v_{2}$ and $w_{2} \dot{\rightarrow} v_{2}$. 

A vertex $v$ of $\Gamma$ is called \emph{reduced} if there is no directed edge issuing from $v$.

A rewrite system is \emph{complete} if it is terminating and confluent. We say that a monoid presentation
$\mathcal{P}$ is complete if the associated string-rewriting system $\Gamma(\mathcal{P})$ is complete.
\end{definition}

\begin{theorem} \cite{Newman}
If the rewrite system $\Gamma$ is terminating and locally confluent, then $\Gamma$ is confluent. \qed
\end{theorem} 

\begin{corollary} \label{corollary:unique} (unique reduced forms) 
If $\mathcal{P} = \langle \Sigma \mid \mathcal{R} \rangle$ is a complete monoid presentation, then any connected component of $\Gamma(\mathcal{P})$ contains a unique reduced word.

In particular, any word $\omega \in \Sigma^{\ast}$ is equivalent to a unique reduced word modulo $\approx$. \qed  
\end{corollary}

%%%%%%%%%%%%%%%%%%%%%%
\subsection{Basic Definitions of the Rewrite Systems}
%%%%%%%%%%%%%%%%%%%%%%

\begin{definition} \label{definition:gamma2,3}
We define monoid presentations, $\mathcal{P}_{2}$, 
$\widehat{\mathcal{P}}_{2}$, $\mathcal{P}_{3}$, and $\widehat{\mathcal{P}}_{3}$, as follows.

\begin{enumerate}
\item $\mathcal{P}_{2} = \langle A, B, C_{2}, c_{2} \mid \mathcal{R}_{2} \rangle$, where $\mathcal{R}_{2}$ consists of the relations appearing in the top box of Table \ref{table:R2}.
\item $\widehat{\mathcal{P}}_{2} = \langle A, B, C_{2}, a, b, c_{2}, 0 \mid 
\widehat{\mathcal{R}}_{2} \rangle$, where $\widehat{\mathcal{R}}_{2}$ consists of the relations that appear in Table \ref{table:R2}.
\item $\mathcal{P}_{3} = \langle A, B, C_{3}, c_{3} \mid \mathcal{R}_{3} \rangle$, where $\mathcal{R}_{3}$ consists of the relations appearing in the top box of Table \ref{table:R3}.
\item $\widehat{\mathcal{P}}_{3} = \langle A, B, C_{3}, a, b, c_{3}, 0 \mid 
\widehat{\mathcal{R}}_{3} \rangle$, where $\widehat{\mathcal{R}}_{3}$ consists of the relations that appear in Table \ref{table:R3}.
\end{enumerate}
We note that the occurrences of ``$X$" in the tables represent arbitrary generators.
\end{definition}

\begin{table} [!h]
\begin{center}
\begin{tabular} {|lll|}
\hline
 && \\
$C_{2}AA \rightarrow AC_{2}$ & $C_{2}B \rightarrow BBC_{2}$ &
$C_{2}AB \rightarrow BAc_{2}$ \\ &&\\  $c_{2}A \rightarrow AAc_{2}$ &
$c_{2}BB \rightarrow Bc_{2}$ & 
$c_{2}BA \rightarrow ABC_{2}$ \\ &&\\
$C_{2} c_{2} \rightarrow 1$  &
$c_{2} C_{2} \rightarrow 1$ & \\ && \\
\hline 
&& \\
$C_{2}aa \rightarrow aC_{2}$ & $C_{2}b \rightarrow bbC_{2}$ &
$C_{2}ab \rightarrow bac_{2}$ \\ && \\
 $c_{2}a \rightarrow aac_{2}$ &
$c_{2}bb \rightarrow bc_{2}$ & 
$bA \rightarrow 0$ \\ && \\
$aB \rightarrow 0$ &
$bB \rightarrow 1$  &
$aA \rightarrow 1$ \\ && \\ $0X \rightarrow 0$   &
$X0 \rightarrow 0$ & \\ && \\
\hline  
\end{tabular}
\end{center}
\vspace{5pt}
\caption{The relations of $\widehat{\mathcal{R}}_{2}$. The relations in the top box are $\mathcal{R}_{2}$. The ``$X$'' stands for any of the generators.}
\label{table:R2}
\end{table}

\begin{table}[!t]
\begin{center}
\begin{tabular}{|lll|}
\hline
&& \\
$C_{3}AAA \rightarrow AC_{3}$ & 
$C_{3}AAB \rightarrow BAAc_{3}$ &
$C_{3}ABA \rightarrow BABC_{3}$ \\ && \\
$C_{3}ABB \rightarrow BBAc_{3}$ &
$C_{3}B \rightarrow BBBC_{3}$ &
$c_{3}A \rightarrow AAAc_{3}$ \\ && \\
$c_{3}BAA \rightarrow AABC_{3}$ &
$c_{3}BAB \rightarrow ABAc_{3}$ &
$c_{3}BBA \rightarrow ABBC_{3}$  \\ && \\
$c_{3}BBB \rightarrow Bc_{3}$ &
$C_{3} c_{3} \rightarrow 1$ &
$c_{3} C_{3} \rightarrow 1$ \\ && \\
\hline
 && \\
$c_{3}a \rightarrow aaac_{3}$ & 
$C_{3}aab \rightarrow baac_{3}$ &
$c_{3}bab \rightarrow abac_{3}$ \\ && \\
$C_{3}abb \rightarrow bbac_{3}$ &
$c_{3}bbb \rightarrow bc_{3}$ &
$C_{3}aaa \rightarrow aC_{3}$ \\ && \\
$c_{3}baa \rightarrow aabC_{3}$ &  
$C_{3}aba \rightarrow babC_{3}$ &
$c_{3}bba \rightarrow abbC_{3}$ \\ && \\
$C_{3}b \rightarrow bbbC_{3}$ &
$C_{3} c_{3} \rightarrow 1$ &
$c_{3} C_{3} \rightarrow 1$ \\ && \\
$0X \rightarrow 0$ & $X0 \rightarrow 0$ &
$bB \rightarrow 1$ \\ && \\
$aA \rightarrow 1$ &
$aB \rightarrow 0$ & $bA \rightarrow 0$ \\ &&\\ 
\hline
\end{tabular}
\vspace{5pt}
\caption{The relations of $\widehat{\mathcal{R}}_{3}$. The relations in the top box are $\mathcal{R}_{3}$.}
\label{table:R3}
\end{center}
\end{table}

\begin{proposition}
The string-rewriting systems  $\Gamma(\mathcal{P}_{2})$, $\Gamma(\mathcal{P}_{3})$,
$\Gamma(\widehat{\mathcal{P}}_{2})$, and $\Gamma(\widehat{\mathcal{P}}_{3})$ are locally confluent and terminating. In particular, each word $\omega$ has a unique reduced form $r(\omega)$.
\end{proposition}

\begin{proof}
The proof is straightforward, but somewhat tedious. We will sketch the basic ideas.

One important observation is that, for a given word $\omega$ in the generators, occurrences in $\omega$ of the lefthand sides of the relations are generally pairwise disjoint. In such cases, local confluence is trivial to check.

A typical case of overlap involves relations such as $cC \rightarrow 1$ and $Cc \rightarrow 1$. A case of overlap such as $cCc$ is easily seen to reduce to $c$, no matter whether $cC \rightarrow 1$ is applied before $Cc \rightarrow 1$, or the reverse. This implies that the local confluence condition is satisfied in such a case, as well.

The ``terminating" condition follows from the fact that every relation either ``moves" $C$ closer to the end of the word (possibly changing the occurrence of $C$ to $c$ in the process), or shortens the word.  
\end{proof}

\begin{definition}
(The monoids $M_{2}$ and $M_{3}$) Let $T_{A} : [0,1) \rightarrow [0,1/2)$ be the transformation defined by the rule
\[ T_{A}(x) = \frac{x}{x+1}. \]
Thus, $T_{A}$ is exactly the transformation denoted $A$ in Section \ref{section:family}. We similarly define $T_{B}$, $T_{C_{2}}$, and $T_{C_{3}}$ as $B$, $C_{2}$, and $C_{3}$ were defined in Section \ref{section:family}.

For $i=2,3$, we let $M_{i}$ be the monoid generated by the transformations 
$T_{A}$, $T_{B}$, $T_{C_{i}}$; i.e., the collection of functions generated by these transformations under the operation of concatenation.

For $i=2,3$, we let $\widehat{M}_{i}$ be the inverse monoid generated by the transformations
$T_{A}$, $T_{B}$, $T_{C_{i}}$; i.e., the collection of functions generated by these transformations and their inverses, under the operation of concatenation.
\end{definition} 

\begin{definition}
(The maps $\pi_{2}$, $\hat{\pi}_{2}$, $\pi_{3}$, $\hat{\pi}_{3}$, and $\pi$)
For each $X \in \{ A, B, C_{2}, C_{3} \}$, we set $\pi(X) = T_{X}$. We extend this map to 
the lower-case letters $a$, $b$, $c_{2}$, $c_{3}$ by sending each to the relevant inverses; i.e.,
$\pi(a) = T^{-1}_{A}$, $\pi(b) = T^{-1}_{B}$, etc. We define $\pi(0)$ to be the empty function (with empty domain and codomain).  

For $i=2,3$, we define monoid homomorphisms $\pi_{i}: M(\mathcal{P}_{i}) \rightarrow M_{i}$ by letting $\pi_{i}$ agree with $\pi$ on the relevant generating sets. We similarly 
define monoid homomorphisms $\hat{\pi}_{i}: M(\widehat{\mathcal{P}}_{i}) \rightarrow \widehat{M}_{i}$ (for $i=2,3$) by letting $\hat{\pi}_{i}$ agree with $\pi$ on the relevant generating sets.

We note that the homomorphisms $\hat{\pi}_{i}$ restrict to $\pi_{i}$, for $i=2,3$ (and, indeed, $M(\mathcal{P}_{i})$ is a submonoid of $M(\widehat{\mathcal{P}}_{i})$, as the latter remark implies).
\end{definition}

\begin{remark} \label{remark:surjective}
The proof that $\pi_{i}$ and $\hat{\pi}_{i}$ ($i=2,3$) are monoid homomorphisms depends on showing that the defining relations of $M(\mathcal{P}_{i})$ and $M(\widehat{\mathcal{P}}_{i})$ are satisfied by their images in $M_{i}$ and $\widehat{M}_{i}$. This verification is routine, and is left to the reader.

It is clear that the maps $\pi_{i}$ and $\hat{\pi}$ are surjective. 
\end{remark}

\begin{remark}
Note that, although $T^{-1}_{B} T_{B} = 1$ (where ``$1$'' here denotes the identity function on $[0,1]$), $T_{B}T^{-1}_{B} = id_{[1/2, 1)} \neq 1$. Similarly, $T_{A}T^{-1}_{A} = id_{[0,1/2)}$.

In a similar vein, the transformation $T_{B} T^{-1}_{A} \neq T^{-1}_{A} T_{B}$, etc.
\end{remark}

\begin{remark}
The rewrite rules $1X \rightarrow X$ and $X1 \rightarrow X$ are implicit in the definitions of $M(\mathcal{P}_{i})$ and $M(\widehat{\mathcal{P}}_{i})$. It is technically unnecessary to include them, since ``$1$" is simply notation for the empty string.
\end{remark}

%%%%%%%%%%%%%%%%%%%%%%%5
\subsection{The ``no potential cancellations" condition}
%%%%%%%%%%%%%%%%%%%%%%%%%%

Throughout this subsection, we will write ``$C$" in place of $C_{2}$ or $C_{3}$, and similarly write ``$c$" in place of $c_{2}$ or $c_{3}$.

\begin{definition} ($C$-tracks)
A subword $\omega'$ of $\omega \in \{ A, B, C, a, b, c, 0 \}^{\ast}$ is called a \emph{$C$-track}
if 
\begin{enumerate}
\item $\omega'$ contains at most one occurrence of $C$ or $c$ (not both);
\item any occurrence of $C$ or $c$ is at the beginning of the word $\omega'$;
\item $\omega'$ is a maximal subword with respect to properties (1) and (2).
\end{enumerate}
\end{definition}

\begin{remark} Any word $\omega \in \{ A, B, C, a, b, c, 0 \}^{\ast}$ has a unique decomposition
\[ \omega \equiv \omega_{1} \ldots \omega_{n} \]
as a product of $C$-tracks.
\end{remark}

\begin{definition} (\cite{LM}, Definition 5.7) (advancing an occurrence of $C$ or $c$) (
To \emph{advance} an occurrence of $C$ (or $c$) is to apply one of the relations from Definition \ref{definition:gamma2,3}, other than those of the form $Cc \rightarrow 1$,
$cC \rightarrow 1$, $0X \rightarrow 0$, and $X0 \rightarrow 0$, to a subword containing that occurrence of $C$ or $c$. 
\end{definition}

\begin{definition} (\cite{LM}, Definition 5.8) (no potential cancellations)
Assume that $\omega \in \{ A, B, C, c \}^{\ast}$. Let 
\[ \omega \equiv \omega_{1} \ldots \omega_{n} \]
be the unique decomposition into $C$-tracks. We say that $\omega$ has \emph{no potential cancellations} if the words
\[ r(\omega_{1})\omega_{2} \ldots \omega_{n}, \quad  \omega_{1}r(\omega_{2})\ldots \omega_{n}, \quad \ldots \quad \omega_{1}\omega_{2}\ldots r(\omega_{n}) \]
contain no occurrences of $cC$ or $Cc$ as subwords.
\end{definition}

\begin{proposition} (\cite{LM}, Lemma 5.9) \label{proposition:npc}
If $\omega \in \{A, B, C, c \}^{\ast}$ has no potential cancellations and $\omega'$ is the result of advancing a $c$ or $C$ exactly once, then $\omega'$ has no potential cancellations.
\end{proposition}

\begin{table}[!b] 
\begin{center}
\begin{tabular}{|c|c|c|c|}
\hline
$\ell \rightarrow r$ & $x \equiv C$ \text{ or }$c$ & $x \equiv CA$ & $x \equiv cB$ \\
\hline \hline
$CAA \rightarrow AC$ & $CAC \dot{\rightarrow} CAC$ & $CAAC \dot{\rightarrow} ACC$ & $cBAC \dot{\rightarrow} ABCC$ \\
\hline
$CB \rightarrow BBC$ & $CBBC \dot{\rightarrow} B^{4}C^{2}$ & $CAB^{2}C \dot{\rightarrow} BAcBC$ &
$cB^{3}C \dot{\rightarrow} cB^{3}C$ \\
\hline
$CAB \rightarrow BAc$ & $CBAc \dot{\rightarrow} BBCCAc$ &  $CABAc \dot{\rightarrow} BAAAcc$ &  $cBBAc \dot{\rightarrow}  BAAcc$ \\
\hline
$cA \rightarrow AAc$ & $cAAc \dot{\rightarrow} A^{4}cc $ & $CA^{3}c \dot{\rightarrow} ACAc $ & $cBAAc \dot{\rightarrow} ABCAc $ \\
\hline 
$cBB \rightarrow Bc$ & 
$cBc \dot{\rightarrow} cBc$ & $CABc \dot{\rightarrow} BAcc$ & $cBBc \dot{\rightarrow} Bcc$ \\
\hline
$cBA \rightarrow ABC$ & $cABC \dot{\rightarrow} AAcBC$ &  $CAABC \dot{\rightarrow} AB^{2}C^{2}$ &  $cBABC \dot{\rightarrow} AB^{3}C^{2}$ \\
\hline
\end{tabular}

\end{center}
\vspace{5pt}
\caption{The proof of Proposition \ref{proposition:npc} in the case of $M(\mathcal{P}_{2})$.}
\label{table:proof1}
\end{table}

\begin{table}[!t] 
\begin{center}
\begin{tabular}{|c|c|c|}
\hline
$\ell \rightarrow r$ & $x \in \{c, C, CA, CAA \}$ & $x \in \{CAB, cB, cBA, cBB\}$ \\
\hline \hline
$CA^{3} \rightarrow AC$ & $CAC \dot{\rightarrow} CAC$ & $CABAC \dot{\rightarrow} BABCC$ \\
 &  $CAAC \dot{\rightarrow} CAAC$ &  $cBAC \dot{\rightarrow} cBAC$ \\
 &  $CAAAC \dot{\rightarrow} ACC$ & $cBAAC \dot{\rightarrow} AABCC$ \\
 &  & $cBBAC \dot{\rightarrow} ABBCC$ \\
\hline
$CABA \rightarrow BABC$ & $CBABC \dot{\rightarrow} B^{3}CABC$ & $CABBABC \dot{\rightarrow} B^{2}A^{4}cBC$ \\
 &  $CABABC \dot{\rightarrow} BAB^{4}C^{2}$ &  $cBBABC \dot{\rightarrow} AB^{5}C^{2}$ \\
 &  $CAABABC \dot{\rightarrow} BA^{5}cBC$ & $cBABABC \dot{\rightarrow} ABA^{4}cBC$ \\
 &  & $cB^{3}ABC \dot{\rightarrow} BA^{3}cBC$ \\
\hline
$CAAB \rightarrow BAAc$ & $CBAAc \dot{\rightarrow} B^{3}CAAc$ & $CABBAAc \dot{\rightarrow} A^{4}B^{2}CAc$ \\
 &  $CABAAc \dot{\rightarrow} BABCAc$ &  $cBBAAc \dot{\rightarrow} ABBCAc$ \\
 &  $CAABAAc \dot{\rightarrow} BA^{8}cc$ & $cBABAAc \dot{\rightarrow} ABA^{7}cc$ \\
 &  & $cB^{3}AAc \dot{\rightarrow} BA^{6}cc$ \\
\hline
$CABB \rightarrow BBAc$ & $CBBAc \dot{\rightarrow} B^{6}CAc$ & $CABBBAc \dot{\rightarrow} BBAcBAc$ \\
 &  $CABBAc \dot{\rightarrow} B^{2}A^{4}cc$ &  $cBBBAc \dot{\rightarrow} BA^{3}cc$ \\
 &  $CAABBAc \dot{\rightarrow} BA^{2}cBAc$ & $cBABBAc \dot{\rightarrow} ABAcBAc$ \\
 &  & $cB^{4}Ac \dot{\rightarrow} BcBAc$ \\
\hline
$CB \rightarrow BBBC$ & $CBBBC \dot{\rightarrow} B^{9}C^{2}$ & $CABBBBC \dot{\rightarrow} BBAcBBC$ \\
 &  $CABBBC \dot{\rightarrow} BBAcBC$ &  $cBBBBC \dot{\rightarrow} BcBC$ \\
 &  $CAABBBC \dot{\rightarrow} BA^{2}cBBC$ & $cBABBBC \dot{\rightarrow} ABAcBBC$ \\
 &  & $cB^{5}C \dot{\rightarrow} BcBBC$ \\
\hline
$cA \rightarrow AAAc$ & $cAAAc \dot{\rightarrow} A^{9}c^{2}$ & $CABAAAc \dot{\rightarrow} BABCAAc$ \\
& $CA^{4}c \dot{\rightarrow} ACAc$ & $cBAAAc \dot{\rightarrow} AABCAc$ \\
& $CA^{5}c \dot{\rightarrow} ACAAc$ & $cBAAAAc \dot{\rightarrow} AABCAAc$ \\
& & $cBBAAAc \dot{\rightarrow} ABBCAAc$ \\
\hline
$cBAA \rightarrow AABC$ & $cAABC \dot{\rightarrow} A^{6}cBC$ & $CABAABC \dot{\rightarrow} BABCABC$ \\
& $CAAABC \dot{\rightarrow} AB^{3}C^{2}$ & $cBAABC \dot{\rightarrow} A^{2}B^{4}C^{2}$ \\
& $CA^{4}BC \dot{\rightarrow} ACABC$ & $cBA^{3}BC \dot{\rightarrow} AABCABC$ \\
& & $cBBAABC \dot{\rightarrow} ABBCABC$ \\
\hline
$cBAB \rightarrow ABAc$ & $cABAc \dot{\rightarrow} A^{3}cBAc$ & $CABABAc \dot{\rightarrow} BAB^{4}CAc$ \\
& $CAABAc \dot{\rightarrow} BA^{5}cc$ & $cBABAc \dot{\rightarrow} ABA^{4}cc$\\
& $CAAABAc \dot{\rightarrow} AB^{3}CAc$ & $cBAABAc \dot{\rightarrow} A^{2}B^{4}CAc$ \\
& & $cBBABAc \dot{\rightarrow} AB^{5}CAc$ \\
\hline
$cBBA \rightarrow ABBC$ & $cABBC \dot{\rightarrow} A^{3}cBBC$ & $CABABBC \dot{\rightarrow} BAB^{7}CC$ \\
& $CAABBC \dot{\rightarrow} BAAcBC$ & $cBABBC \dot{\rightarrow} ABAcBC$ \\
& $CA^{3}BBC \dot{\rightarrow} AB^{6}CC$ & $cBAABBC \dot{\rightarrow} AAB^{7}CC$ \\
& & $cBBABBC \dot{\rightarrow} AB^{8}C^{2}$ \\
\hline
$cB^{3} \rightarrow Bc$ & $cBc \dot{\rightarrow} cBc$ & $CABBc \dot{\rightarrow} BBAcc$ \\
& $CABc \dot{\rightarrow} CABc$ & $cBBc \dot{\rightarrow} cBBc$ \\
& $CAABc \dot{\rightarrow} BAAcc$ & $cBABc \dot{\rightarrow} ABAcc$ \\
& & $cBBBc \dot{\rightarrow} Bcc$ \\
\hline
\end{tabular}

\end{center}
\vspace{5pt}
\caption{The proof of Proposition \ref{proposition:npc} in the case of $M(\mathcal{P}_{3})$.}
\label{table:proof2}
\end{table}

\begin{proof}
Let $\omega \equiv \omega_{1} \omega_{2} \ldots \omega_{n}$, where the right side of the equation is the unique decomposition of $\omega$ into $C$-tracks. Suppose that $\omega'$ is the result of advancing an occurrence of $C$ (or $c$) exactly once; suppose that the advanced  occurrence of $C$ or $c$ appears in $\omega_{i}$, and let $\ell \rightarrow r$ be the relation that advances this $C$ or $c$. Thus $\omega_{i} \equiv \ell \beta$ for some word 
$\beta$. Let 
\[ \omega' \equiv \omega'_{1} \omega'_{2} \ldots \omega'_{n} \]
be the unique decomposition of $\omega'$ into $C$-tracks. It follows directly that
$\omega'_{i-1} \omega'_{i} \equiv \omega_{i-1} r \beta$, while $\omega'_{j} \equiv \omega_{j}$
if $j \in \{ 1, \ldots, n \} - \{ i-1, i \}$. Note that the subword $\omega_{i-1}r$ consists of the $C$-track of the $(i-1)$st occurrence of a $C$ (or $c$) in $\omega'$, followed by a $C$ (or $c$); note also that the only chance of an occurrence of $Cc$ or $cC$ in the words $\omega'_{1} \ldots r(\omega'_{j-1}) \omega'_{j} \ldots \omega'_{n}$ might occur when $j=i$.

To prove the proposition, it therefore suffices to prove that, during the reduction of
the subword $\omega_{i-1}r$, no occurrence of $Cc$ or $cC$ can arise. Note that 
$\omega_{i-1}r$ begins and ends with occurrences of $C$ (or $c$), while all intermediate letters are $A$ or $B$. Thus, after reducing $\omega_{i-1}$, it suffices to show that an occurrence of $Cc$ or $cC$ cannot arise in (further) reducing $r(\omega_{i-1})r$. Finally, we note that the reduced word $r(\omega_{i-1})$ ends in a reduced word $x$ that begins with a $C$ or $c$. There are only finitely many possibilities for $x$: indeed, $x \in \{ c, C, CA, cB \}$
in the case of $M(\mathcal{P}_{2})$, while $x \in \{ c, C, CA, CAA, CAB, cB, cBA, cBB \}$ in the case of $M(\mathcal{P}_{3})$. Furthermore, in either case, one of the cases $x \equiv C$ or $x \equiv c$ can be ruled out, since $x \ell$ contains no occurrence of $Cc$ or $cC$ by hypothesis.

Thus, in summary, it suffices to show that, for each rewriting rule $\ell \rightarrow r$, 
no occurrence of $cC$ or $Cc$ can appear when reducing the word $xr$, where $x$ runs over the above possibilities and further satisfies the condition that $x \ell$ itself contains neither $cC$ nor $Cc$. 

The relevant calculations are summarized in Tables \ref{table:proof1} and \ref{table:proof2}.
\end{proof}

\begin{definition}
(negative-to-positive words)
Let $\omega$ be a word in the alphabet $\{ A, B, C, a, b, c, 0, 1 \}$. We say that $\omega$ is \emph{negative-to-positive} if all occurrences of $a$ and $b$ (if any) occur before any occurrence of either $A$ or $B$. 
\end{definition}

\begin{remark}
Note that, if $\omega \neq 1, 0$ is a negative-to-positive word containing no occurrences of 
$bB$ or $aA$ or $1$, then each $C$-track in $\omega$ is a word in either $\{ a, b, C, c \}$
or $\{ A, B, C, c \}$. We call a $C$-track in the former alphabet \emph{negative}, while a $C$-track in the latter alphabet is \emph{positive}. Of course, negative $C$-tracks occur before positive ones in a negative-to-positive word.
\end{remark}

\begin{definition} (no potential cancellations in negative-to-positive words)
Let $\omega \in \{ A, B, C, a, b, c \}^{\ast}$. Assume that the reduced form of $\omega$ is not $0$, 
and that $\omega$ also contains no occurrences of $bB$ or $aA$.

 Let 
\[ \omega \equiv \omega_{1} \ldots \omega_{n} \]
be the unique decomposition into $C$-tracks. We say that $\omega$ has \emph{no potential cancellations} if the words
\[ r(\omega_{1})\omega_{2}\ldots \omega_{n}, \quad \omega_{1}r(\omega_{2})\ldots \omega_{n}, \quad \ldots \quad \omega_{1}\omega_{2}\ldots r(\omega_{n}) \]
contain no occurrences of $cC$ or $Cc$. 
\end{definition}

\begin{proposition} \label{proposition:npcinntp}
Assume that
\begin{enumerate}
\item $\omega$ is negative-to-positive;
\item $\omega$ has no subword of the form $aA$ or $bB$;
\item the reduced form of $\omega$ is not $0$;
\end{enumerate}
Let $\omega'$ be the result of advancing a $c$ or $C$ exactly once, and then removing any occurrence of $aA$ or $bB$, along with any occurrence of ``$1$''. The word $\omega'$ also has no potential cancellations.
\end{proposition}

\begin{proof}
The proof is like that of Proposition \ref{proposition:npc}. We assume that $\omega'$ is the result of advancing the $i$th occurrence of $C$ or $c$ in $\omega$ exactly once. Let 
$\omega_{1} \omega_{2} \ldots \omega_{n}$ be the $C$-track decomposition of $\omega$.
There are three cases: $\omega_{i-1}$ and $\omega_{i}$ are both positive $C$-tracks,
or $\omega_{i-1}$ and $\omega_{i}$ are both negative, or $\omega_{i-1}$ is negative and
$\omega_{i}$ is positive.

We note that the first case (in which both $C$-tracks are positive) is already handled by the proof of Proposition \ref{proposition:npc}. The second case is also handled by the proof of Proposition \ref{proposition:npc}. This follows from the observation that each rewriting rule between words in the alphabet $\{ a, b, C, c \}$ corresponds to a rewriting rule in between words in the alphabet $\{ A, B, C, c \}$. One need only replace a $C$ with $c$ (or the reverse), an $A$ with a $b$, and a $B$ with an $A$. Thus, for instance, the rewriting rule
$CAA \rightarrow AC$ corresponds to $cbb \rightarrow bc$. Using this substitution, we can transform Tables \ref{table:proof1} and \ref{table:proof2} into tables that prove the negative-to-negative case.

It remains to consider the case in which $\omega_{i-1}$ is negative and $\omega_{i}$ is positive. This case still follows the pattern of the proof of Proposition \ref{proposition:npc}, but we need to handle more cases. We note that $x \in \{ c, C, Ca, cb \}$ in the case of $M(\widehat{\mathcal{P}}_{2})$ and $x \in \{ c, C, Ca, Caa, Cab, cb, cba, cbb \}$
in the case of $M(\widehat{\mathcal{P}}_{3})$. We can further reduce to the cases
$x \in \{ Ca, cb \}$ and $x \in \{ Ca, Caa, Cab, cb, cba, cbb \}$ (respectively), since the proof of Proposition \ref{proposition:npc} also handles the cases in which $x \equiv C$ or $x \equiv c$. Furthermore, subwords of the form
$aB$ and $bA$ are ruled out by hypothesis (since $\omega \neq 0$). 

We summarize the necessary calculations in Tables \ref{table:3} and \ref{table:4}.
\end{proof}

\begin{table} [!h] 
\begin{center}
\begin{tabular} {|c|c|c|}
\hline
$CaAC \dot{\rightarrow} CC$ & $CaAAc \dot{\rightarrow} CAc$ & 
$CaABC \dot{\rightarrow} BBCC$ \\ 
\hline
$cbBBC \dot{\rightarrow} cBC$ & $cbBAc \dot{\rightarrow} AAcc$ & 
$cbBc \dot{\rightarrow} cc$  \\
\hline 
\end{tabular}
\end{center}
\vspace{5pt}
\caption{ No potential cancellations -- the negative-to-positive case for $M(\widehat{\mathcal{P}}_{2})$.}
\label{table:3}
\end{table}

\begin{table} [!h] 
\begin{center}
\begin{tabular} {|c|c|c|}
\hline
$ CaAC \dot{\rightarrow} CC$ & $CaAAAc \dot{\rightarrow} CAAc$ & 
$ CaAABC \dot{\rightarrow} CABC$ \\ 
\hline
$ CaABAc \dot{\rightarrow} BBBCAc$ & $CaABBC \dot{\rightarrow} B^{6}C^{2} $ & 
$CaaAC \dot{\rightarrow} CaC$  \\
\hline 
$CaaAAAc \dot{\rightarrow} CAc$ & $CaaAABC \dot{\rightarrow} B^{3}C^{2} $ & 
$ CabBABC \dot{\rightarrow} B^{3}C^{2}$  \\
\hline 
$ CabBAAc \dot{\rightarrow} CAc $ & $CabBc \dot{\rightarrow} Cac$ & 
$cbBABC \dot{\rightarrow} A^{3}cBC$  \\
\hline
$ cbBAAc \dot{\rightarrow} A^{6}c^{2} $ & $ cbBBAc \dot{\rightarrow} cBAc$ & 
$ cbBBBC \dot{\rightarrow} cBBC $  \\
\hline
$ cbBc \dot{\rightarrow} c^{2}  $ & $cbaAC \dot{\rightarrow} cbC $ & 
$ cbaABAc \dot{\rightarrow} A^{3}c^{2}$  \\
\hline
$ cbaABBC \dot{\rightarrow} cBC$ & $cbbBBAc \dot{\rightarrow} A^{3}c^{2} $ & 
$ cbbBBBC \dot{\rightarrow} cBC$  \\
\hline    
$ cbbBc \dot{\rightarrow} cbc$ & & \\
\hline
\end{tabular}
\end{center}
\vspace{5pt}
\caption{ No potential cancellations -- the negative-to-positive case for $M(\widehat{\mathcal{P}}_{3})$.}
\label{table:4}
\end{table}

%%%%%%%%%%%%%%%%%%%%%%%%%%%%%%5
\subsection{Presentations and normal forms for the monoids $M_{2}$ and $M_{3}$}
%%%%%%%%%%%%%%%%%%%%%%%%%%%5

\begin{proposition} (\cite{LM}, Lemma 5.10) \label{proposition:terminalC}
Let $\omega$ be a word in the generators $\{ A, B, C, c \}$. Assume that $\omega$ has no potential cancellations. 

There is a word $\tau \in \{ A, B \}^{\ast}$ such that $r(\omega \tau) \equiv \widehat{\omega}C^{\epsilon}$, where $\widehat{\omega}$ is a word in $\{ A, B \}$ and $\epsilon \geq 0$ is the total exponent 
of $C$ and $c$ in $\omega$.
\end{proposition}

\begin{proof}
The proof is by induction on the (combined) exponent $\epsilon$ of $C$ and $c$ in $\omega$. We note that, due to the ``no potential cancellations" condition, it is not possible to reduce (or, indeed, increase) the exponent $\epsilon$ by applying any of the monoid relations. 

We first consider the case of $M(\mathcal{P}_{2})$; assume $\epsilon = 1$, the case $\epsilon = 0$ being trivial. We note that $r(\omega)$ ends with one of the strings $C$, $c$, $CA$, or $cB$ (and the only occurrences of $C$ or $c$ occur in these strings). In the case of $C$, there is nothing to prove. If $r(\omega)$ ends with $c$, we can let $\tau \equiv BA$ and then reduce the result. If $r(\omega)$ ends with either $CA$ or $cB$, we can let $\tau \equiv A$ and then reduce the result. This proves the base case.

Now let $\epsilon > 1$. We can express $\omega$ as a product $\omega_{1} \omega_{2}$, where 
the total combined exponent of $C$ and $c$ in $\omega_{2}$ is $\epsilon - 1$, and $\omega_{1}$ contains a single occurrence of $C$ or $c$. By induction, we can find $\tau_{1} \in \{ A, B \}^{\ast}$ such that 
$r(\omega_{2}\tau_{1}) \equiv \widehat{\omega}_{2} C^{\epsilon - 1}$, where $\widehat{\omega}_{2} \in \{ A, B \}^{\ast}$. Thus, after reducing the word $\omega_{1} \omega_{2} \tau_{1}$, we obtain 
a word $\omega' \in \{ A, B, C, c \}^{\ast}$ that ends with $C^{\epsilon}$, $CAC^{\epsilon -1}$, or $cBC^{\epsilon -1}$. [Note that the case $cC^{\epsilon - 1}$ is ruled out by the ``no potential cancellations" hypothesis.] In the first case, we are finished; set $\tau_{2} \equiv 1$. If $\omega'$ ends with $CAC^{\epsilon - 1}$
or $cBC^{\epsilon -1}$, we can set $\tau_{2} \equiv A^{2^{\epsilon-1}}$. After reducing the word
$\omega' \tau_{2}$, we have a string of the required form, so the required $\tau$ is $\tau_{1} \tau_{2}$.

Now we consider $M(\mathcal{P}_{3})$. Let $\omega \in \{ A, B, C, c \}$ and define $\epsilon$ as before. We first consider the case $\epsilon =1$. The word $r(\omega)$ ends with $c$, $C$, $CA$, $CAB$, $CAA$, $cB$, $cBA$, or $cBB$. If $r(\omega)$ ends with $c$, we can let $\tau \equiv BBA$ and apply the relation $cBBA \rightarrow ABBC$. If $r(\omega)$ ends with $C$, there is nothing to prove ($\tau \equiv 1$). In the remaining cases, we let $\tau \equiv AA, A, A, AA, A,$ or $A$ (respectively). 

Now suppose $\epsilon > 1$. We can write $\omega$ as the product $\omega_{1}\omega_{2}$, where
the total combined exponent of $C$ and $c$ in $\omega_{2}$ is $\epsilon -1$, and $\omega_{1}$ contains a single occurrence of either $C$ or $c$. Proceeding as in the case of $M(\mathcal{P}_{2})$, we can right multiply by some $\tau_{1} \in \{ A, B \}^{\ast}$ and reduce to arrive at a word $\omega'$ that ends with one of the following strings:  
$C^{\epsilon}$, $CAC^{\epsilon - 1}$, $CABC^{\epsilon - 1}$, $CAAC^{\epsilon - 1}$, $cBC^{\epsilon - 1}$, $cBAC^{\epsilon -1}$, $cBBC^{\epsilon-1}$. In the first case, there is nothing to prove; let $\tau_{2} \equiv 1$. In the remaining cases, we multiply by $\tau_{2} \equiv A^{2 \cdot 3^{\epsilon - 1}}$, $A^{3^{\epsilon - 1}}$, $A^{3^{\epsilon - 1}}$,
$A^{2 \cdot 3^{\epsilon - 1}}$, $A^{3^{\epsilon - 1}}$, or $A^{3^{\epsilon - 1}}$, respectively. Thus, the required $\tau$ is $\tau_{1}\tau_{2}$.
\end{proof}

\begin{proposition} \label{proposition:ntpadvancing}
Let $\omega \in \{ A, B, C, a, b, c \}^{\ast}$ be a negative-to-positive word with no potential cancellations. Assume that there is no $\tau \in \{ A, B \}^{\ast}$ such that  $r(\omega \tau) \equiv 0$. 

There is some $\tau' \in \{ A, B \}^{\ast}$ such that $r(\omega \tau') \equiv \widehat{\omega} C^{\epsilon}$,
where $\widehat{\omega} \in \{ A, B \}^{\ast}$ and $\epsilon$ is the total combined exponent of 
$C$ and $c$ in $\omega$.
\end{proposition}

\begin{proof}
We prove this by induction on the sum $k$ of the combined exponents of $a$ and $b$ in $\omega$. The case $k = 0$ is handled by Proposition \ref{proposition:terminalC}.

Let $\omega \in \{ A, B, C, a, b, c, \}^{\ast}$, let $k$ be defined as above, and suppose that the proposition is known to be true for smaller $k$. We can write $\omega \equiv \omega_{1}\omega_{2}$, where $\omega_{1}$ involves no occurrences of $A$ or $B$, and $\omega_{2}$ involves no occurrences of $a$ or $b$. We may further assume that $\omega_{1}$ ends with an occurrence of $a$ or $b$, since any occurrence of $C$ or $c$ may be subsumed by $\omega_{2}$. 

By Proposition \ref{proposition:terminalC}, we can find a word $\tau_{1} \in \{ A, B \}^{\ast}$ such that 
$r(\omega_{2} \tau_{1}) \equiv \widehat{\omega}C^{\epsilon_{2}}$, where $\epsilon_{2}$ is the total exponent sum of $C$ and $c$ in $\omega_{2}$ and $\widehat{\omega} \in \{ A, B \}^{\ast}$. If $\widehat{\omega}$ is not the empty word, then it must be that the initial letter of $\widehat{\omega}$ cancels with the terminal letter of $\omega_{1}$ in 
$\omega_{1} \widehat{\omega}C^{\epsilon_{2}}$. (This is because occurrences of $aB$ and $bA$ cannot arise, by the hypothesis that $r(\omega \tau)$ is never $0$.) Now we can call the inductive hypothesis, to find $\tau_{2}$ such that $r(\omega_{1} \widehat{\omega}C^{\epsilon_{2}} \tau_{2}) \equiv \widetilde{\omega}C^{\epsilon_{1} + \epsilon_{2}}$, where $\widetilde{\omega} \in \{ A, B \}^{\ast}$ and $\epsilon_{1}$ is the total exponent of $c$ and $C$ in $\omega_{1}$. This completes the induction, under the assumption that $\widehat{\omega}$ is not the empty word.

If $\widehat{\omega} \equiv 1$, we simply multiply by a suitable word $\tau_{3/2}$: either $A^{2^{\epsilon_{2}}}$ or
$A^{3^{\epsilon_{2}}}$ (depending on whether we are considering $M(\mathcal{P}_{2})$ or $M(\mathcal{P}_{3})$). After reducing, we find that $r(\omega_{2} \tau_{1} \tau_{3/2}) \equiv \widehat{\omega}'C^{\epsilon_{2}}$, where $\widehat{\omega}'$ is non-empty. This reduces us to the previous case, completing the induction and the proof.
\end{proof}

\begin{proposition} \label{proposition:normalforms} 
(Normal Forms in $M(\mathcal{P}_{i})$)
The reduced words modulo the presentation $\mathcal{P}_{2}$ take the form
\[ \omega_{1} \omega_{2} \omega_{3}, \]
where $\omega_{1} \in \{ A, B \}^{\ast}$, $\omega_{2} \in \{ C^{n}A, c^{m}B \mid m,n \in \mathbb{N} \}^{\ast}$,
and $\omega_{3} \in \{ C, c \}^{\ast}$.

The reduced words modulo the presentation $\mathcal{P}_{3}$ take the form
\[ \omega_{1} \omega_{2} \omega_{3}, \]
where $\omega_{1} \in \{ A, B \}^{\ast}$, $\omega_{2} \in \{ C^{n_{1}}A, C^{n_{2}}AA,
C^{n_{3}}AB, c^{n_{4}}B, c^{n_{5}}BA, c^{n_{6}}BB \mid n_{i} \in \mathbb{N} \}^{\ast}$, and 
$\omega_{3} \in \{ C, c \}^{\ast}$.
\end{proposition}

\begin{proof}
It is clear that the words in question are reduced. Thus, the main point is to show that every word in the generators can be reduced to a word of the given type. This is easily done by induction on the length of the word.
\end{proof}

\begin{theorem} \label{theorem:presentation} (Monoid Presentations for $M_{2}$ and $M_{3}$)
The monoid homomorphisms $\pi_{i} : M(\mathcal{P}_{i}) \rightarrow M_{i}$ are isomorphisms, for
$i=2,3$. 

In particular, $\mathcal{P}_{i}$ is a presentation for $M_{i}$, for $i=2,3$.
\end{theorem}

\begin{proof}
In view of Remark \ref{remark:surjective}, it suffices to show that $\pi_{i}$ is injective, for $i=2,3$. We suppose, for a contradiction, that $\pi_{2}$ is not injective. Let 
\[ S' = \{ \{ \omega_{1}, \omega_{2} \} \mid \omega_{1} \not \equiv \omega_{2}; \pi_{2}(\omega_{1}) = \pi_{2}(\omega_{2}); \omega_{1} \text{ and } \omega_{2} \text{ are reduced} \}. \] We let 
\[ S'' = \{ \{ \omega_{1}, \omega_{2} \} \mid \{ \omega_{1}, \omega_{2} \} \in S'; \, \omega_{1}, \omega_{2} \text{ begin with different letters} \}. \]  
We note that $S'$ is non-empty by hypothesis, and it follows easily that $S''$ is also non-empty. (It suffices to cancel the maximal common prefix of the words $\omega_{1}, \omega_{2}$, where $\{ \omega_{1}, \omega_{2} \} \in S'$.) Next, we note that if $\{ \omega_{1}, \omega_{2} \} \in S''$, then one of $\omega_{1}$ or $\omega_{2}$ begins with $C$ or $c$, or is trivial. (The case in which $\omega_{1}$ begins with ``$A$" and $\omega_{2}$ begins with ``$B$'' (or the reverse) can be ruled out, since $\pi_{2}(\omega_{1})$ cannot be equal to $\pi_{2}(\omega_{2})$ under these conditions.)

Consider $\{ \omega_{1}, \omega_{2} \} \in S''$ such that the total exponent of $C$ and $c$ in $\omega_{1}$ and $\omega_{2}$ is a minimum. We assume that $\omega_{1}$ begins with either $C$ or $c$. (The case in which $\omega_{1}$ is trivial is easier, and will be handled in the course of the more difficult argument.) It follows that $\omega_{1} \equiv \omega'_{1} C^{k}$, where $\omega'_{1} \in \{ C^{n}A, c^{m}B \mid n,m \in \mathbb{N} \}^{\ast}$ and $k \in \mathbb{Z}$, by Proposition \ref{proposition:normalforms}. Indeed, we can assume that $k=0$, for if $k \neq 0$, then we simply multiply both $\omega_{1}$ and $\omega_{2}$ on the right by $c^{k}$. The word $\omega'_{2} :\equiv \omega_{2}c^{k}$ is then necessarily reduced, by the hypothesis
that the total exponent of $C$ and $c$ in $\omega_{1}$ and $\omega_{2}$ is a minimum in $S''$. We can then replace the pair $\{ \omega_{1}, \omega_{2} \}$ by
$\{ \omega'_{1}, \omega'_{2} \}$, where the latter is still in $S''$.

Thus, we can assume that $\omega_{1} \in \{ C^{n}A, c^{m}B \mid n,m \in \mathbb{N} \}^{\ast}$, $\{ \omega_{1}, \omega_{2} \} \in S''$, and the total combined exponent of $C$ and $c$ in $\omega_{1}$ and $\omega_{2}$ is a minimum within $S''$. We claim that the words $\omega_{1}^{-1}$ and $\omega_{2}$ have no potential cancellations. This is obvious in the case of $\omega_{2}$, since it is reduced.   In $\omega_{1}^{-1}$, every occurrence of $a$ is followed by $c$, every occurrence of $b$ is followed by $C$, and there are no occurrences of  $A$, $B$, $Cc$, or $cC$. Now note that no occurrences of $Cc$ or $cC$ can occur when reducing subwords of the form $C^{\pm 1}ac$
or $C^{\pm 1}bC$; from this it follows that $\omega^{-1}_{1}$ has no potential cancellations.  

Next we claim that $\omega_{1}^{-1} \omega_{2}$ has no potential cancellations. (Here the claim is obvious if $\omega_{1}$ is the trivial word; thus, the argument from this point is the general case.) Indeed, the only $C$-track of $\omega_{1}^{-1} \omega_{2}$ that could cause a problem is the one that begins with the terminal letter ($c$ or $C$) of $\omega_{1}^{-1}$. Assume that 
$\omega_{1} \equiv C \hat{\omega}_{1}$ (without loss of generality), and suppose that $\omega^{-1}_{1} \omega_{2}$ has a potential cancellation. Replacing $\{ \omega_{1}, \omega_{2} \}$ by 
$\{ \hat{\omega}_{1}, r(c \omega_{2}) \}$, we find (possibly after cancelling common prefixes) that the latter is a pair in $S''$ of smaller total exponent in $C$ and $c$. This contradicts the choice of $\{ \omega_{1}, \omega_{2} \}$, which proves the claim. 

Since $\pi_{2}(\omega_{1}) = \pi_{2}(\omega_{2})$, $\widehat{\pi}_{2}(\omega_{1}^{-1}\omega_{2}) = \widehat{\pi}(1) = id_{[0,1)}$.
In particular, this means that $\omega_{1}^{-1}\omega_{2}$ satisfies the hypotheses of Proposition \ref{proposition:ntpadvancing}. We can therefore find a word $\tau \in \{ A, B \}^{\ast}$ such that
$r(\omega_{1}^{-1} \omega_{2} \tau) \equiv \widehat{\omega} C^{k}$, where $k \geq 0$ is the total combined exponent of $C$ and $c$ in the word $\omega_{1}^{-1}\omega_{2}$ and $\widehat{\omega} \in \{ A, B \}^{\ast}$. Thus, we have $\pi_{2}(\widehat{\omega}C^{k}) = \pi_{2}(\tau)$. We now cancel 
the maximal common prefix of $\widehat{\omega}C^{k}$ and $\tau$. We continue to denote the resulting strings $\widehat{\omega}C^{k}$ and $\tau$, respectively, but now either $\widehat{\omega}$
or $\tau$ is trivial. 

If $\widehat{\omega}$ is trivial, but $\tau$ is not, then $\pi_{2}(\widehat{\omega}C^{k}) = C^{k}$,
while $\pi_{2}(\tau)$ is a transformation whose range is a proper subinterval of $[0,1)$. This is a contradiction. If $\tau$ is trivial, but not $\widehat{\omega}$, we find that $\pi_{2}(\tau)$ has the image
$[0,1)$, but $\pi_{2}(\widehat{\omega}C^{k})$ does not, which is also a contradiction. Finally, if both $\widehat{\omega}$ and $\tau$ are trivial, we find that $\pi_{2}(C^{k}) = id_{[0,1)}$, which is possible only if $k=0$. The latter implies that $\omega_{1}$ is the trivial word, $\omega_{2} \in \{ A, B \}^{\ast}$, and $\omega_{2}$ is not the trivial word. This leads us to conclude that $\pi_{2}(\omega_{2}) = id_{[0,1)}$, which is impossible since the image of $\pi_{2}(\omega_{2})$ is a proper subinterval of $[0,1)$. Thus, $\pi_{2}$ is injective.

The case of $\pi_{3}$ is similar.  
\end{proof}

\begin{remark} (presentations for certain submonoids of $\mathrm{Isom}(\mathbb{H}^{2})$)
Let $\widetilde{A}$ denote the transformation of the projective line $\mathbb{P}_{1}$ ($= \partial \mathbb{H}^{2}$) that agrees with $A$ (as defined in Definition \ref{definition:S}) on the $[0,1)$. Similarly define $\widetilde{B}$, $\widetilde{C}_{2}$, and so forth.
The transformations $\widetilde{A}$, $\widetilde{B}$, $\widetilde{C}_{2}$, $\widetilde{C}_{3}$, and their inverses may equivalently be considered isometries of $\mathbb{H}^{2}$. Let $\widetilde{S}_{i} = \{ \widetilde{A}, \widetilde{B}, \widetilde{C}_{i}, \widetilde{c}_{i} \}^{\ast}$, for $i=2,3$. 
There are obvious homomorphisms $\phi_{i}: \widetilde{S}_{i} \rightarrow M_{i}$ for $i=2,3$. It is just as clear that these homomorphisms are surjective. If $\phi_{i}(\alpha) = \phi_{i}(\beta)$, 
but $\alpha \neq \beta$, then $\alpha$ and $\beta$ are transformations of $\partial \mathbb{H}^{2}$ that agree on $[0,1)$. This is impossible, however, since any isometry of $\mathbb{H}^{2}$ is determined by its effect on any three boundary points. Thus, $\phi_{i}$ is injective, for $i=2,3$.

It follows from all of this that $\widetilde{S}_{2}$ and $\widetilde{S}_{3}$ admit the same presentations and normal forms as do $M_{2}$ and $M_{3}$.
\end{remark}

%%%%%%%%%%%%%%%%%%%%
\section{An intermediate value theorem for the expansion scheme $\mathcal{E}_{i}$} \label{section:IVT}
%%%%%%%%%%%%%%%%%%%%%%

In this section, we will argue that $N(T)$ is always a set of consecutive integers, if $T$ is a non-trivial subdivision tree. This will be the main ingredient to our proof that the expansion schemes $\mathcal{E}_{i}$ and $\mathcal{E}'_{i}$ are $n$-connected, for all $n$. The proof of the latter fact will be assembled in Section \ref{section:F-infty}.

The essential idea, that of a ``sufficiently expanded subdivision tree", is again drawn from \cite{LM} (Definition 5.5).

%%%%%%%%%%%%%%%%5
\subsection{The case of $M_{2}$}
%%%%%%%%%%%%%%%%%

We remind the reader that ``node" means ``interior node" (see Definition \ref{definition:subdivtree}). 

\begin{definition} \label{definition:fullyexpanded} (sufficiently expanded subdivision trees)
A subdivision tree is \emph{sufficiently expanded} if there is a directed arc $p$ in $T$ from the root $\epsilon$ to a leaf $\ell$ such that
\begin{enumerate}
\item each non-root node on the arc has a non-zero label;
\item if $p$ passes through a non-root node $v$ and the label of $v$ is positive, then $p$ also passes through the left child of $v$;
\item if $p$ passes through a non-root node $v$ and the label of $v$ is negative, then $p$ also passes through the right child of $v$.
\end{enumerate}
A sufficiently expanded subdivision tree is \emph{left sufficiently expanded} (respectively, \emph{right sufficiently expanded}) if some directed arc $p$ as described above passes through the left (respectively, right) child of the root.
\end{definition}

\begin{lemma} \label{lemma:lferfe}
Let $T$ be a subdivision tree.
\begin{enumerate}
\item If $n(T) = k$ and $k-1 \not \in N(T)$, then there is $T' \approx T$ such that $n(T') = k$ and $T'$ is left sufficiently expanded;
\item if $n(T) = k$ and $k+1 \not \in N(T)$, then there is $T' \approx T$ such that $n(T') = k$ and $T'$ is right sufficiently expanded.
\end{enumerate}
\end{lemma}

\begin{proof}
We prove both parts simultaneously by induction on the number of carets in the subdivision tree $T$. The induction begins trivially, since a subdivision tree with a single caret is necessarily both left and right sufficiently expanded.

Now we consider a subdivision tree $T$, and assume that the lemma is true of all subdivision trees containing fewer carets. We will argue for (1); the argument proving (2) is similar. Thus, we let $n(T) = k$ and assume that $k-1 \not \in N(T)$.
 We note that $0 \not \in N(T_{\ell})$; otherwise (up to equivalence) $T$ takes the form in Figure \ref{figure:0inN}.
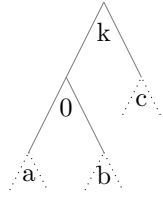
\begin{figure}[!h]
\begin{center}
\begin{tikzpicture}
%figure
\draw[gray] (0,1) -- (1,3);
\draw[gray] (.5,2) -- (1,1);
\draw[gray] (1,3) -- (1.5,2);
\draw[dotted] (0,1) -- (.25, .5);
\draw[dotted] (-.25,.5) -- (0,1);
\draw[dotted] (1,1) -- (1.25, .5);
\draw[dotted] (.75,.5) -- (1,1);
\draw[dotted] (1.25,1.5) -- (1.5,2);
\draw[dotted] (1.75,1.5) -- (1.5,2);
\node at (1,2.6){k};
\node at (.5,1.6){0};
\node at (0,.7){a};
\node at (1,.7){b};
\node at (1.5, 1.7){c};
\end{tikzpicture}
\end{center}
\caption{The case in which $0 \in N(T_{\ell})$.}
\label{figure:0inN}
\end{figure}
This allows us to apply an elementary equivalence from Definition \ref{definition:ee}, resulting in a $T' \approx T$ such that $n(T') = k-1$. However, this implies that $k-1 \in N(T)$, a contradiction. Thus, $0 \not \in N(T_{\ell})$, as claimed.

We let $k_{1}$ be either: i) the smallest positive member of $N(T_{\ell})$, or
ii) the largest negative member of $N(T_{\ell})$. We can assume that $n(T_{\ell}) = k_{1}$ (possibly after replacing $T_{\ell}$ with an equivalent tree and applying Lemma \ref{lemma:l-rlemma}).  We note that $k_{1} - 1 \not \in N(T_{\ell})$ (in case i)), or 
$k_{1}+1 \not \in N(T_{\ell})$ (in case ii)); therefore the inductive hypothesis applies, and we conclude that $T_{\ell}$ is equivalent to a left sufficiently expanded tree $T'_{\ell}$ (in case i)), or to a right sufficiently expanded tree $T'_{\ell}$ (in case ii)). In either case, we replace $T_{\ell}$ by $T'_{\ell}$. We let $T'$ denote the result of replacing $T_{\ell}$ by $T'_{\ell}$ in the tree $T$. We note that $n(T') = n(T) = k$ and $T' \approx T$ by Lemma \ref{lemma:l-rlemma}.

Assume that we are in case i); case ii) is similar. Since $T'_{\ell}$ is left sufficiently expanded, there is a path $p'$ from the root of $T'_{\ell}$ to a leaf of $T'_{\ell}$ satisfying the properties in Definition \ref{definition:fullyexpanded}, such that $p'$ also passes through the left child of the root of $T'_{\ell}$. Let $p$ be the concatenation of $e$ and $p'$, where $e$ is the edge connecting the root of $T'$ to its left child, the root of $T'_{\ell}$. The path $p$ satisfies all of the properties from Definition \ref{definition:fullyexpanded} and passes through the left child of the root in $T'$, so $T'$ is left sufficiently expanded. This completes the induction.
\end{proof}

\begin{proposition} \label{proposition:minmaxfe}
Let $T$ be a non-trivial subdivision tree.
\begin{enumerate}
\item If $T$ is left sufficiently expanded and $T'$ satisfies $n(T') < n(T)$ then
$T' \not \approx T$.
\item If $T$ is right sufficiently expanded and $T'$ satisfies $n(T') > n(T)$ then 
$T' \not \approx T$.
\end{enumerate}
\end{proposition}

\begin{proof}
We first prove (1). Assume that $T$ is left sufficiently expanded and $T'$ is such that $n(T') < n(T)$ and $T' \approx T$. After letting a suitable power of $C$ act at the roots of $T$ and $T'$, we can assume that $n(T) > 0$ and $n(T')=0$. Since $T$ is left sufficiently expanded, there is a directed arc $p$ from the root of $T$ to a leaf $\ell$ satisfying the conditions of Definition \ref{definition:fullyexpanded}; the label of $\ell$ is a reduced word $\omega$. There is a leaf $\ell'$ of $T'$ that corresponds to $\ell$; let $\omega'$ be the label of $\ell'$. We have $\omega = \omega' C^{k}$, for some $k \in \mathbb{Z}$, by Proposition \ref{proposition:equalityofleaves}. After reducing, we find 
\[ \omega \equiv r(\omega' C^{k}). \]
However, these words cannot be equal letter-by-letter, since $\omega$ necessarily begins with an occurrence of $C$, but $\omega'C^{k}$ (and, thus, $r(\omega'C^{k})$) begins with either $A$ or $B$. This is a contradiction.
    
The proof of (2) is similar. One can reduce to the case in which $n(T) < 0$ and $n(T')=0$, and then argue that $\omega$ begins with a $c$, while $r(\omega'C^{k})$ begins with either $A$ or $B$.
\end{proof}

\begin{proposition} \label{proposition:0notinM2}
Let $T$ be a non-trivial subdivision tree.
\begin{enumerate}
\item If $0 \not \in N(T_{\ell})$, then $n(T) = \mathrm{min}(N(T))$.
\item If $0 \not \in N(T_{r})$, then $n(T) = \mathrm{max}(N(T))$.
\end{enumerate}
\end{proposition}

\begin{proof}
We prove (1), the proof of (2) being similar.

We can find a subdivision tree $T'_{\ell} \approx T_{\ell}$ such that 
$n(T'_{\ell})$ is either the smallest positive number in $N(T_{\ell})$ or the largest negative number in $N(T_{\ell})$. In either case, the hypothesis of Lemma \ref{lemma:lferfe} applies, and we can replace $T'_{\ell}$ by $T''_{\ell}$, where $T''_{\ell}$ is sufficiently expanded. We can then find a directed arc $p''$ from the root of $T''_{\ell}$  to a leaf $\ell$, 
where $p''$ satisfies the conditions from Definition \ref{definition:fullyexpanded}. We can then replace the tree $T_{\ell}$ by $T''_{\ell}$ within the tree $T$, to create a new $T'$ such that $n(T') = n(T)$, $T' \approx T$, and 
$T''_{\ell}$ is the left branch of the tree $T'$. Now let $p' = ep''$, where $e$ is the edge connecting the root of $T'$ to the root of $T''_{\ell}$. The path $p'$ satisfies all of the conditions of Definition \ref{definition:fullyexpanded}, and shows that $T'$ is left sufficiently expanded. 

If $T'' \approx T$ and $n(T'') < n(T)$, then $n(T'') < n(T')$ and $T'' \approx T'$, which contradicts 
Proposition \ref{proposition:minmaxfe}(1). It follows that $n(T) = \mathrm{min}(N(T))$.  
\end{proof}

\begin{theorem} \label{theorem:M2exp} (The intermediate value theorem for $M_{2}$)
If $T$ is a non-trivial subdivision tree, then 
\[ N(T) = [m,M] \cap \mathbb{Z}, \]
where $m = \mathrm{min}(N(T))$ and $M = \mathrm{max}(N(T))$.
\end{theorem}

\begin{proof}
Suppose that $k \in (m,M) \cap \mathbb{Z}$ but $k \not \in N(T)$. Assume further that $k$ is the minimal such integer.  

There is a subdivision tree $T' \approx T$ such that $n(T') = k-1$. It must be that $0 \not \in  N(T'_{r})$ (otherwise, we can apply an elementary equivalence to produce a $T'' \approx T'$ such that $n(T'') =k$). Thus, 
$k-1 = n(T') = \mathrm{max}(N(T')) = \mathrm{max}(N(T)) = M$, a contradiction.
\end{proof}

%%%%%%%%%%%%%%%%%%%5
\subsection{The case of $M_{3}$}
%%%%%%%%%%%%%%%%%%%%

\begin{definition} \label{definition:blocking} (blocking trees)
Let $T$ be a subdivision tree. We say that $T$ is a \emph{blocking tree} if either
\begin{enumerate}
\item both of the children of the root of $T$ are nodes, and the three vertices (the root and its children)
are not all labelled by $0$, or 
\item one of these three vertices is a leaf, or $T$ is trivial.
\end{enumerate}
\end{definition}

\begin{definition} \label{definition:fullyexpandedM3} (sufficiently expanded in $M_{3}$)
A subdivision tree $T$ over $M_{3}$ is \emph{sufficiently expanded} if there is a directed arc $p$ from the root $\epsilon$ to some leaf $\ell$ such that 
\begin{enumerate}
\item each non-root node on the arc $p$ is the root of a blocking (sub)tree;
\item if a non-root node $v$ on the arc $p$ has a positive label, then the arc $p$ passes through the left child of $v$;
\item if a non-root node $v$ on the arc $p$ has a negative label, then the arc $p$ passes through the right child of $v$. 
\item if the arc $p$ passes through a non-root node $v$ labelled by ``$0$", then the next node along $p$ (if any) has a non-zero label.
\end{enumerate}
A sufficiently expanded subdivision tree is \emph{left sufficiently expanded} (respectively, \emph{right sufficiently expanded}) if $p$ passes through the left (respectively, the right) child of $\epsilon$.
\end{definition}

\begin{lemma} \label{lemma:lferfeM3}
Let $T$ be a subdivision tree.
\begin{enumerate}
\item If $n(T) = k$ and $k-1 \not \in N(T)$, then there is $T' \approx T$ such that $n(T') = k$ and $T'$ is left sufficiently expanded;
\item if $n(T) = k$ and $k+1 \not \in N(T)$, then there is $T' \approx T$ such that $n(T') = k$ and $T'$ is right sufficiently expanded.
\end{enumerate}
\end{lemma}

\begin{proof}
The proof resembles that of Lemma \ref{lemma:lferfe}. We argue by induction on the number of carets in the subdivision tree $T$. If $T$ consists of a single caret, then it is necessarily both left sufficiently expanded and right sufficiently expanded; thus, the base case is satisfied.

Now consider an arbitrary subdivision tree $T$, and suppose that the lemma has been proved for all subdivision trees having fewer carets. We assume that $T$ satisfies (1); the case of (2) is similar. Since $k-1 \not \in N(T)$, the left branch $T_{\ell}$ of $T$ is a blocking tree. (Indeed, all trees in the equivalence class of $T_{\ell}$ are blocking trees, by Lemma \ref{lemma:l-rlemma}.) There are two possibilities for $T_{\ell}$: either $0 \in N(T_{\ell})$ or
$0 \not \in N(T_{\ell})$. In the latter case, we can proceed essentially as in the proof of Lemma \ref{lemma:lferfe}. We therefore assume that $0 \in N(T_{\ell})$; indeed, we can assume that $n(T_{\ell}) = 0$ without loss of generality. Since all trees in the equivalence class of $T_{\ell}$ are blocking trees, it must be that either $0 \not \in N(T_{\ell \ell})$ or $0 \not \in N(T_{\ell r})$. We assume that $0 \not \in N(T_{\ell \ell})$. Thus, by induction, we can replace $T_{\ell \ell}$ by a subdivision tree $T'_{\ell \ell} \approx T_{\ell \ell}$ such that there is a path $\hat{p}$ from the root of $T'_{\ell \ell}$ to a leaf of $T'_{\ell \ell}$
that satisfies the conditions of Definition \ref{definition:fullyexpandedM3}. We let $T'$ be the result of replacing $T_{\ell \ell}$ with $T'_{\ell \ell}$ in $T$. Now, letting $p_{1}$ denote the path from the root of $T'$ to the root of $T'_{\ell \ell}$ and $p=p_{1} \hat{p}$, the path $p$ shows that $T'$ is left sufficiently expanded, completing the induction.  
\end{proof}

\begin{proposition} \label{proposition:minmaxfeM3}
Let $T$ be a subdivision tree.
\begin{enumerate}
\item If $T$ is left sufficiently expanded and $T'$ satisfies $n(T') < n(T)$ then
$T' \not \approx T$.
\item If $T$ is right sufficiently expanded and $T'$ satisfies $n(T') > n(T)$ then 
$T' \not \approx T$.
\end{enumerate}
\end{proposition}

\begin{proof}
The proof is no different from that of Proposition \ref{proposition:minmaxfe}; again the crucial observation is that the left sufficiently expanded tree $T$ has a leaf $\ell$ whose label is a reduced word, and the corresponding leaf $\ell'$ in $T'$ has a leaf whose label, after reduction, cannot be equivalent to that of $\ell$.
\end{proof}

\begin{proposition} \label{proposition:0notinM3}
Let $T$ be a subdivision tree.
\begin{enumerate}
\item If, whenever $T' \approx T_{\ell}$, $T'$ is a blocking tree, then $n(T) = \mathrm{min}(N(T))$.
\item If, whenever $T' \approx T_{r}$, $T'$ is a blocking tree, then $n(T) = \mathrm{max}(N(T))$.
\end{enumerate}
\end{proposition}

\begin{proof}
We prove (1), the proof of (2) being similar.

The proof of Lemma \ref{lemma:lferfeM3}
allows us to to replace $T$ with an equivalent $\widehat{T}$ such that $\widehat{T}$ is left sufficiently expanded and $n(T) = n(\widehat{T})$.

Let $\widetilde{T} \approx T$. Thus, $\widetilde{T} \approx \widehat{T}$, so $n(\widetilde{T}) \geq n(\widehat{T})$, by Proposition \ref{proposition:minmaxfeM3}. Thus, $n(\widetilde{T}) \geq n(T)$, which implies 
$n(T) = \mathrm{min}(N(T))$.
\end{proof}

\begin{theorem} \label{theorem:M3exp} (The intermediate value theorem for $M_{3}$)
If $T$ is a subdivision tree over $M_{3}$, then 
\[ N(T) = [m,M] \cap \mathbb{Z}. \]
\end{theorem}

\begin{proof}

Suppose that $k \in (m,M) \cap \mathbb{Z}$ but $k \not \in N(T)$. Assume further that $k$ is the minimal such integer.  

There is a subdivision tree $T' \approx T$ such that $n(T') = k-1$. It must be that all subdivision trees that are equivalent to $T'_{r}$ are blocking trees (otherwise, we can first replace $T'_{r}$ by a non-blocking equivalent tree $T''_{r}$, and then apply an elementary equivalence to produce a $T'' \approx T'$ such that $n(T'') =k$). Thus, 
$k-1 = n(T') = \mathrm{max}(N(T')) = \mathrm{max}(N(T)) = M$, a contradiction.
\end{proof}

%%%%%%%%%%%%%%%%%%%%%%%%%%
\section{The proof of the $F_{\infty}$ property} \label{section:F-infty}
%%%%%%%%%%%%%%%%%%%%%%%%%%

In this section, we will complete the proof that the expansion schemes $\mathcal{E}_{i}$ and $\mathcal{E}'_{i}$ ($i=2,3$) 
are $n$-connected for all $n$. This involves assembling a few pieces from Section \ref{section:IVT}.

We will also complete the proofs that the groups $F(S_{i})$, $F(S'_{i})$, $T(S_{i})$, $V(S_{i})$, and $V(S'_{i})$ have type $F_{\infty}$, for $i=2,3$. These proofs are almost entirely like the ones from \cite{FH2}. 

Recall that the approach in this paper departed from that of \cite{FH2} in using a proper subset $\mathcal{D}_{gen}^{+}$ of the domains $\mathcal{D}^{+}$ as the foundation for the original directed set construction. This makes little difference in the final arguments, but rather than simply referring the reader to 
\cite{FH2} (which runs to over sixty pages), we will sketch the necessary changes when it seems appropriate to do so.

%%%%%%%%%%%%%%%%%%%%%%%%%%
\subsection{Brown's finiteness criterion}
%%%%%%%%%%%%%%%%%%%%%%%%%%

Here we briefly recall Brown's finiteness criterion for the reader's convenience.

\begin{theorem} \cite{Brown} \label{theorem:Brown} (Brown's Finiteness Criterion)
Let $X$ be a CW-complex. Let $G$ be a group acting on $X$. If
\begin{enumerate}
\item $X$ is $(n-1)$-connected; 
\item $G$ acts cellularly on $X$, and
\item there is a filtration $X_{1} \subseteq X_{2} \subseteq \ldots \subseteq X_{k} \subseteq \ldots \subseteq X$
such that
\begin{enumerate}
\item $X = \bigcup_{k=1}^{\infty} X_{k}$;
\item $G$ leaves each $X_{k}^{(n)}$ invariant and acts cocompactly on each $X_{k}^{(n)}$; 
\item each $p$-cell stabilizer has type $F_{n-p}$, and
\item for sufficiently large $k$, $X_{k}$ is $(n-1)$-connected,
\end{enumerate}
\end{enumerate}
then $G$ is of type $F_{n}$. \qed
\end{theorem}

%%%%%%%%%%%%%%%%%%%%%%%
\subsection{Contractibility of the complexes $\Delta^{\mathcal{E}_{i}}$ and $\Delta^{\mathcal{E}'_{i}}$}
%%%%%%%%%%%%%%%%%%%%%%%

In this subsection, we will prove that the complexes $\Delta^{\mathcal{E}_{i}}$ and
$\Delta^{\mathcal{E}'_{i}}$ are contractible, for $i=2,3$. This completes a line of argument that was begun at the end of  Section \ref{section:expansion}, and extended through Sections \ref{section:FP} and \ref{section:IVT}.

 Recall that the directed set constructions of the classifying spaces for the groups $F(S)$, $T(S)$, $V(S)$ differed in details (see Section \ref{section:directedset}). We will use the same notation, $\Delta^{\mathcal{E}_{i}}$ and $\Delta^{\mathcal{E}'_{i}}$, to denote the subcomplexes determined by $\mathcal{E}$- (or $\mathcal{E}'$-) expansions in all cases, trusting that the precise meaning will always be clear from the context.  

\begin{theorem} \label{theorem:contractibilityofDelta}
(Contractibility of the complexes $\Delta^{\mathcal{E}_{i}}$ and
$\Delta^{\mathcal{E}'_{i}}$)
The complexes $\Delta^{\mathcal{E}_{i}}$ are contractible, 
for each of the groups $F(S), T(S)$, and $V(S)$ ($S \in \{ S_{2}, S_{3} \}$).

The complexes $\Delta^{\mathcal{E}'_{i}}$ are contractible, for each of the groups
$F(S)$ and $V(S)$ ($S \in \{ S'_{2}, S'_{3} \}$).
\end{theorem}

\begin{proof}
By Theorem \ref{theorem:nconnected}, it suffices to show that the 
expansion schemes $\mathcal{E}_{i}$ and $\mathcal{E}'_{i}$ ($i=2,3$) are $n$-connected for all $n$. By the discussion at the end of Example \ref{example:THEexpansionscheme}, it suffices to consider the expansion scheme $\mathcal{E}$. 

In the groups $F(S_{i}), T(S_{i}), V(S_{i})$ ($i=2,3$), there is just one domain type, namely $[I]$. By equivariance of 
$\mathcal{E}_{i}$, it suffices to show that, whenever $\{ [id_{I},I] \} \leq v$, the ascending link of $\{ [id_{I},I] \}$ relative to $v$ is contractible. 

Since $\{ [id_{I},I] \} \leq v$, $v$ can be represented by a subdivision tree $T$ (by Theorem \ref{theorem:thelink}). By Theorem \ref{theorem:M2exp} or \ref{theorem:M3exp}, $N(T) = [m,M] \cap \mathbb{Z}$, for some integers $m$ and $M$. Thus, $v \geq u_{k}$ for an integer $k$ if and only if $k \in [m,M]$, where $u_{k}$ is as defined in Example \ref{example:THEexpansionscheme}). 

Now we must consider the ``fractional" vertices $u_{k-1/2}$. Note first that, if $k-1/2 \not \in [m,M] \subseteq \mathbb{R}$, then $u_{k-1/2} \not \leq v$, since, if it were, we would conclude that $u_{k-1} \leq v$ 
and $u_{k} \leq v$ (since $u_{k-1}, u_{k} \leq u_{k-1/2}$ in the expansion partial order). This contradicts our hypothesis, since at least one of $k-1$ and $k$ is not in $[m,M]$. Now assume that $k-1/2 \in [m,M]$. It follows from this that $k-1, k \in N(T)$, since $u_{k -1} \leq u_{k-1/2}$ and
$u_{k} \leq u_{k-1/2}$ in the expansion partial order. It follows that $m < k \leq M$ (in the linear order on $\mathbb{R}$). If we are in the case $S = S_{2}$ or $S'_{2}$, then Proposition \ref{proposition:0notinM2} and the inequality $u_{k} \leq v$ show that $0 \in N(T_{\ell})$ (since $k-1 \in N(T)$. Thus, there is some $T'$, $T' \approx T$, such that the root of $T'$ is labelled by $k$ and the left child of the root is labelled by $0$. Since $T'$ represents the vertex $v$, it follows that $u_{k-1/2} \leq v$.
If we are in the case $S = S_{3}$ or $S'_{3}$, then Proposition \ref{proposition:0notinM3} and the inequality $u_{k} \leq v$ show that 
there is some $T'$, $T' \approx T$, such that the root of $T'$ is labelled by $k$ and the left branch of $T'$ is not a blocking tree. It now follows directly that $u_{k-1/2} \leq v$ in this case, as well.

Thus, the ascending link of $\{ [id_{I},I] \}$ relative to $v$ corresponds exactly to the portion of the cellulated line $\ell$ between the vertices $u_{m}$ and $u_{M}$, where $\ell$ is as depicted in Figure \ref{figure:E}. The ascending link is therefore contractible.
%%% I is the interval [0,1) 
\end{proof}

\begin{remark}
We review some of the relevant ideas from \cite{FH2}. 

The basic approach to proving $n$-connectedness is laid out in Lemma 2.6 from
\cite{FH2}. Let $\widehat{\Delta}$ be a simplicial complex whose vertices are a directed set. Let $h$ be a height function defined on the vertex set, such that
$h(v_{1}) < h(v_{2})$ when $v_{1} < v_{2}$. Assume further that $\widehat{\Delta}$  is a subcomplex of the simplicial realization.
Lemma 2.6 says that $\widehat{\Delta}$ is $n$-connected if, for every two vertices
$v_{1}, v_{2}$ in $\widehat{\Delta}$ such that $v_{1} < v_{2}$, the ascending link of $v_{1}$ relative to $v_{2}$ (Definition \ref{definition:intervalsubcomplexes}) is always $(n-1)$-connected. 

Lemma 2.6 applies to our complexes $\Delta^{\mathcal{E}_{i}}$ and $\Delta^{\mathcal{E}'_{i}}$ directly, where the height function $h$ sends a vertex to its cardinality (as in Definition \ref{definition:filtration}). It therefore suffices to show that the relative ascending link is $n$-connected for all $n$. 

We can argue the latter point directly as follows. If 
\[ v_{1} = \{ [f_{1},D_{1}], \ldots, [f_{m},D_{m}] \} \]
is a vertex of either $\Delta^{\mathcal{E}_{i}}$ or $\Delta^{\mathcal{E}'_{i}}$, 
and $v_{1} < v_{2}$, then we can write
\[ v_{2} = \bigcup_{k=1}^{m} p_{k}, \]
where, for $k=1, \ldots, m$, $p_{k}$ is a pseudovertex having the same support
as $\{ [f_{k}, D_{k}] \}$. The ascending link of $v_{1}$ relative to $v_{2}$
is homeomorphic to the joins of the ascending links of $\{ [f_{k},D_{k}] \}$
relative to $p_{k}$, for $k = 1, \ldots, m$. (This can be argued exactly as in the proof of Theorem 6.9 from \cite{FH2}.)  All of the latter ascending links are contractible if they are non-empty, by the proof of Theorem \ref{theorem:contractibilityofDelta} given above. At least one of the latter ascending links is non-empty (since $v_{1} \neq v_{2}$), so the ascending link of $v_{1}$ relative to $v_{2}$ is contractible, as claimed.  
\end{remark}

%%%%%%%%%%%%%%%%%%%%%%%%%%
\subsection{$\Gamma$-finite filtrations of $\Delta^{\mathcal{E}_{i}}$
and $\Delta^{\mathcal{E}'_{i}}$}
%%%%%%%%%%%%%%%%%%%%%%%%%

In this subsection, we will describe natural filtrations of the complexes 
$\Delta^{\mathcal{E}_{i}}$
and $\Delta^{\mathcal{E}'_{i}}$. We will denote the acting group by $\Gamma$; here
$\Gamma$ might be any of the groups $\{ F(S), T(S), V(S) \}$, where $S \in \{ S_{2}, S_{3}, S'_{2}, S'_{3} \}$.

\begin{definition} \label{definition:filtration}
($\Gamma$-finite filtrations of the complexes $\Delta^{\mathcal{E}_{i}}$ and $\Delta^{\mathcal{E}'_{i}}$)
Let $v$ be a vertex in $\Delta^{\mathcal{E}_{i}}$ or $\Delta^{\mathcal{E}'_{i}}$, or a pseudovertex. We let $|v|$ denote the cardinality of $v$, which we will call the \emph{height} of $v$. For $n \geq 1$, we let $\Delta^{\mathcal{E}_{i}}_{n}$ denote the subcomplex of $\Delta^{\mathcal{E}_{i}}$ spanned by vertices of height $n$ or less. Similarly define the subcomplexes $\Delta^{\mathcal{E}'_{i}}_{n}$ of $\Delta^{\mathcal{E}'_{i}}$.
\end{definition}

\begin{proposition} \label{proposition:cocompactness}
The group $\Gamma$ acts on each $\Delta^{\mathcal{E}_{i}}_{n}$ (or $\Delta^{\mathcal{E}'_{i}}_{n}$, as the case may be) cocompactly, and
\[ \Delta^{\mathcal{E}_{i}} = \bigcup_{n=1}^{\infty} \Delta^{\mathcal{E}_{i}}_{n}. \]
A similar equality is true of $\Delta^{\mathcal{E}'_{i}}$ and the subcomplexes $\Delta^{\mathcal{E}'_{i}}_{n}$.
\end{proposition}

\begin{proof}
Let us first note that the $\Gamma$-action preserves height, and therefore acts on $\Delta^{\mathcal{E}_{i}}_{n}$ (or $\Delta^{\mathcal{E}'_{i}}_{n}$). It is easy to see that $\Delta^{\mathcal{E}_{i}}$ is the union of the subcomplexes in the filtration. 

We temporarily let $\mathcal{E}$ denote an arbitrary expansion scheme.
Definition 6.12 from \cite{FH2} describes an action $\star$ of $\mathbb{S}(D,D)$ ($D \in \mathcal{D}^{+}_{S}$, or $D \in \mathcal{D}^{+}_{gen}$, as in our case) on the 
set $\mathcal{E}([f,D])$ as follows:
\[ h \star v = (fhf^{-1}) \cdot v, \]
where $h \in \mathbb{S}(D,D)$. If the action of $\mathbb{S}(D,D)$ on the simplicial realization of $\mathcal{E}([f,D])$ is always cocompact, for all $D$, then $\mathcal{E}$
is said to be \emph{$\mathbb{S}$-finite}. 

In the current situation, the group $\mathbb{S}(D,D)$ is isomorphic either to $\mathbb{Z}$ or to the trivial group (when $[D] = [I]$ or $[D] = [ [0,\infty) ]$, respectively). In either case, the action of $\mathbb{S}(D,D)$ is cocompact. Indeed,
the action of $\mathbb{Z}$ on $\mathcal{E}([id_{I},I])$ is by translation (i.e., the integer $n$ moves a vertex $u_{k}$ to $u_{n+k}$, where the vertices $u_{j}$ are as described in Example \ref{example:THEexpansionscheme}). This is clearly cocompact; see Figure \ref{figure:E}. This reasoning applies equally to all $[f,D]$ such that $[D] = [I]$ due to the equivariance of the expansion scheme $\mathcal{E}_{i}$ (or $\mathcal{E}'_{i}$). If $[D] = [[0,\infty)]$, there is nothing to prove, since the set $\mathcal{E}'_{i}([f,D])$ is compact. It follows that both $\mathcal{E}_{i}$ and $\mathcal{E}'_{i}$ are $\mathbb{S}$-finite.

We can now apply Proposition 6.13 from \cite{FH2}, which says that 
when an expansion scheme $\mathcal{E}$ is $\mathbb{S}$-finite and $\mathbb{S}$ has finitely many domain types, then the action of $\Gamma$ on each subcomplex $\Delta^{\mathcal{E}}_{n}$ is cocompact; this proves that the action of $\Gamma$ on the filtration is cocompact.

The final equality in the proposition is clear.
\end{proof}

\begin{remark}
We sketch a more direct proof that $\Gamma$ acts cocompactly.

We assume that $\Gamma = V(S_{i})$, the proofs for the other groups being similar. Two vertices $v_{1}$ and $v_{2}$ are in the same $\Gamma$-orbit if and only if they have the same type (Definition \ref{definition:vertices}). In the current context, the latter condition is equivalent to having the same height. 

Now assume that $\Gamma$ fails to act cocompactly on $\Delta^{\mathcal{E}_{i}}_{n}$, for some $n$. We note that the dimension 
of $\Delta^{\mathcal{E}_{i}}_{n}$ is no more than $n-1$, since a simplex in $\Delta^{\mathcal{E}_{i}}_{n}$ is an ascending chain
\[ v_{0} < v_{1} < v_{2} < \ldots < v_{k}, \]
and the height function strictly increases along such chains. Thus, assuming that the action of $\Gamma$ is not cocompact, there are infinitely many $\Gamma$-orbits of $k$-simplices, for some $k$. Since there are only finitely many $\Gamma$-orbits of vertices, this implies that there
is a vertex $v' = \{ b_{1}, \ldots, b_{\ell} \}$ such that infinitely many $\Gamma$-orbits of $k$-simplices have $v'$ as their minimal vertex. This, however, sets up the contradiction, since all of the $k$-simplices in question are obtained by $\mathcal{E}_{i}$-expansion from $v'$, and there are only finitely many such $\mathcal{E}_{i}$-expansions modulo the action $\star$. 

The details of the remainder of the argument follow that of the proof of Proposition 6.13 from \cite{FH2}.
\end{remark}

%%%%%%%%%%%%%%%%%%%%%%%%%%
\subsection{The $F_{\infty}$ property for $V(S_{n})$ and $V(S'_{n})$}
%%%%%%%%%%%%%%%%%%%%%%%%%%

\begin{definition} \label{definition:contractingpseudo}
(Contracting pseudovertices)
Let $\mathcal{E}$ be an arbitrary expansion scheme.
We say that a pseudovertex $v$ is \emph{contracting relative to $\mathcal{E}$} if $v$ has the same type as
some $w \in \mathcal{E}(b)$, where $b \in \mathcal{B}$. (Recall that ``same type" was defined in Definition \ref{definition:vertices}.)
\end{definition}

\begin{definition} \label{definition:rich}
(Rich in contractions)
Let $\mathcal{E}$ be an expansion scheme. We say that $\mathcal{E}$ is 
\emph{rich in contractions} if there is some constant $C$ such that, if $v$ is a pseudovertex of height at least $C$, then there is some contracting pseudovertex
$v'$ such that $v' \subseteq v$.
\end{definition}

\begin{theorem} \label{theorem:FHFinfty} (\cite{FH2}, Theorem 8.2)
(Groups of type $F_{\infty}$) Let $\mathbb{S}$ be an $S$-structure with finitely many domain types, such that the group $\mathbb{S}(D,D)$ has type $F_{\infty}$
for $D \in \mathcal{D}^{+}$. Let $\mathcal{E}$ be an expansion scheme such that 
\begin{enumerate}
\item $\mathcal{E}$ is $n$-connected for all $n$;
\item $\mathcal{E}$ is rich in contractions;
\item each set $\mathcal{E}(b)$ ($b \in \mathcal{B}$) is finite.
\end{enumerate}
The group $\Gamma_{S}$ has type $F_{\infty}$.
\end{theorem}

\begin{theorem} \label{theorem:FinftyforV}
The groups $V(S_{i})$ and $V(S'_{i})$ are of type $F_{\infty}$, for $i=2,3$.
\end{theorem}

\begin{proof}
Our strategy is to apply the proof of Theorem \ref{theorem:FHFinfty} (Theorem 8.2 from \cite{FH2}) to the groups $\Gamma$. (We note that the groups $\Gamma_{S}$ under consideration in Theorem \ref{theorem:FHFinfty} are analogous to Thompson's group $V$, in  that there is no assumption that $\Gamma_{S}$ preserves no linear or cyclic order.) Let us note that condition (3) is violated, since 
the sets $\mathcal{E}_{i}(b)$ and $\mathcal{E}'_{i}(b)$ are not finite when $b = [f,D]$ and $[D] = [I]$, so the statement does not apply directly.

We have already seen that $\mathcal{E}_{i}$ and $\mathcal{E}'_{i}$ are $n$-connected expansion schemes for all $n$. 

We claim that the expansion scheme $\mathcal{E}_{i}$ is rich in contractions with constant $C = 2$ when $i=2$ or $3$. 
Let $\{ [f_{1},D_{1}], [f_{2}, D_{2}] \} \subseteq \mathcal{B}$ be a pseudovertex. Since $D_{1}, D_{2} \in \mathcal{D}^{+}_{gen}$, we have
$D_{1} = \omega_{1}I$ and $D_{2} = \omega_{2}I$, for some words $\omega_{1}, \omega_{2} \in \{ A, B \}^{\ast}$. Thus, 
\[ [f_{n}, D_{n}] = [f_{n}, \omega_{n}I] = [f_{n} \omega_{n}, I], \]
for $n=1,2$.  Define $g$ on $[0,1)$ by the following rule:
\[  g(x) = \left\{ \begin{array}{ll} f_{1}\omega_{1}a(x) & \text{if } 
x \in [0,1/2) \\
f_{2}\omega_{2}b(x) & \text{if x} \in [1/2,1) \end{array} \right. \]
The pseudovertex $\{ [g,I] \}$ expands to
\begin{align*}
\{ [g,AI], [g,BI] \} &= \{ [f_{1}\omega_{1}a,AI], [f_{2}\omega_{2}b,BI] \} \\ &= \{ [f_{1}\omega_{1},I], [f_{2}\omega_{2},I] \} \\  &=
\{ [f_{1},D_{1}], [f_{2},D_{2}] \}. 
\end{align*}
This proves the claim.
 
The expansion scheme $\mathcal{E}'_{i}$ is also rich in contractions with constant $C=2$. If $\{ [f_{1},D_{1}], [f_{2},D_{2}] \} \subseteq \mathcal{B}$ is a pseudovertex and $D_{1}, D_{2}$ have the same domain type as $I$, then the proof of the previous paragraph shows that a contraction can be performed on 
$\{ [f_{1},D_{1}], [f_{2},D_{2}] \}$. The only remaining case to consider is when 
$[D_{1}] = [I]$ and $[D_{2}] = [[0,\infty)]$. We will write $R$ in place of $[0,\infty)$, to simplify notation. In this case, 
$D_{1} = \omega_{1}I$ and $D_{2} = T^{m} R$, where
$\omega_{1} \in \{ A, B, T \}^{\ast}$ and $m \geq 0$. We have the equalities:
\[ [f_{1}, \omega_{1}I] = [f_{1}\omega_{1},I] \quad \text{and}
\quad [f_{2}, T^{m}R] = [f_{2}T^{m},R]. \] 
Define $g: [0,\infty) \rightarrow [0,\infty)$ as follows:  
\[  g(x) = \left\{ \begin{array}{ll} f_{1}\omega_{1}(x) & \text{if } 
x \in [0,1) \\
f_{2}T^{m-1}(x) & \text{if x} \in [1,\infty) \end{array} \right. \] 
The pseudovertex $\{ [g,R] \}$ expands to 
\begin{align*}
\{ [g,R] \} &= \{ [g,I], [g,TR] \} \\
&= \{ [f_{1}\omega_{1},I], [f_{2}T^{m-1},TR] \} \\
&= \{ [f_{1}, D_{1}], [f_{2}, T^{m}R] \} \\
&= \{ [f_{1}, D_{1}], [f_{2}, D_{2}] \}.
\end{align*}
It follows that $\{ [f_{1}, D_{1}], [f_{2},D_{2}] \}$ is also a contracting vertex relative to $\mathcal{E}'_{i}$.

The assumption that $\mathcal{E}(b)$ is always finite is used in the proof of Theorem \ref{theorem:FHFinfty} in three ways: (1) to prove that $\Gamma$ acts cocompactly on the complexes $\Delta^{\mathcal{E}}_{n}$; (2) to prove that the cell stabilizers have type $F_{\infty}$, and (3) to define a certain constant $C_{0}$. We have already established (1) and (2) by other means: indeed, cell stabilizers are virtually finitely generated free abelian groups, and therefore have type $F_{\infty}$, and the cocompactness of the actions on the complexes $\Delta^{\mathcal{E}_{i}}_{n}$ and $\Delta^{\mathcal{E}'_{i}}_{n}$ was proved as part of Proposition \ref{proposition:cocompactness}. The constant $C_{0}$ is the largest height (i.e., cardinality) of a contracting pseudovertex. Clearly we have an independent bound of $C_{0} = 3$ when $S \in \{ S_{2}, S'_{2} \}$, or $C_{0} = 5$ when $S \in \{ S_{3}, S'_{3} \}$. (Refer to the definitions of $\mathcal{E}_{i}$ and $\mathcal{E}'_{i}$ in Example \ref{example:THEexpansionscheme}.)  
\end{proof}

\begin{remark} \label{remark:Finfty}
We will offer a sketch of the argument here.

We check the hypotheses of Brown's finiteness criterion (Theorem \ref{theorem:Brown}). First, we note that $\Delta^{\mathcal{E}_{i}}$ and $\Delta^{\mathcal{E}'_{i}}$ are contractible by Theorem \ref{theorem:contractibilityofDelta}. It is clear that the relevant actions are cellular. Properties (3)(a) and (b) are settled in Proposition \ref{proposition:cocompactness}. Cell stabilizers are virtually finitely generated free abelian (and therefore of type $F_{\infty}$) by Proposition \ref{proposition:freeabelian}. 

This leaves only (3)(d) to check; i.e., we must show, for each $n \in \mathbb{N}$, that $\Delta^{\mathcal{E}_{i}}_{k}$ and 
$\Delta^{\mathcal{E}'_{i}}_{k}$ are $(n-1)$-connected for sufficiently large $k$. The proof of the latter follows a now-standard strategy: we show that the descending link of a vertex becomes highly connected as the height of the vertex increases. A few basics of this strategy are summarized in Subsection 7.2 of \cite{FH2}, although the methods of argument go back to \cite{Brown} and \cite{BB}. 

We consider $\Delta^{\mathcal{E}_{2}}$; the other cases are similar.
Let $v$ be a vertex of height $k$ in $\Delta^{\mathcal{E}_{2}}$.
The \emph{descending link} of $v$  is its link in 
$\Delta^{\mathcal{E}_{2}}_{k}$. Our analysis of the descending link uses the Nerve Theorem (as it appears in \cite{AB}; the Nerve Theorem is also Theorem 2.10 in \cite{FH2}). Let 
\[ v= \{ b_{1}, \ldots, b_{k} \}. \]
We cover the descending link of $v$ by a number of subcomplexes, called 
\emph{partitioned downward links}, which are each determined by a partition of $v$, and which we now define.

 Let $\mathcal{P}$ be a partition of $v$. The \emph{partitioned downward star} $\mathrm{st}_{\downarrow}(v_{\mathcal{P}})$ (Definition 7.7 from \cite{FH2}), is the subcomplex of 
$\Delta^{\mathcal{E}_{2}}_{k}$  consisting of the vertex $v$ and all simplices resulting from $\mathcal{E}_{2}$-contractions that are supported within members of $\mathcal{P}$. For instance, if 
\[ \mathcal{P} = \{ \{ b_{1}, b_{2} \}, \{ b_{3}, \ldots, b_{k} \} \}, \]
then a contraction supported on the subset $\{ b_{1}, b_{2} \}$, or on the subset $\{ b_{3}, b_{4}, b_{7} \}$ (if $k \geq 7$) (or indeed a combination of such contractions), results in a simplex of 
$\mathrm{st}_{\downarrow}(v_{\mathcal{P}})$, but a contraction supported
on $\{ b_{2}, b_{3} \}$ would not. We then define the partitioned downward link 
$\mathrm{lk}_{\downarrow}(v_{\mathcal{P}})$ as the link of $v$ in  
$\mathrm{st}_{\downarrow}(v_{\mathcal{P}})$. 

For each contracting pseudovertex $w \subseteq v$, we let 
\[ \mathcal{P}_{w} = \{ v-w, w \}. \]
(We note that, in the current context,  ``contracting pseudovertex'' is the same as  ``pseudovertex with two or three members'', by the description of $\mathcal{E}_{2}$ from Example \ref{example:THEexpansionscheme}.) The collection 
\[ \mathcal{C} = \{ \mathrm{lk}_{\downarrow}(v_{\mathcal{P}_{w}}) \mid w
\text{ is a contracting pseudovertex} \} \]
is a cover of $\mathrm{lk}_{\downarrow}(v)$. We apply the Nerve Theorem to $\mathcal{C}$. The intersection of two members of $\mathcal{C}$ is another partitioned downward link:
\[   \mathrm{lk}_{\downarrow}(v_{\mathcal{P}_{w'}}) \cap \mathrm{lk}_{\downarrow}(v_{\mathcal{P}_{w''}}) =  
\mathrm{lk}_{\downarrow}(v_{\mathcal{P}_{w'} \wedge \mathcal{P}_{w''}}), \]
where $\mathcal{P}_{w'} \wedge \mathcal{P}_{w''}$ is the coarsest common refinement of $\mathcal{P}_{w'}$ and $\mathcal{P}_{w''}$. The generalization to finite intersections is straightforward. 

For a partition $\mathcal{P} = \{ P_{1}, \ldots, P_{\ell} \}$ of $v$, there is a natural 
join structure (see Corollary 7.9 from \cite{FH2}):
\[ \mathrm{lk}_{\downarrow}(v_{\mathcal{P}}) \cong 
\bigast_{j=1}^{\ell} \mathrm{lk}_{\downarrow}(P_{j}), \]
where the latter descending links depend only on the types of the pseudovertices $P_{j}$. (In the current case, the type is entirely determined by the cardinality.) Recall that, if $X_{1}$ and $X_{2}$ are $n_{1}$-connected and $n_{2}$-connected complexes (respectively), then the join $X_{1} \ast X_{2}$ is $(n_{1} + n_{2} + 2)$-connected. It follows that $\mathrm{lk}_{\downarrow}(v_{\mathcal{P}})$ is at least as connected as the most highly connected factor
$\mathrm{lk}_{\downarrow}(P_{j})$. 

Finally, we note that a pseudovertex of height two or more has a non-empty descending link (since every such pseudovertex contains a contracting pseudovertex). 
This gives us the base case of an induction; the above considerations allow us to prove inductively that pseudovertices of increasing height have increasing connectivity. The actual induction is done in the proof of Theorem 8.2 from \cite{FH2}; we omit further details. We will consider a similar induction in more detail  in the next subsection.   
\end{remark}

%%%%%%%%%%%%%%%%%%%%%%%%%
\subsection{The $F_{\infty}$ property for the remaining groups}
%%%%%%%%%%%%%%%%%%%%%%%%%%

\begin{theorem} \label{theorem:FINF}
The groups $F(S)$, where $S \in \{ S_{2}, S_{3}, S'_{2}, S'_{3} \}$, and $T(S)$, where $S \in \{ S_{2}, S_{3} \}$, have type $F_{\infty}$.
\end{theorem}

\begin{proof}
We consider the group $F(S_{2})$ (the Lodha-Moore group). The proofs that the other groups have type $F_{\infty}$ differ in minor details.

We turn to an analysis of the descending link; all of the other ingredients of the proof can be assembled exactly as in Remark \ref{remark:Finfty}. Let 
\[ v = \{ b_{1}, b_{2}, \ldots, b_{k} \} \]
be either a vertex of $\Delta^{\mathcal{E}_{2}}$, or a pseudovertex.
We assume that the $b_{i}$ are linearly ordered, in the following sense: Each $b_{i} = [f_{i}, D_{i}]$, for appropriate $f_{i}$ and $D_{i} \in \mathcal{D}^{+}_{gen}$, where $f_{i}: D_{i} \rightarrow [0,1)$ is a locally $S_{2}$-embedding that is, moreover, continuous and increasing. We assume that $f_{1}(D_{1})$, $f_{2}(D_{2})$, 
$\ldots$, $f_{k}(D_{k})$ are arranged from left to right. With this assumption, each $\mathcal{E}_{2}$-contraction must be performed on two or three consecutive $b_{i}$. For a subset $K \subseteq \{ 1, \ldots, k \}$, we define
\[ \mathcal{P}_{K} = \{ \{ b_{j} \mid j \in K \}, \{ b_{j} \mid j \not \in K \} \}. \]
We then define
\[ \mathcal{C} = \{ \mathrm{lk}_{\downarrow}(v_{\mathcal{P}_{K}}) \mid
K \in \{ \{ 1, 2 \}, \{ 2, 3 \}, \{ 1, 2, 3 \}, \{ 2, 3, 4 \}, \{ 3, 4, 5 \} \} \}. \]
(If $k < 5$, then the possible subsets $K$ are restricted accordingly.)

We claim that $\mathcal{C}$ is a cover of $\mathrm{lk}_{\downarrow}(v)$. Indeed, let $\sigma$ be a simplex in $\mathrm{lk}_{\downarrow}(v)$. Thus, there is an increasing sequence
\[ v_{0} < v_{1} < v_{2} < \ldots < v_{\ell -1} < v_{\ell} = v, \]
where each $v_{\alpha}$ is obtained by $\mathcal{E}_{2}$-expansion from $v_{0}$, and
\[ \sigma = v_{0} < v_{1} < \ldots < v_{\ell -1}. \]
If $v_{0} = \{ b'_{0}, b'_{1}, \ldots, b'_{q} \}$, where the members are linearly ordered, then there is a leftmost $b'_{\beta}$ that is expanded when we pass from
$v_{0}$ to $v$. In expanding at $b'_{\beta}$, we replace $b'_{\beta}$ with either two or three pairs from $\mathcal{B}$. The latter will occur consecutively in $v$. Thus, the result of expanding at $b'_{\beta}$ will contribute either $\{ b_{\alpha}, b_{\alpha +1}, b_{\alpha +2} \}$ or $\{ b_{\alpha}, b_{\alpha +1} \}$ to $v$, for some $\alpha$. 
All other expansions from $v_{0}$ to $v$ will contribute a disjoint subset of $b_{i}$'s to $v$. It follows easily from this that $\sigma$ is contained in at least one member of the cover $\mathcal{C}$.

For instance, if the expansion at $b'_{\beta}$ contributes $b_{1}$ and $b_{2}$ to $v$, any other expansion from $v_{0}$ must contribute some subset of $\{ b_{3}, \ldots, b_{k} \}$. Thus, in this case, $\sigma \subseteq \mathrm{lk}_{\downarrow}(v_{\mathcal{P}_{\{1, 2 \}}})$. If $b'_{\beta}$ contributes $b_{2}$ and $b_{3}$, then $\sigma \subseteq \mathrm{lk}_{\downarrow}(v_{\mathcal{P}_{\{2, 3 \}}})$.
If $b'_{\beta}$ contributes $b_{m}$ and $b_{m+1}$, for some $m \geq 3$, then
$\sigma \subseteq \mathrm{lk}_{\downarrow}(v_{\mathcal{P}_{\{1, 2 \}}})$ (since, indeed, all expansions contribute some subset of $\{ b_{3}, \ldots, b_{k} \}$ under this hypothesis). 

Now we establish the connectivity of the descending link, as a function of the height $k$. We note first that $\mathrm{lk}_{\downarrow}(v)$ is non-empty provided that $k \geq 2$. It follows from this that each $\mathrm{lk}_{\downarrow}(v_{P_{K}})$ is connected when  $k \geq 7$, since each is a join of two non-empty complexes. Now, if $k \geq 7$, then
$\mathrm{lk}_{\downarrow}(v)$ is connected, since it is covered by a collection 
$\mathcal{C}$ of non-empty subcomplexes, which have a non-empty intersection. (A contraction at $\{ b_{6}, b_{7} \}$ lies in all of the partitioned descending links simultaneously.) 

In general, $\mathrm{lk}_{\downarrow}(v)$ is $n$-connected provided that
$k \geq 5n+7$. We have proved this already for $n=-1$ and $n=0$. Assume that the result is true for $n$. We consider a vertex $v$ of height $k$ at least $5n+12$. 
Each $\mathrm{lk}_{\downarrow}(v_{P_{K}})$ is $(n+1)$-connected, since each is a join of two complexes: one non-empty and one isomorphic to the descending link of a vertex of height at least $5n+7$, and therefore $n$-connected by induction. Moreover, any subcollection of $\mathcal{C}$ containing two of more members intersects in a subcomplex that is at least $n$-connected. (Any such intersection is a join, and one of the factors of the join is the descending link on $\{ b_{\gamma}, \ldots, b_{k} \}$, where $\gamma \leq 6$.)

By the Nerve Theorem \cite{AB}, $\mathrm{lk}_{\downarrow}(v)$ is $(n+1)$-connected if $t$-fold intersections of the cover are $(n-t+2)$-connected and the nerve of the cover is $(n+1)$-connected. Since the nerve is easily seen to be a four-dimensional simplex, and $t$-fold intersections have the required connectivity (by the previous paragraph), $\mathrm{lk}_{\downarrow}(v)$ is $(n+1)$-connected, completing the induction.

By well-established principles (as summarized in Proposition 7.6 from \cite{FH2}, for instance), the 
connectivity of the subcomplex $\Delta^{\mathcal{E}_{i}}_{k}$ tends to infinity as $k$ increases, completing the proof.
\end{proof}

\bibliographystyle{plain}
\bibliography{biblio}

 \end{document}